\documentclass[10pt]{article}
\usepackage{amsmath,amsfonts,amssymb,color,hyperref}
\usepackage{a4wide,amsthm}
\usepackage[english]{babel}
\usepackage{graphicx}
\usepackage{hyperref,pdfsync}

\newcommand{\vv}[1]{\mathbf{#1}}
\newcommand{\vc}[1]{\mathit{#1}}
\newcommand{\sslash}{\text{\tiny{\!\slash\!\slash\,}}}
\newcommand{\stokes}{\eqref{eq_stokes1}--\eqref{eq_stokes4}(+\eqref{eq_stokes5})}

\definecolor{darkred}{rgb}{0.2,0.1,0.4}
\newtheorem{theorem}{Theorem}

\newtheorem{proposition}[theorem]{Proposition}
\newtheorem{lemma}[theorem]{Lemma}
\newtheorem{definition}[theorem]{Definition}

\begin{document}

\title{Justification of lubrication approximation: \\ an application to fluid/solid interactions}
\author{ M. Hillairet\footnote{I3M, Universit\'e de Montpellier, Place Eugene Bataillon, 34095 Montpellier Cedex} \& T. Kela\"i\footnote{IMJ-PRG et Universit\'e Paris 7, 56-58, Avenue de France, 75205 Paris France}}
\date{\today}

\maketitle

\begin{abstract}
We consider the stationary Stokes problem in a three-dimensional fluid domain $\mathcal F$ with non-homogeneous Dirichlet boundary conditions.
We assume that this fluid domain is the complement of a bounded obstacle $\mathcal B$  in a bounded or an exterior smooth container $\Omega$. 
We compute sharp asymptotics of the solution to the Stokes problem when the distance between the obstacle and the container boundary is small. 
\end{abstract}

\bigskip



\bigskip

In this paper, we consider the 3D-Stokes problem 
\begin{eqnarray}
\label{eq_stokes1}\Delta \vv{u} - \nabla p &=& 0\,,\, \\
\label{eq_stokes2} \nabla \cdot \vv{u} &=&0 \,,
\end{eqnarray}
in a fluid domain $\mathcal F.$ Without restricting the generality, we set the viscosity of the fluid to $1.$ We assume that $\mathcal F= \Omega \setminus \overline{\mathcal B}$ is the complement of a smooth simply-connected bounded domain $\mathcal B$ inside a container $\Omega.$ The container $\Omega$ is either a relatively compact simply-connected smooth open set or the exterior of a simply-connected smooth compact set $\mathcal B^*.$ In both cases, it has a smooth compact connected boundary. We complete then \eqref{eq_stokes1}-\eqref{eq_stokes2} with boundary conditions:
\begin{eqnarray}
\label{eq_stokes3} \vv{u} &=& \vv{u}^* \,, \phantom{0} \,\,\, \text{ on $\partial \Omega$}\,,\\
\label{eq_stokes4} \vv{u} &=& \mathbf 0 \,, \phantom{\vv{u}^*} \,\,\,  \text{ on $\partial \mathcal B$}\,, 
\end{eqnarray}
where $\vv{u}^* \in C^{\infty}(\partial \Omega)$ does not prescribe any flux through the container boundary:
\begin{equation} \label{eq_nofluxintro}
\int_{\partial \Omega} \vv{u}^* \cdot \vv{n}{\rm d}\sigma =0\,.
\end{equation}
If $\Omega$ is an exterior domain, we add a vanishing condition at infinity:
\begin{eqnarray} \label{eq_stokes5}
\lim_{|(x,y,z)| \to \infty}\vv{u}(x,y,z) = 0\,.
\end{eqnarray}
With the above assumptions on the boundary data $\vv{u}^*$ it is classical that system \eqref{eq_stokes1}--\eqref{eq_stokes4}(+\eqref{eq_stokes5}) admits a unique classical solution 
$(\vv{u},p)$ (the pressure $p$ being unique up to a constant), see \cite{Galdibooknew} for instance. {\em Our aim in this paper is to give a sharp description of the solution to the Stokes problem {\em \eqref{eq_stokes1}--\eqref{eq_stokes4}(+\eqref{eq_stokes5})} when the distance between $\mathcal B$ and $\partial \Omega$ is small.}

\medskip

Computing such asymptotics is an important issue related to the modeling of solid-body motion inside a viscous fluid. The typical configuration we have in mind is that $\mathcal B$ (resp. $\mathcal B$ and $\mathcal B^*$) is a (are) moving solid body (bodies) inside the container $\Omega$ (in the whole space) which is filled by a viscous incompressible constant-density fluid. A typical issue is then to determine whether the fluid viscosity prevents the moving body $\mathcal B$ from touching other solid boundaries ({\em i.e.} the boundary of the container or the boundary of the other solid body) and more generally to measure the influence of the viscosity on the close-contact dynamics of $\mathcal B$. To this end, one remarks that, when a solid body is about to collide another solid boundary with moderate relative velocity, the fluid Reynolds number tends to $0$ in the gap so that a stationary Stokes system is sufficient to predict the force and torque exherted by the fluid on the moving body. With this particular application in mind, several authors consider the free-fall of a sphere above a ramp \cite{Cooley&ONeill68,Cooley&ONeill69,Dean&Oneill64,ONeill64,ONeill&Stewartson67} in a Stokes fluid. Explicit values for the solution to the Stokes problem \eqref{eq_stokes1}--\eqref{eq_stokes4}(+\eqref{eq_stokes5}) and the associated force and torque are provided. 
A formal lubrication approximation is also proposed in \cite{Cox74,Happel&Brenner65} which generalizes these formulas to arbitrary
configurations. In the limit regime where there is contact, solutions to the Stokes problem are also computed in \cite{Nazarov95} under further assumptions on the boundary data $\mathbf u^*$ (broadly, the boundary data $\mathbf u^*$ has to vanish sufficiently where there is contact). Related computations for a perfect fluid are provided in \cite{CardoneNazarovSokolowski09,HouotMunnier08,MunnierRamdanipp}.

\medskip

In this paper, we fill the gap between the explicit formulas of \cite{Cooley&ONeill68,Cooley&ONeill69,Dean&Oneill64,ONeill64,ONeill&Stewartson67}, the formal asymptotics of \cite{Cox74,Happel&Brenner65}
and the analysis in \cite{Nazarov95}. We justify rigorously lubrication approximation in the spirit of \cite{Chambayada86,Cimatti83}. Compared to these latter references, we are interested inhere in fluid films that do not vanish uniformly in their widths. This fact leads to severe new difficulties. First, the lubrication scaling acts on coordinates in both tangential and orthogonal directions to the boundaries (see \eqref{eq_u1}--\eqref{eq_p4}). Second, the asymptotic pressure and velocity-fields yielding from the formal lubrication approximation are defined on an asymptotic fluid domain which is not simply related to the fluid-domain for a given positive body/boundary distance $h.$ Consequently, in order to compare the values of the solution to the Stokes problem with its asymptotic value, we introduce 
an intermediate velocity-field which embeds the lubrication approximation in the effective fluid-domain
(see Section \ref{sec_asvf}). The construction of this intermediate velocity-field is an important step of our analysis. 
Indeed, the intermediate velocity-field is a key-ingredient in order to extend the computations on the close-contact dynamics of bodies in a Stokes fluid to more complicated models:
Navier-Stokes/Newton models or Navier-Stokes/elasticity models. 
This fact has already been shown for the Navier-Stokes/Newton model in simple configurations
(see \cite{DGVH10,Hillairet07,HillairetTakahashi09,HillairetTakahashi10}).
We believe our approach extends to other incompresible models (for instance to the case of potential flows as in \cite{MunnierRamdanipp}). However, the incompressibility condition is crucial to our computation and the extension to the compressible case is completly open.

\medskip

To fix ideas, we make more specific the geometry of the gap between $\mathcal B$ and $\partial \Omega.$ 
We introduce a set of cartesian coordinates $(x,y,z) \in \mathbb R^3$ and $(r,\theta,z) \in (0,\infty) \times (-\pi,\pi) \times \mathbb R$ the corresponding cylindrical coordinates. These coordinates are associated with two orthonormal basis of $\mathbb R^3$ denoted by
$(\vv{e}_x,\vv{e}_y,\vv{e}_z)$ and $(\vv{e}_r,\vv{e}_{\theta},\vv{e}_z)$ respectively.
We consider that, in a neighborhood $\mathcal U$ of the origin, $\Omega$ and $\mathcal B$ satisfy: 
\begin{eqnarray} \label{eq_bdyb}
(x,y,z) \in (\mathbb R^3 \setminus \overline{\Omega}) \cap \mathcal U &\Leftrightarrow& \left\{(x,y) \in B(\mathbf 0,L) \text{ and }  z < \gamma_b(x,y) \right\}\,,  \\ \label{eq_bdyt}
(x,y,z) \in \mathcal B \cap \mathcal U &\Leftrightarrow& \left\{(x,y) \in B(\mathbf 0,L) \text{ and }  z > h + \gamma_t(x,y) \right\}\,.
\end{eqnarray}    
Here $(h,L)$ are given positive parameters. The first one measures the distance between $\partial \Omega$ and $\partial \mathcal B,$ the second one is a characteristic length on which parametrizing the boundaries of $\Omega$ and $\mathcal B$ by $(x,y)$-variables is relevant. The two functions $\gamma_t$ and $\gamma_b$ are smooth on 
$B(\mathbf 0,L) \subset \mathbb R^2$ and we assume throughout the paper that they satisfy the following assumptions: 
\begin{itemize}
\item[$({\rm A}1)$] $\nabla \gamma_t(\mathbf 0) = \nabla \gamma_b(\mathbf 0) = \mathbf 0,$ \\[-8pt]
\item[$({\rm A}2)$] there exists $R_1>0$ such that (in the sense of symmetric matrices)
$$
 \nabla^2 \gamma_t(x,y) - \nabla^2 \gamma_b (x,y) \geq \dfrac{1}{R_1} \mathbb{I}_3 \,, \quad \forall \, (x,y) \in B(\mathbf 0,L)\,.
$$ 
\end{itemize}
We complement these local assumptions in $\mathcal U$ with a global one concerning $\mathcal B$ and $\partial \Omega$ outside $\mathcal U:$  
\begin{equation} \tag{A3}
\exists\,  \delta >0  \text{ s.t. } \; {\rm dist(\partial \Omega \setminus \mathcal U, \mathcal B)} > \delta.  
\end{equation}
The parametrizations \eqref{eq_bdyb} and \eqref{eq_bdyt} together with assumptions (A1)-(A2)-(A3) are paradigmatic of what we call a "single non-degenerate contact". Indeed, considering that $\mathcal B$ is a translating particle inside the container $\Omega$ amounts to let the parameter $h$ depend on time. Assumption (A3) implies then that a contact between $\mathcal B$ and $\partial \Omega$ may only hold in $\mathcal U.$ 
Assumption (A2) yields that the contact in $\mathcal U$ is unique and non-degenerate {\em i.e.} when $h=0$ the "vertical" distance $\gamma_t- \gamma_b$ vanishes in $0$ only and with minimal vanishing order.
We emphasize that, in the smooth case we consider here, the uniform boundedness we require in (A2) reduces to assuming that
\begin{equation} \label{eq_A'2} \tag{A'2}
 \nabla^2 \gamma_t(\mathbf 0) - \nabla^2 \gamma_b (\mathbf 0) \geq \dfrac{1}{R_1} \mathbb{I}_3 \,,
 \end{equation}
 up to restrict the size of $L$ {and change the value of $R_1.$}
In this "single non-degenerate contact"-case, the assumptions \eqref{eq_bdyb}-\eqref{eq_bdyt} together with (A1) do not restrict the generality: they only amount to choose the origin of coordinates in the only point of $\partial \Omega$ realizing the distance between $\partial \Omega$ and $\partial \mathcal B,$ and to choose the system of coordinates so that the common normal to $\partial \Omega$ and $\partial \mathcal B,$ in the pair of points realizing the distance between $\partial \Omega$ and $\partial \mathcal B,$ is $\vv{e}_z$.

\medskip

Our aim is to compute the asymptotics of the solution to the Stokes problem \stokes\, for a given boundary condition $\vv{u}^* \in C^{\infty}(\partial \Omega)$ in the geometry depicted above under the further assumption that $h$ is small and other parameters are of order $1.$ 
To introduce our main result, we recall the main steps of the formal computations in \cite{Cox74} for the case where $\mathcal B$ is a sphere of radius $S$ and $\Omega = \mathbb R^3 \setminus \mathcal B^*$ with $\mathcal B^*$ a sphere of radius $R$. First, given the shape of the aperture between both spheres, one looks for a solution $(\vv{u},p) := ((\vc{u}_x,\vc{u}_y,\vc{u}_z),p)$ that reads
\begin{eqnarray} \label{eq_u1}
\vc{u}_x(x,y,\vc{z}) &=& h^{ \alpha  - \frac 12 } \tilde{{u}}_x( h^{-\frac 12}{\vc{x}},h^{-\frac 12}{\vc{y}},h^{-1}{\vc{z}})  \,, \\
 \vc{u}_y(x,y,\vc{z}) &=& h^{ \alpha - \frac 12} \tilde{{u}}_y( h^{-\frac 12}{\vc{x}},h^{-\frac 12}{\vc{y}},h^{-1}{\vc{z}})  \,, \\
 \vc{u}_z(x,y,\vc{z}) &=&h^{\alpha } \tilde{{u}}_z( h^{-\frac 12}{\vc{x}},h^{-\frac 12}{\vc{y}},h^{-1}{\vc{z}})\,, \\[4pt]
  {p}(x,y,\vc{z}) & = &  h^{ \alpha - 2}\tilde{{p}}( h^{-\frac 12}{\vc{x}},h^{-\frac 12}{\vc{y}},h^{-1}{\vc{z}})\,,  \label{eq_p4}
\end{eqnarray}
in the aperture between the spheres. The parameter $\alpha$ is chosen depending on the values of $\vv{u}^* := (u_x^*,u_y^*,u_z^*).$
For instance, if $u^* = u^*_{\bot} \mathbf{e}_z$ with $u^*_{\bot} \in \mathbb R \setminus \{0\},$ one chooses $\alpha=0.$ We proceed with this particular case. We denote with tildas the new space variables 
$$
 \tilde{\vv{x}} := ({h}^{-\frac 12}x,{h}^{-\frac 12}y,h^{-1}\vc{z}) \,.
$$
These new coordinates belong to the set $\tilde{\mathcal G}_{h}$ that "converges" (when $h \to 0$) to 
$$
\tilde{\mathcal G}^{lub} :=\left \{(\tilde{x},\tilde{y},\tilde{z}) \in \mathbb R^3\, \text{ s.t. }  \, \tilde{z} \in \left(- \dfrac{\tilde{x}^2+\tilde{y}^2}{2R} , 1+ \dfrac{\tilde{x}^2+\tilde{y}^2}{2S}  \right) \right\}.
$$ 
Substituting ansatz \eqref{eq_u1}--\eqref{eq_p4} into \eqref{eq_stokes1}--\eqref{eq_stokes2} yields:
\begin{eqnarray*}
\partial_{\tilde x} \tilde{\vc{u}}_x + \partial_{\tilde y} \tilde{\vc{u}}_y + \partial_{\tilde z} \tilde{\vc{u}}_z &=& 0\,,\\
\partial_{\tilde z\tilde z} \tilde{\vc{u}}_x - \partial_{\tilde x} \tilde{p} &=& 0 \,, \\
\partial_{\tilde z \tilde z} \tilde{\vc{u}}_y - \partial_{\tilde y} \tilde{p} &=& 0\,,\\
\partial_{\tilde z} \tilde{p} &=&0\,,
\end{eqnarray*}
completed with boundary conditions:
\begin{eqnarray*}
\tilde{u}_x = \tilde{u}_y = 0 && \text{ on $\partial \tilde{\mathcal G}^{lub}\,,$}\\[3pt]
\tilde{u}_z = 0 && \text{ on $\partial \tilde{\mathcal G}^{lub}_2 := \left\{ \tilde{z} = \tilde \gamma_2(\tilde x,\tilde y) =: 1+ \dfrac{\tilde{x}^2+\tilde{y}^2}{2S}\,, (\tilde x,\tilde y) \in \mathbb R^2\right\}$}\,,\\
\tilde{u}_z = u^*_{\bot}  &\quad &\text{ on $\partial \tilde{\mathcal G}^{lub}_1 := \left\{ \tilde{z} = \tilde \gamma_1(\tilde x,  \tilde y) =:- \dfrac{\tilde{x}^2+\tilde{y}^2}{2R}\,, (\tilde x,\tilde y) \in \mathbb R^2\right\}$}\,.
\end{eqnarray*}
The pressure is normalized by assuming that it vanishes at infinity. Introducing $\tilde \gamma := \tilde \gamma_2 - \tilde \gamma_1,$ the unique solution to this problem reads $\tilde{\vv u}(\tilde{\mathbf{x}}) = \vv{u}^{lub}(\tilde{\mathbf{x}}),$ $\tilde{p} (\tilde{\mathbf{x}})= p^{lub}(\tilde{x},\tilde{y})$ with:
\begin{eqnarray*}
{u}^{lub}_x(\tilde{x},\tilde{y},\tilde{z})  &=&  \dfrac{1}{2} \partial_{\tilde x} {p}^{lub}(\tilde{x},\tilde{y}) ( \tilde{z} - \tilde{\gamma}_1(\tilde{x},\tilde{y}))( \tilde{z} - \tilde{\gamma}_2(\tilde{x},\tilde{y}))\\
{u}^{lub}_y(\tilde{x},\tilde{y},\tilde{z})  &=&\dfrac{1}{2} \partial_{\tilde y} {p}^{lub}(\tilde{x},\tilde{y}) ( \tilde{z} - \tilde{\gamma}_1(\tilde{x},\tilde{y}))( \tilde{z} - \tilde{\gamma}_2(\tilde{x},\tilde{y}))\\
{u}^{lub}_z(\tilde{x},\tilde{y},\tilde{z}) &=&  \dfrac 12  \text{div}_{\tilde{x},\tilde{y}} \left[ \int^{\tilde{\gamma}_2(\tilde{x},\tilde{y})}_{\tilde{z}} \!\!\!\left((s-\tilde{\gamma}_2(\tilde{x},\tilde{y}))(s-\tilde{\gamma}_1(\tilde{x},\tilde{y})) \right)\nabla_{\tilde x,\tilde y} p^{lub}(\tilde{x},\tilde{y}) {\rm d}s \right]\,,	  
\end{eqnarray*}
and where ${p}^{lub}$ is the unique solution to 
\begin{eqnarray}
- \dfrac{1}{12} {\rm div} (\tilde{\gamma}^3\nabla {p}^{lub})&=&  u^*_{\bot}\,, \quad  \text{on $\mathbb R^2$}\,,\\
\lim_{|(\tilde{x},\tilde{y})| \to \infty}{p}^{lub}(\tilde{x},\tilde{y}) &=&0\,. 
\end{eqnarray}

\medskip

Herein, we justify these formal computations rigorously. A (nonetheless formal) statement of our main result
reads:
\begin{theorem}
Assume {\em (A1)--(A3)} are in force and $\mathbf{u}^* \in C^{\infty}(\partial \Omega)$ satisfies  \eqref{eq_nofluxintro}. Let $\mathbf{v}^*_{\sslash}$ and $v_{\bot}^*$ be given by:
\begin{equation} \label{eq_convention}
\vv{v}^*_{\sslash} = u_x^*(\mathbf 0) \vv{e}_x + u_y^*(\mathbf 0) \vv{e}_y\,, \quad 
v_{\bot}^*(x,y) = u_z^*(\mathbf 0) +   x \partial_x {u}_z^*(\mathbf 0)  + y \partial_y {u}_z^*(\mathbf 0) \,.
\end{equation}
When $h << 1$ while $(\gamma_t,\gamma_b,R_1,\mathbf{u}^*)$ remain of order $1,$ the main contribution to the velocity-field of the solution $(\vv{u},p)$ to the Stokes problem {\em \eqref{eq_stokes1}--\eqref{eq_stokes4}(+\eqref{eq_stokes5})} is given by $\mathbf{v} = (\mathbf{v}_{\sslash},v_{\bot})$ with
\begin{eqnarray} \label{eq_vx}
\mathbf{{v}}_{\sslash}(x,y,z)  &=&  \dfrac{1}{2}  ( z - (h+{\gamma}_t({x},{y})))( z - {\gamma}_b({x},{y})) \nabla_{x,y} {q}({x},{y})+ \dfrac{(h+\gamma_t(x,y)-z)}{\gamma(x,y)}\mathbf{v}^*_{\sslash}\\\label{eq_vy}
{v}_{\bot}(x,y,z) &=&  \dfrac 12  {\rm div}_{x,y} \left[ \int^{h+\gamma_t({x},{y})}_{z} \!\!\!\left((s-\gamma_b({x},{y}))(s-(h+\gamma_t({x},{y}))) \right)\nabla_{x,y} {q}(x,y) {\rm d}s \right]	  \\ 
&& +   \int^{h+\gamma_t({x},{y})}_{z}  {\rm div}_{x,y} \left[ \dfrac{(h+\gamma_t(x,y)-s)}{\gamma(x,y)} \vv{v}^*_{\sslash}  \right]{\rm d}s  \notag
\end{eqnarray}
in the aperture domain $\mathcal G_{L/2} :=  \{(x,y,z) \in \mathbb R^3 \text{ s.t. } (x,y) \in B(\mathbf 0,L/2) \text{ and } z \in (\gamma_b(x,y),h+\gamma_t(x,y))\},$ and where $q$ is the unique solution to 
\begin{eqnarray}
- \dfrac{1}{12} {\rm div} (\gamma^3 \nabla q)&=&  v^*_{\bot} - \frac 12 {\rm div}_{x,y} ((\gamma_t + \gamma_b) \vv{v}^*_{\sslash})\,, \quad  \text{on $B(\mathbf 0,L)$}\,, \label{eq_q}\\
 q &=&0   \quad  \quad \quad  \quad \quad \quad \quad \quad \,\,\,\, \quad \quad \quad \; \; \text{ on $\partial B(\mathbf 0,L)$}\,,\label{cab_q} 
\end{eqnarray}
with $\gamma := h + \gamma_t - \gamma_b.$
\end{theorem}

We introduce here the notations $\sslash$ and $\bot$ for the components of a vector $\vv{v}$ that are respectively parallel and normal to the tangent space to $\partial \Omega$ in the origin. Corresponding decomposition of $\nabla$ are denoted by $\nabla_{x,y}$ and $\partial_z.$ We keep this convention in what follows.
Our geometric assumptions imply that $\vv{e}_x$ and $\vv{e}_y$ are tangent to $\partial \Omega$ in the origin. Hence, the derivatives $\partial_x \vv{u}^*(\mathbf 0)$ and $\partial_{y} \vv{u}^*(\mathbf 0)$ are well-defined. 
 Even though the result concerns an asymptotic behavior when $h \to 0$
the value of this parameter is fixed in all computations. For this reason, we do not let the parameter $h$ appear in most of our notations (such as the fluid domain $\mathcal F$, the distance function $\gamma$ ...). 

\medskip

A more quantitative statement of the above theorem is given in next section. In particular, we make precise for which norms the extracted contribution is the dominating term, the size of remainder terms and the dependencies w.r.t. $\gamma_t,\gamma_b,R_1,\mathbf u^*.$
We prove this result by 
\begin{itemize}
\item computing {\em a priori} estimates on the pressure $q$ as constructed in this statement (in the regime $h<<1$)
\item relating the difference between the exact solution $\vv{u}$ to the Stokes problem and (an extension of) $\vv{v}$ to the obtained estimates on the pressure $q.$
\end{itemize}
As a corollary, we obtain that the leading order in the asymptotics of the Stokes solution is given by the first order expansion of the boundary data $\mathbf u^*$ in the origin. 
This dependance occurs through a pressure solution to a simpler problem but still not explicit. 
We explain in next section that in some special case (when $\mathcal B$ is a sphere and $\Omega= \mathbb R^3 \setminus \mathcal B^*$ with $\mathcal B^*$ a sphere) we can compute accurate informations on the pressure and get explicit expansion of the Stokes solution w.r.t. $h$ 
and the boundary data $\vv{u}^*_{\sslash}(\mathbf 0), u_{\bot}^*(\mathbf 0), \nabla_{x,y} u_{\bot}^*(\mathbf 0).$ 
\medskip

In the set of assumptions we introduced up to now, we enforced the boundary condition to vanish on $\partial \mathcal B.$ We can always reduce asymptotic computations with general Dirichlet boundary conditions to this case. 
We also emphasize that the computations, that we present here in the "single non-denegerate contact" case, extend to general "non-degenerate contacts". A general non-degenerate contact would correspond to the case where the moving solid $\mathcal B$ may collide $\partial \Omega$ in several points satisfying assumption (A2). There would then exist at most a finite number of "contact points" that might be treated separately as "non-degenerate isolated contacts". The linearity of the Stokes problem ensures that the exact solution behaves as the sum of the asymptotic solutions that we compute in the vicinity of each contact point. Previous calculations due to V. Starovoitov show that these "non-degenerate isolated contacts" are the only remaining ones to rule out in order to show that no contact between solid bodies occur in a Stokes or a Navier-Stokes fluid (see \cite{Starovoitov04}).

\medskip

The outline of the paper is as follows. In next section, we recall the classical theory for solving the Stokes problem \stokes\, and give a quantitative statement of our main results, see Theorem \ref{thm_mainexact} and Theorem \ref{cor_mainexact}.
In Section \ref{sec_Gilbarg}, we study the properties of the problem \eqref{eq_q}-\eqref{cab_q}. We compute estimates satisfied by the solution $q$ with respect to the data $(v_{\bot}^*,\vv{v}_{\sslash}^*)$ and compute the divergence-rate of $q$ when $h \to 0$. One particular feature of the estimates we obtain is that they are "uniform" w.r.t. the distance $h$. The last two sections are devoted to the proofs of Theorem \ref{thm_mainexact} and Theorem \ref{cor_mainexact} respectively.


\section{Quantitative statement of main results} \label{sec_galdi}
In this section, we recall function-spaces and existence results for problem \stokes. We give then a quantitative statement of our main result. We conclude by exhibiting two criterions which measure whether a velocity-field is a good approximation of a solution to the Stokes problem or not.

\subsection{Solving the Stokes problem} 

We first recall the way the system \stokes\ is tackled in \cite{Galdibooknew}. 
Let denote: 
\begin{itemize}
\item $\mathcal V :=\{ \vv{u}_{|_{\mathcal F}}, \vv{u} \in C^{\infty}_c (\mathbb R^3) \text{ s.t. } \nabla \cdot \vv{u} =0  \}$ and $\mathcal V_0 :=\{ \vv{u} \in C^{\infty}_c ({\mathcal F}) \text{ s.t. } \nabla \cdot \vv{u} =0  \}\,,$ 
\item $V$ (resp. $V_0$) the completion of $\mathcal V$ (resp. $\mathcal V_0$) endowed with the norm:
$$
\|\vv{u} ; V\| :=  \left[ \int_{\mathcal F} |\nabla \vv{u}|^2\right]^{\frac 12} \,.
$$
\end{itemize}
This makes $(V,\|\ \cdot \ ; V\|)$ (resp. $(V_0,\|\ \cdot\ ; V\|)$) to be a Hilbert space endowed with the scalar product:
$$
((\vv{u},\vv{v} )) = \int_{\mathcal F} \nabla \vv{u} : \nabla \vv{v} \,.
$$
In the terminology of \cite{Galdibooknew}, our space $V$ (resp. $V_0$) is a closed subspace 
of $D^{1,2}(\mathcal F)$ (resp. $D_0^{1,2}(\mathcal F)$), see \cite[p.80]{Galdibooknew}. 
In particular, applying \cite[Theorem II.6.1 (i), (II.6.22)]{Galdibooknew} we have that $V$ embeds in  $L^6(\mathcal F)$ so that $V \subset W^{1,2}_{loc}(\mathcal F).$ This entails that  $\vv{v} \in V$ vanishes at infinity in a weak sense, if required,
and has a well-defined trace $\mathbf{v}_{|_{\partial \mathcal F}} \in H^{\frac 12}(\partial \mathcal F).$ In particular, for arbitrary 
$\vv{u}^* \in H^{\frac 12}(\partial \Omega),$ we might define the affine-subspace of $V$:
$$
V[\vv{u}^*] := \{ \vv{u} \in V \text{ such that } \vv{u}_{|_{\partial \mathcal B}} = 0 \text{ and } \vv{u}_{|_{\partial \mathcal B_*}} = \vv{u}^* \}\,.
$$
We recall that this set is not empty as soon as $\vv{u}^*$ prescribes no flux through $\partial \Omega$ and that we have
the identity $V[0] = V_0$ (see \cite[Theorem II.7.1 (i)]{Galdibooknew}).

\medskip

Following \cite{Galdibooknew} we introduce the definition of generalized solution to \stokes:
\begin{definition}
Let $\vv{u}^* \in H^{\frac 12}(\partial \Omega)$ we call generalized solution to {\em \stokes} any 
$\vv{u} \in V[\vv{u}^*]$ such that 
\begin{equation} \label{eq_Stokesweak}
\int_{\mathcal F} \nabla \vv{u} : \nabla \vv{v} = 0\,, \quad \forall \, \vv{v} \in V_0\,.  
\end{equation}
\end{definition} 
\noindent and we have the classical theorem (see \cite[Theorem V.2.1 and Theorem V.1.1]{Galdibooknew}): 
\begin{theorem} \label{thm_galdi}
Given $\vv{u}^* \in H^{\frac 12}(\partial \Omega),$ such that :
\begin{equation} \label{eq_noflux}
\int_{\partial \Omega} \vv{u}^* \cdot \vv{n}{\rm d}\sigma = 0\,,
\end{equation}
there exists a unique generalized solution $\vv{u}$ to {\em \stokes.} Moreover, 
\begin{itemize}
\item there exists $p \in L^2_{loc}(\mathcal F)$ such that  \eqref{eq_stokes1} holds in 
the sense of distributions,\\[-10pt]
\item if $\vv{u}^* \in C^{\infty}(\partial  \Omega)$ there holds
$(\vv{u},p)  \in C^{\infty}(\overline{\mathcal F} ; \mathbb R^4)\,.$
\end{itemize}
\end{theorem}
A consequence of this result is that generalized solutions with smooth data are classical solutions.
As we only consider this particular case throughout the paper, we drop the adjective "generalized" in what follows. 
Also, we abusively call solution to \stokes\, a velocity-field $\vv{u},$ even though a solution to the Stokes problem
is a pair velocity/pressure $(\vv{u},p).$ 

\subsection{Main results} The main contribution of this paper is the following theorem: 
\begin{theorem} \label{thm_mainexact}
Assume {\em (A1)--(A3)} and boundary condition $\mathbf{u}^* \in C^{\infty}(\partial \Omega)$ satisfies \eqref{eq_noflux}. Let denote by $\vv{u} \in V$ the unique associated solution to {\em \stokes.} If $h<1,$ there exists $\vv{v} \in  V$ and a constant $K := K(R_1,C_b,\delta,\partial \Omega)$ such that:
\begin{itemize}
\item $\vv{v}$ is given by \eqref{eq_vx}-\eqref{eq_vy} in  $\mathcal G_{L/2}$\\[-10pt]
\item If $\|\gamma_t;  C^3(B(\mathbf 0,L)) \| + \|\gamma_b ; C^3(B(\mathbf 0,L)) \| \leq C_b$ and $\|\vv{u}^* ; H^3(B(\mathbf 0,L))\| + \|\vv{u}^* ; H^{\frac 12}(\partial \Omega) \| \leq C_b$ there holds:
$$
\| \vv{u} - \vv{v} ; V \|  \leq K \left( 1+ |u^*_z(\mathbf 0)| |\ln(h)|^{\frac 12}  \right).
$$
\end{itemize} 
If the distance function $\gamma$ is moreover radial, we have the following improvement of the second assertion:
$$
\| \vv{u} - \vv{v} ; V \|  \leq K.
$$
\end{theorem}

We remark that $\vv{u}$ and $\vv{v}$ are bounded uniformly in $h$ as along as this parameter ranges a compact subset of $(0,1].$ Consequently, this theorem brings relevant informations in the limit $h \to 0.$ It shows in particular that even though the fluid is incompressible (so that the Stokes
problem is non-local), when $h$ goes to $0$ the leading order of the 
velocity-field is completely fixed by the boundary conditions and the geometric properties of the boundaries in the aperture $\mathcal G_{L}$.
It will be clear from the example given below that the velocity-field $\vv{v}$ we exhibit is larger than the remainder in the sense of the $V$-norm. We mention
that our first interest in this problem was to compute the diverging term in the case of spheres. As the remainder is bounded in this radial case,  we did not look for a more complicated expansion of the solution. Nevertheless, a corollary of our method is the construction of the linear problem on which could rely the computation of a full expansion of the solution in terms of $h.$ However, we would only describe the solution in the gap $\mathcal G_L$ and the computation of $O(1)$-terms would require  another tool enabling to compute the expansion of the solution outside the gap also. 

\medskip

To highlight the relevance of the above theorem, we detail the case where $\mathcal B$ is a sphere and $\Omega =\mathbb R^3 \setminus \mathcal B^*$ with $\mathcal B^*$ a sphere, as in the computations in the introduction.
Keeping the convention that $\mathcal B^*$ has radius $R$ and $\mathcal B$ has radius $S,$ we have that, close to the origin, the functions $\gamma_t$ and $\gamma_b$ satisfy:
\begin{eqnarray*}
\gamma_t(x,y) &=& \dfrac{x^2 + y^2}{2S} + O((x^2+y^2)^{\frac 32})\,, \\
\gamma_b(x,y) &=& - \dfrac{x^2 + y^2}{2R} + O((x^2+y^2)^{\frac 32})\,, \\   
\end{eqnarray*}
and also that $\gamma$ satisfies:
\begin{eqnarray*}
\gamma(x,y) &=& \dfrac{x^2 + y^2}{2R_1} + \dfrac{(x^2+y^2)^2}{8R_3^3} + O((x^2+y^2)^{3})   
 \end{eqnarray*}
with $(R_1,R_3) \in (0,\infty)^2$ given by:
\begin{equation} \label{eq_R}
\dfrac{1}{R_1} = \dfrac{1}{S}  +  \dfrac{1}{R}\,, \quad \dfrac{1}{R_3^3} = \dfrac{1}{S^3} +   \dfrac{1}{R^3}\,.
\end{equation}
In this case of spheres, we obtain:
\begin{theorem} \label{cor_mainexact}
Given a boundary condition $\mathbf{u}^* \in C^{\infty}(\partial \Omega)$ satisfying \eqref{eq_noflux} there exists a constant $K[R,S]$ depending only on $R,S$ so that if  $h < 1,$ the unique solution $\vv{u} \in V$ to the associated Stokes problem {\em \stokes,} satisfies:
\begin{eqnarray*}
\|\vv{u} ; V\|^2 &=& {6\pi |u^*_{\bot}(\mathbf 0)|^2}\left[\dfrac{R_1^2}{h} +   \left(\dfrac{16 R_1}{5}    -  \dfrac{8 R_1^3}{RS}  -  \dfrac{3  R_1^4}{R_3^3}   \right) |\ln(h)| \right]  \\[6pt] 
&& \quad + 
  \left( 2\pi R_1  |\vv{u}^*_{\sslash\!}(\mathbf 0)|^2 + \dfrac{24 \pi}{5} R_1  | R_1\nabla_{x,y} u_{\bot}^*(\mathbf 0) +\dfrac {(S-R)}{2(R+S)} \vv{u}^*_{\sslash\!}(\mathbf 0)|^2 \right) |\ln(h)|   \\[4pt]
&& \quad + \,   K[R,S](\|\vv{u}^* ; H^3(B(\mathbf 0,L))\|^2 + \|\vv{u}^* ; H^{\frac 12},\partial \Omega\|^2 )\,.
 \end{eqnarray*}
\end{theorem}

\medskip
 
The norms that we compute in this theorem are related to the forces and torques exerted by the fluid flow on the spheres. 
Indeed, assume that the sphere $\mathcal B$ is moving with a rigid velocity and that we want to compute the
forces and torques exerted by a Stokes flow on $\mathcal B$ when the distance to the other sphere $\mathcal B^*$ is small.
By symmetry, we can assume that $\mathcal B^*$ is moving with a rigid velocity $\vv{u}^*$ and compute the force and torque exerted on $\mathcal B^*.$
Then, we recall that the set of rigid velocities is a 6-dimensional vector space whose elements are characterized by a translation
and angular velocity (computed with respect to the origin for simplicity). Let denote by 
$(\vv{U}_{\to,i},P_{\to,i})_{i=1,2,3}$ (resp. $(\vv{U}_{\circlearrowright,i},P_{\circlearrowright,i})_{i=1,2,3}$) the solutions to the Stokes problem
on $\mathcal F$ with elementary translational (resp. rotational) boundary conditions
$$
\begin{array}{rcll}
\vv{U}_{\to,i} &=& \vv{e}_i\,, & \text{ on $\partial \mathcal B^*$}\,, \\
\vv{U}_{\to,i} &=& \vv{0}\,, & \text{ on $\partial \mathcal B$}\,, \\  
\end{array}
\quad 
\left(
\text{ resp. }
\begin{array}{rcll}
\vv{U}_{\circlearrowright,i} &=& \vv{e}_i \times \vv{x}\,, & \text{ on $\partial \mathcal B^*$}\,, \\
\vv{U}_{\circlearrowright,i} &=& \vv{0}\,, & \text{ on $\partial \mathcal B$}\,, \\  
\end{array}
\right)
\quad \text{ for $i=1,2,3\,.$}
$$ 
Due to the linearity of the Stokes problem \stokes, the solution $(\vv{u},p)$ associated with boundary condition $\vv{u}^*$ is then a combination
of the $(\vv{U}_{\to,i},P_{\to,i})_{i=1,2,3}$ and $(\vv{U}_{\circlearrowright,i},P_{\circlearrowright,i})_{i=1,2,3}.$ Furthermore, straightforward integration 
by parts yield that the $j^{th}$ component of the force exerted by the flow on $\mathcal B^*$ reads:
\begin{eqnarray*}
F_j := \int_{\partial \mathcal B^*} (2D(\vv{u})- p\mathbb{I}_3)\vv{\nu} {\rm d}\sigma \cdot \vv{e}_j &=&  2 \int_{\mathcal F}  D(\vv{u}) : D(\vv{U}_{\to,j})    \, \\
&=& \int_{\mathcal F} \nabla \vv{u} : \nabla \vv{U}_{\to, j} \,,
\end{eqnarray*}
(here $\vv{\nu}$ stands for the normal to $\partial \mathcal B^*$ pointing towards $\mathcal B^*$). If $\vv{u}^* = \vv{e}_j,$ we obtain:
$$
F_j = \int_{\mathcal F} |\nabla \vv{U}_{\to,j}|^2\,,
$$
and the asymptotic behavior or $F_j$ when $h \to 0$ yields by applying the above theorem. 
By standard algebraic formulas this identity is generalized to all components of the forces and torques associated with any rigid velocity $\vv{u}^*.$
Then, one can see the asymptotic expansions we compute in Theorem \ref{cor_mainexact} as a justification and an improvement of the asymptotic values for the matrix $\mathcal K$ provided in \cite[Section 7]{Cox74} (see [7.6]).

\subsection{Two approximation criterions}
We conclude this section by providing two criterions which will enable us to measure the distance between an approximation $\vv{v}$ and the exact solution $\vv{u}$ of the Stokes problem \stokes.

\medskip

First, we recall that one way to prove the existence part of {Theorem \ref{thm_galdi}} is to construct $\vv{u}_{bdy} \in V[\vv{u}^*]$ and to remark that we have $V[\vv{u}^*] = \vv{u}_{bdy} + V_0.$ Existence and uniqueness of a generalized solution then yields from an
application of the Stampacchia theorem. This proof entails the expected result together with the following variational characterization:
\begin{proposition} \label{prop_varcar}
Given $\vv{u}^* \in C^{\infty}(\partial \Omega)$ satisfying \eqref{eq_noflux}, 
 the solution $\vv{u}$  to {\em \stokes} realizes:
$$
\|\vv{u};V\|^2 = \min \left\{ \int_{\mathcal F} |\nabla \vv{v}|^2\,, \quad { \vv{v} \in V[\vv{u}^*]} \right\}\,.
$$ 
\end{proposition}
As a consequence, given $\vv{u}^* \in C^{\infty}(\partial \Omega),$ and $\vv{u} \in V$ the generalized solution to \stokes, we have conversely that any $\vv{v} \in V[\vv{u}^*]$ satisfies:
\begin{equation} 
\|\vv{v} \, ; V\|  \geq \|\vv{u}\, ; V\|\,.
\end{equation}
This enables to compute many bounds from above for $\|\vv{u} ; V\|$ by choosing approximations $\vv{v}$. The more relevant the approximation, the sharper the bound. To control the distance between an approximation $\vv{v}$ and the exact solution $\vv{u}$
we have more precisely:
\begin{proposition} \label{prop_car2}
Let $\vv{u}^*\in C^{\infty}(\partial \Omega)$ satisfy  \eqref{eq_noflux} and denote by $\vv{u} \in V$ the solution to {\em \stokes}. Given $({\vv v},q) \in  V[\vv{u}^*]  \times C^{\infty}(\mathcal F)$ we denote:
$$
C[{\vv{v}},{q}] := \|\Delta {\vv v} - \nabla q ; V_0^*\|\,, 
$$
{\em i.e.} $C[{\vv{v}},{q}]$ is the best constant C such that:
\begin{equation} \label{eq_defC}
\left| _{{\mathcal D}'(\mathcal F)}\langle \left( \Delta {\vv{v}} - \nabla q \right), \vv{w} \rangle_{\mathcal D(\mathcal F)} \right|  \leq C \|\vv{w}\, ; \, V\|\,,\quad \forall \, \vv{w} \in \mathcal V_0\,.
\end{equation}
 Then, there holds:
 \begin{equation} \label{eq_criteria2}
 \|{\vv v} - \vv{u} \, ;  \, V\| \leq C[{\vv{v}},{q}] \,.
 \end{equation}
\end{proposition}
\begin{proof}
The proof is standard. For completeness, we recall it briefly. There holds:
\begin{eqnarray*}
\| {\vv{v}} - \vv{u}\, ; V\|^2 &=& \int_{\mathcal F} |\nabla {\vv{v}} - \nabla \vv{u}|^2 \\
						&=& \int_{\mathcal F} \nabla {\vv{v}} : \nabla ({\vv{v}}- \vv{u}) - \int_{\mathcal F} \nabla {\vv{u}} :  \nabla ({\vv{v}}- \vv{u}) \,.
\end{eqnarray*}
Applying \eqref{eq_Stokesweak} to $\vv{w} :=  {\vv{v}}- \vv{u} \in V_0,$ we obtain:
$$
 \int_{\mathcal F} \nabla {\vv{u}} :  \nabla ({\vv{v}}- \vv{u}) =0\,.
$$ 
Introducing then a sequence $\vv{w}_n \in \mathcal V_0$ such that $\lim_{n \to \infty} \vv{w}_n = \vv{w} $ in $V,$ we have:
$$
\| {\vv{v}} - \vv{u}\, ; V\|^2 = \lim_{n \to \infty} \int_{\mathcal F} \nabla {\vv{v}} : \nabla {\vv{w}}_n\,,
$$
where, for all $n \in \mathbb N,$ there holds:
$$ \int_{\mathcal F} \nabla {\vv{v}} : \nabla {\vv{w}}_n  =  \int_{\mathcal F}  (\nabla {\vv{v}} - {q} \mathbb{I}_3 ) :  \nabla {\vv{w}}_n
 =-  \langle (\Delta {\vv{v}} - \nabla q ) ,  {\vv{w}}_n \rangle\,.
$$
Again, as $\vv{w}_n \in V_0,$ we apply definition \eqref{eq_defC} of $C[{\vv{v}},q].$ This entails:
$$
\left| \int_{\mathcal F} \nabla {\vv{v}} : \nabla {\vv{w}}_n\right| \leq C[{\vv{v}},q] \, \|{\vv{w}}_n  ; V\|\,,
$$
and, in the limit $n \to \infty:$
$$
\| {\vv{v}} - \vv{u}\, ; V\|^2 \leq C[{\vv{v}},q] \| {\vv{v}} - \vv{u}\, ; V\|\,.
$$
This ends the proof.
\end{proof}

We emphasize that in the statement of this criterion there is no geometrical constant in front of $C[\mathbf v,q]$
in \eqref{eq_criteria2}. This is particularly important as we want to consider the influence of the geometry
on the relevance of the approximation $\vv{v}$. Actually, these geometrical dependencies are hidden in the computation 
of $C[\mathbf v,q].$

\medskip

\section{Preliminary results on the lubrication problem} \label{sec_Gilbarg}
In this section, we consider the two-dimensional divergence problem: 
\begin{eqnarray}
- \dfrac{1}{12} {\rm div} \left[ \gamma^3  \nabla \varpi \right] &=&  f \,, \quad  \text{on $B(\mathbf 0,L)$}\,,\label{eq_papp1bis}\\[4pt]
\varpi({x},{y}) &=&0\,, \quad  \text{on $\partial B(\mathbf 0,L).$} \label{eq_papp2bis}
\end{eqnarray}
We restrict to source terms $f \in C^{\infty}(\overline{B(\mathbf 0,L)})$ which have the special form:
\begin{equation} \label{eq_formf}%
f =  w^* - \dfrac{1}{2} {\rm div} \, ((\gamma_t + \gamma_b) \vv{v}^*) \,, \quad \text{ where $(\vv{v}^*,w^*) \in C^{\infty}(\overline{B(\mathbf 0,L)} ; \mathbb R^2 \times \mathbb R)$}\,.
\end{equation}
The weight $\gamma \in C^{\infty}(B(\mathbf 0,L))$ is computed with respect to $\gamma_t$ and $\gamma_b$: 
$$
\gamma(x,y) = h + \gamma_t(x,y) - \gamma_b(x,y) \,, \quad \forall \, (x,y) \in B(\mathbf 0,L),
$$ 
with a small but positive distance $h.$ The assumptions (A1)--(A3) are in force.
It is standard that, since $\gamma$ is smooth and bounded from below by a strictly positive constant, there exists a unique smooth solution $\varpi$ to \eqref{eq_papp1bis}-\eqref{eq_papp2bis}.  
Our aim is to compute estimates on quantities of the form:
$$
\int_{B(\mathbf 0,L)} \gamma^n |\nabla^{k} \varpi|^2\,, \quad n \in \mathbb N\,,  k \in \mathbb N\,.
$$
with explicit dependencies in the data $f,$ the distance $h,$ $\gamma_t$ and $\gamma_b.$ 
In these estimates will appear constants depending directly on $\gamma,\gamma_t,\gamma_b.$ We state the definition of these constants as a lemma: 
\begin{lemma} \label{def_W}
There exist constants $(C_{cvx},C_{ell}) \in (0,\infty)$ depending on $\gamma_t$ and $\gamma_b$ only and $(C^{reg}_{k})_{k \in \mathbb N} \in (0,\infty)^{\mathbb N}$   
such that there holds:
\begin{eqnarray} 
 \label{eq_Ck}\|\gamma_t \, ; \, C^{k}(\overline{B(\mathbf 0,L)})\|  + \|\gamma_b \, ; \, C^{k}(\overline{B(\mathbf 0,L)})\|\leq C^{reg}_{k}\,, && \forall \, k \in \mathbb N\,,\\
 \label{eq_Ccvx} C_{cvx} \leq   \Delta \gamma(x,y)  \,, && \forall \, (x,y) \in B(\mathbf 0,L)\,, 
\end{eqnarray}
and, for $h< 1$:
\begin{eqnarray}
\label{eq_Cell}  h+  C_{ell} (x^2 + y^2)  \leq \gamma(x,y) \,, && \forall \, (x,y) \in B(\mathbf 0,L)\,.
\end{eqnarray}
\end{lemma}
The proof of this lemma is an obvious consequence of (A1)-(A3) and is left to the reader.
We remark that the constants $C_{cvx}$ and $C_{ell}$ are related to the constant $R_1$ appearing
in assumption (A2). Assumption (A1) together with \eqref{eq_Cell} imply there exists a constant $K$ depending on $C_{ell}$ and $C^{reg}_2$ such that: 
\begin{eqnarray}
\label{eq_CL}  |\nabla \gamma_t(x,y)| + |\nabla \gamma_b(x,y)|  &\leq& K |\gamma(x,y)|^{\frac 12}\,,  \quad \forall \, (x,y) \in B(\mathbf 0,L)\,,\\
\label{eq_boundtbg}
\gamma_t(x,y) + \gamma_b(x,y) &\leq& K \gamma(x,y)\,,  \qquad \forall \, (x,y) \in B(\mathbf 0,L)\,.
\end{eqnarray}

With these conventions, the main results of this section are the following propositions. We first write two weighted first-order estimates on the solution $\varpi$
of  \eqref{eq_papp1bis}-\eqref{eq_papp2bis} depending on the behavior of the data $\vv{v}^*$ and $w^*$ (recall that $f$ is given by \eqref{eq_formf})
in the origin. 

\begin{proposition} \label{prop_favorable}
Assume that the data $\vv{v}^*,w^*$ satisfy:
\begin{equation} \label{ass_favorable}
\vv{v}^*(\mathbf 0) = \mathbf 0\,, \quad w^*(\mathbf 0) = 0, \quad \nabla w^*(\mathbf 0) = \mathbf 0\,.
\end{equation}
Let $n \in [0,\infty).$ There exists constant $K_n$ depending on $C_{cvx},C_{ell},C_2^{reg}$ and $n$ for which:
\begin{equation} \label{eq_estimationfavorable}
\int_{B(\mathbf 0,L)} |\gamma|^{\frac 52+n} |\nabla \varpi|^2  \leq K_n \left\{ \|\vv{v}^* ; H^2({B}(0,L))\|^2 + \|w^*; H^3({B(\mathbf 0,L)})\|^2 \right\}\,,
\end{equation}
\end{proposition}

\begin{proposition} \label{prop_reg2}
Let $n \in [0,\infty).$ There exists constant $K_n$ depending on $C_{cvx},C_{ell},C_2^{reg}$ and $n$ for which:
\begin{itemize}
{\item if $n=0$ we have:
 \begin{equation} \label{eq_lemest20}
\int_{B(\mathbf 0,L)} \gamma^{3+n} |\nabla \varpi|^2 \leq \dfrac{K_0}{h} \left\{ \|\vv{v}^* ; H^2(B(\mathbf 0,L))\|^2 + \|w^* ; H^2(B(\mathbf 0,L))\|^2 \right\} ,
\end{equation}
\item if $n \in (0,1)$ there holds:
\begin{equation} \label{eq_lemest20bis}
\int_{B(\mathbf 0,L)} \gamma^{3+n} |\nabla \varpi|^2 \leq  K_n\left\{ \dfrac{|w^*(\mathbf 0)|^2}{h^{1-n}} + \|\vv{v}^* ; H^2(B(\mathbf 0,L))\|^2 + \|w^* ; H^2(B(\mathbf 0,L))\|^2 \right\} ,
\end{equation}
\item if $n=1$ we have:
\begin{equation} \label{eq_lemest22}
\int_{B(\mathbf 0,L)} \gamma^{4} |\nabla \varpi|^2 \leq K_1 \left\{ |w^*(\mathbf 0)|^2 |\ln(h)| +  \|\vv{v}^* ; L^2(B(\mathbf 0,L))\|^2 + \|w^* ; H^2(B(\mathbf 0,L))\|^2 \right\} 
\end{equation}
}
\item if $n > 1,$ there holds:
\begin{equation} \label{eq_lemest2}
\int_{B(\mathbf 0,L)} |\gamma|^{3+ n} |\nabla \varpi|^2 \leq K_n \left\{
 \|\vv{v}^* ; L^2(B(\mathbf 0,L))\|^2 + \|w^* ; H^1(B(\mathbf 0,L))\|^2 \right\}
\end{equation}
\end{itemize}
\end{proposition} 

We emphasize that the constants $K_n$ do not depend on $h.$ In particular, the latter proposition implies that, when $w^*(\mathbf 0)$ vanishes, all quantities:
$$
\int_{B(\mathbf 0,L)} |\gamma|^{3+ n} |\nabla \varpi|^2, \quad n \in \mathbb N,
$$
remain bounded when $h$ goes to $0.$ We complement the study with higher order estimates away and around the singularity point $(x,y) = (0,0)$:
 \begin{proposition} \label{prop_unifreg}
Given a integer $k \geq 0$ and $0 < \varepsilon < L,$
there exists a positive constant $K^{reg}_k$ depending only on $L,\varepsilon,k,C_{ell},C_{cvx},C^{reg}_{k'},$ with $k' = \max (k+1,2),$ for which $\varpi$ satisfies: 
\begin{equation} 
\|\varpi ; H^{k+2}(B(\mathbf 0,L) \setminus B(\mathbf 0,\varepsilon)) \|\\
 \leq K^{reg}_{k}\left[ \|\vv{v}^* ; H^{k+1}(B(\mathbf 0,L))\| + \|w^* ; H^{k'-1}(B(\mathbf 0,L))\|\right]\,.
\end{equation} 
\end{proposition}

\begin{proposition} \label{prop_reg}
Let $k \in \{0,\ldots,6\},$  $n \in  (k,\infty)$ and denote $e_k := \max(k,3).$
There exists a constant $K^{sing}_{n,k}$ depending on $k,n,C_{ell},C_{cvx},C^{reg}_{k'},$ with $k'= \max(k+1,2),$ such that:
\begin{multline} 
\int_{B(\mathbf 0,L)} |\gamma|^{3+n} |\nabla^{k+1} \varpi|^2 \leq K^{sing}_{n,k} \Big\{ 
\int_{B(\mathbf 0,L)} |\gamma|^{3+(n-k)} |\nabla \varpi|^2 +  \|\vv{v}^* ; H^{k+1}(B(\mathbf 0,L))\|^2 \\+ \|w^* ; H^{e_k}(B(\mathbf 0,L))\|^2 \Big\}\,. 
\end{multline}
\end{proposition}

In the remainder of this section, we first give proofs for these propositions.  We then conclude by showing that the first order estimates are optimal in the case
of sphere.

\subsection{First order estimates : proofs of propositions \ref{prop_favorable} and \ref{prop_reg2}}
System  \eqref{eq_papp1bis}-\eqref{eq_papp2bis} is associated with the weak formulation:
\begin{equation}
\dfrac{1}{12}\int_{B(\mathbf 0,L)} \gamma^3 \nabla \varpi \cdot \nabla \varphi = \int_{B(\mathbf 0,L)} f \varphi \,, \quad \forall \, \varphi \in C^{\infty}_c(B(\mathbf 0,L))\,.
\end{equation}
Hence, to construct solutions, one introduces the bilinear form: 
$$
(\varphi,\psi) \mapsto ((\varphi,\psi))_{\gamma} := \dfrac{1}{12}\int_{B(\mathbf 0,L)} \gamma^3 \nabla \varphi \cdot \nabla \psi . 
$$ 
A definition for weak solution to \eqref{eq_papp1bis}-\eqref{eq_papp2bis} then reads:
\begin{definition}
We call weak solution
to \eqref{eq_papp1bis}-\eqref{eq_papp2bis} any $\varpi \in H^1_0(B(\mathbf 0,L))$ such that:
$$
((\varpi, \varphi))_{\gamma} = \int_{B(\mathbf 0,L)} f  \varphi \,, \quad \forall \, \varphi \in H^1_0(B(\mathbf 0,L))\,.
$$
\end{definition}
Thanks to \eqref{eq_Ck}-\eqref{eq_Cell}, $((\cdot,\cdot))_{\gamma}$ defines a coercive continuous bilinear form on $H^1_0(B(\mathbf 0,L))$. Existence and uniqueness of a weak solution to \eqref{eq_papp1bis}-\eqref{eq_papp2bis} yields as a straightforward application of the Stampacchia theorem. Given our regularity assumptions on the distance function $\gamma,$ classical results on divergence problems apply so that this weak solution is smooth on $\overline{B(\mathbf 0,L)}$ and satisfies  \eqref{eq_papp1bis}-\eqref{eq_papp2bis} in a classical sense  (see \cite[Chapter 8]{GilbargTrudingerbook}).   All the estimates in \cite[Chapter 8]{GilbargTrudingerbook} depend {\em a priori} deeply on $h$
so that an alternative approach is required. 

\medskip

In the proofs below, we denote with $K$ a constant which is important to our computations and shall put in brackets its relevant parameters (such as $(C_{cvx},C_{ell}),(C^{reg}_{k})_{k \in \mathbb N})$. These constants may differ from line to line. Most of these constants shall also depend on the parameter $L$ but we include this tacitly. Notations $C$ are used for generic constants that depend only on $L$ or on other parameters that are irrelevant to our computations. Again, they may differ from line to line. Regarding constants $K,$ it might appear that a constant depend on $C^{reg}_k$ and $C^{reg}_j$ for $k\neq j.$ In this case, we have that the constant $K$ depends on the $k$-first derivatives of $\gamma_t$ and $\gamma_b$ and also
on the $j$-first derivatives of $\gamma_t$ and $\gamma_b.$ We shall simplify $K$ then into a function depending only on the $l$-first derivatives of $\gamma_t$ and $\gamma_b$ with $l = \max(k,j),$
{\em i.e.}, on $C^{reg}_{l}.$

\medskip

From now on, we fix data $(\vv{v}^*,w^*)$ and denote the associated (weak) solution $\varpi$.
The first step of the proofs is the following preliminary lemma: 
\begin{lemma} \label{lem_reg2}
Given $\alpha \in [-1,2),$ there exists a constant $K_{\alpha}:= K[C_{cvx},C_2^{reg},C_{ell},\alpha]$  s.t.  $\varpi$ satisfies:
\begin{equation} 
\int_{B(\mathbf 0,L)} |\gamma|^{3+ \alpha} |\nabla \varpi|^2 \leq K_{\alpha}  \int_{B(\mathbf 0,L)} \left[{\gamma^{\alpha-1}}{ |\vv{v}^*|^2} + {\gamma^{\alpha-2}}{|w^*|^2} \right] 
\end{equation}
\end{lemma} 
\begin{proof}
Given $\alpha \geq -1,$ as $\gamma$ is a positive function, explicit computations yield that:
$$
\Delta [ \gamma^{\alpha+3}] = (\alpha+3)[ \Delta \gamma] \gamma^{\alpha+2} + (\alpha+3)(\alpha+2)|\nabla \gamma|^{2} \gamma^{\alpha+1}\,.
$$
Introducing the bound \eqref{eq_Ccvx} on $\Delta \gamma$, we deduce that:
\begin{eqnarray}
|\nabla \gamma|^{2} \gamma^{\alpha+1}& \leq& \dfrac{1}{(\alpha+3)(\alpha+2)} \Delta [\gamma^{\alpha+3}]\,, \label{eq_n+1}\\
\gamma^{\alpha+2} & \leq &  \dfrac{1}{{ 2} C_{cvx}} \Delta [\gamma^{\alpha+3}]\,. \label{eq_n+2}
\end{eqnarray}
Now, for any $\varphi \in  C^{\infty}(\overline{B(\mathbf 0,L)})$, 
integrating by parts and applying H\"older inequality yields:
\begin{eqnarray}\notag
0 \leq  \int_{B(\mathbf 0,L)} \Delta [ \gamma^{\alpha+3}]  |\varphi|^2  &=&  -2(\alpha+3) \int_{B(\mathbf 0,L)} \gamma^{\alpha+2} \nabla \gamma \cdot \varphi \nabla \varphi + \int_{\partial B(\mathbf 0,L)} |\varphi|^2  \partial_{\nu} \gamma^{\alpha+3} {\rm d}\sigma \\ \notag
&\leq& (\alpha+3) \left[ \dfrac{\alpha+2}{2} \int_{B(\mathbf 0,L)} \gamma^{\alpha+1} |\nabla \gamma|^2 \varphi^2 + \dfrac{2}{\alpha+2} \int_{B(\mathbf 0,L)} \gamma^{\alpha+3} |\nabla \varphi|^2 \right]\\
&& + K[C^{reg}_{1},\alpha] \int_{\partial B(\mathbf 0,L)}|\varphi|^2{\rm d}\sigma\,. \notag
\end{eqnarray}
Applying then \eqref{eq_n+1} to bound the first term in the right-hand side of this last inequality,
we obtain:
\begin{multline} \label{eq_tech1}
\left| \int_{B(\mathbf 0,L)} \Delta [ \gamma^{\alpha+3}]  |\varphi|^2 \right|  \leq \dfrac{4(\alpha +3)}{\alpha+2} \int_{B(\mathbf 0,L)} \gamma^{\alpha+3} |\nabla \varphi|^2 +  K[C^{reg}_{1},\alpha] \int_{\partial B(\mathbf 0,L)}|\varphi|^2 {\rm d}\sigma.
\end{multline}

\bigskip
 
We multiply now \eqref{eq_papp1bis} with $\gamma^{\alpha}\varpi.$ This yields:
\begin{multline}
\int_{B(\mathbf 0,L)} \gamma^{\alpha+3} |\nabla \varpi|^2 - \dfrac{\alpha}{2(\alpha+3)} \int_{B(\mathbf 0,L)}  \Delta [ \gamma^{\alpha+3}] |\varpi|^2\\
 = 12 \int_{B(\mathbf 0,L)}  \left(  w^* - \dfrac 12 {\rm div} [(\gamma_t+\gamma_b) \vv{v}^*]  \right) \gamma^{\alpha} \varpi  \,. 
 \label{eq_debase}
\end{multline}
We compute separately the {\em RHS} and the {\em LHS} of this identity. 
If $\alpha < 0$ the last term on the RHS is positive. If $\alpha \in [0,2],$
we remark that $\varpi$ vanishes on $\partial B(\mathbf 0,L)$ and apply \eqref{eq_tech1} to $\varpi.$
This entails:
\begin{equation} \label{eq_LHS}
LHS \geq \left( 1 - \dfrac{2\max(\alpha,0)}{(\alpha+2)}\right) \int_{B(\mathbf 0,L)} \gamma^{\alpha+3} |\nabla \varpi|^2.
\end{equation}
We note here that, since $\alpha < 2$ the factor appearing in the right-hand side of this identity is positive. The fact that this factor changes sign when $\alpha$ crosses $2$ makes this proof irrelevant for $\alpha \geq 2.$  

\medskip

Concerning the {\em RHS}, we have first, applying  \eqref{eq_n+2} and \eqref{eq_tech1} (with $\varphi = \varpi$ and noting that $4(\alpha +3)/(\alpha+2) \leq 8$), that, for arbitrary $\varepsilon >0,$ there holds:
$$
\left| 12 \int_{B(\mathbf 0,L)} w^* \gamma^\alpha \varpi \right| \leq  \dfrac{K[C_{cvx}]}{\varepsilon} \int_{B(\mathbf 0,L)} |w^*|^2 \gamma^{\alpha-2}
+\dfrac{\varepsilon}{8} \int_{B(\mathbf 0,L)}\gamma^{\alpha+3} |\nabla \varpi|^2  \,. 
 $$
 Similarly, by applying \eqref{eq_n+1}, \eqref{eq_n+2}  and \eqref{eq_tech1} (again with $\varphi = \varpi$) and \eqref{eq_boundtbg},
 we obtain that, for $\alpha \in [-1,2)$ and arbitrary $\varepsilon >0,$ there holds:
\begin{multline*}
\left| \int_{B(\mathbf 0,L)} {\rm div}\,( \gamma_t + \gamma_b) \vv{v}^*  \gamma^{\alpha} \varpi \right|
\leq  
  \left|\alpha \int_{B(\mathbf 0,L)}    (\gamma_t+ \gamma_b) \vv{v}^* \cdot \nabla \gamma \,  \gamma^{\alpha -1}  \varpi \right|
+
\left| \int_{B(\mathbf 0,L)}     (\gamma_t+ \gamma_b) \vv{v}^*  \cdot \nabla \varpi \gamma^{\alpha} \right|\, \\
\begin{array}{rcl}
& \leq & \displaystyle
\dfrac{K[C_{ell},C_{2}^{reg}]}{\varepsilon}\int_{B(\mathbf 0,L)}  |\vv{v}^*|^2 \gamma^{\alpha-1} 
+ 
{\varepsilon}\int_{B(\mathbf 0,L)} \gamma^{\alpha+1} |\nabla \gamma|^2 |\varpi|^2
+
{\varepsilon}\int_{B(\mathbf 0,L)} \gamma^{\alpha+3}  |\nabla \varpi|^2\\[12pt]
& \leq &\displaystyle
\dfrac{{K[C_{ell},C_{2}^{reg}]}}{\varepsilon} \int_{B(\mathbf 0,L)}  |\vv{v}^*|^2 \gamma^{\alpha-1} 
+ 
{\varepsilon}\int_{B(\mathbf 0,L)} \Delta[\gamma^{\alpha+3}] |\varpi|^2
+ 
{\varepsilon}\int_{B(\mathbf 0,L)} \gamma^{\alpha+3}  |\nabla \varpi|^2\\[12pt]
& \leq &\displaystyle
\dfrac{{K[C_{ell},C_{2}^{reg}]}}{\varepsilon} \int_{B(\mathbf 0,L)}  |\vv{v}^*|^2 \gamma^{\alpha-1} 
+ 
\dfrac{\varepsilon}{4}\int_{B(\mathbf 0,L)} \gamma^{\alpha+3}  |\nabla \varpi|^2\,.\\
\end{array}
\end{multline*}
This yields finally that we have, for arbitrary $\varepsilon >0$:
\begin{multline} \label{eq_RHS}
RHS \leq \dfrac{K[C_{ell},C_{2}^{reg},C_{cvx}]}{\varepsilon}\left( \int_{B(\mathbf 0,L)}  |\vv{v}^*|^2 \gamma^{\alpha-1} 
+  \int_{B(\mathbf 0,L)} |w^*|^2 \gamma^{\alpha-2} \right)\\
+ \dfrac{\varepsilon}{4}\int_{B(\mathbf 0,L)} \gamma^{\alpha+3}  |\nabla \varpi|^2\,.
\end{multline}
We replace finally the right-hand side and left-hand side of \eqref{eq_debase} with \eqref{eq_RHS} and \eqref{eq_LHS} and obtain the expected result by choosing $\varepsilon$ sufficiently small depending on $\alpha$.
\end{proof}

We are now in position to prove {Proposition \ref{prop_favorable}} and {Proposition \ref{prop_reg2}}.

\begin{proof}[Proof of Proposition \ref{prop_favorable}]
Because $\gamma$ is bounded by $C^{reg}_0$ on $\overline{B(\mathbf 0,L)},$ it is sufficient to prove \eqref{eq_estimationfavorable} in the case $n=0.$ 
Under the further assumption \eqref{ass_favorable}, there holds (from the two-dimensional embedding $H^{m+2}(B(\mathbf 0,L)) \subset C^{m,3/4}(B(\mathbf 0,L))$)
for arbitrary $m \in \mathbb N \cup \{0\})$
\begin{eqnarray}
|\vv{v}^*(x,y)| &\leq& C (|x|^2 + |y|^2)^{\frac 38} \|\vv{v}^* ; H^2({B(\mathbf 0,L)})\| \,, \label{eq_fav1} \\
 |w^*(x,y)| &\leq&  C (|x|^2 + |y|^2)^{\frac 78} \|w^* ; H^3({B(\mathbf 0,L)})\| \,, \label{eq_fav2}
\end{eqnarray}
for $(x,y) \in B(\mathbf 0,L).$

\medskip

Applying then {Lemma \ref{lem_reg2}} to $\varpi$ with $\alpha=-1/2,$ we obtain:
\begin{eqnarray*}
\int_{B(\mathbf 0,L)} \gamma^{\frac 52} |\nabla \varpi|^2 &\leq&  K[C_{ell},C_{cvx},C^{reg}_2] \left( 
\int_{B(\mathbf 0,L)} \dfrac{|\vv{v}^*|^2}{\gamma^{\frac 32}} + \dfrac{|w^*|^2}{\gamma^{\frac 52}} \right) 
\end{eqnarray*}
Here we call \eqref{eq_fav1}-\eqref{eq_fav2} and \eqref{eq_Cell} to obtain:
\begin{eqnarray*}
\int_{B(\mathbf 0,L)} \dfrac{|\vv{v}^*|^2}{\gamma^{\frac 32}} + \dfrac{|w^*|^2}{\gamma^{\frac 52}}  & \leq &
 C \int_{0}^{L} \left\{ \dfrac{r^{\frac 32}\|\vv{v}^* ; H^2({B(\mathbf 0,L)})\|^2}{{C_{ell}^{3/2}}r^3} + \dfrac{r^{\frac 72}\|w^* ; H^3({B(\mathbf 0,L)})\|^2}{{C_{ell}^{5/2}} r^5} \right\}{r{\rm d}r} \\
& \leq &  K[C_{ell}] \left\{\|\vv{v}^* ; H^2({B(\mathbf 0,L)})\|^2 + \|w^* ; H^3({B(\mathbf 0,L)})\|^2 \right\}.
\end{eqnarray*}
This ends the proof.
\end{proof}

\begin{proof}[Proof of Proposition \ref{prop_reg2}] 
Because $\gamma$ is bounded by $C^{reg}_0$ on $\overline{B(\mathbf 0,L)},$ it is sufficient to prove \eqref{eq_lemest2} and \eqref{eq_lemest22} in the cases $n\in [0,1),$ $n=1$ and  $n \in (1,2).$ 

\medskip

\paragraph{Case $n \in (1,2)$} Applying Lemma \ref{lem_reg2} with $\alpha=n $, we have, with a constant $K_n$ depending on $C_{cvx},C_{ell},C_{2}^{reg}$ and $n$:
\begin{eqnarray*}
\int_{B(\mathbf 0,L)} |\gamma|^{3+n} |\nabla \varpi|^2 
&\leq & K_{n} \left[ \int_{B(\mathbf 0,L)}  \gamma^{n-1}|\vv{v}^*|^2 + \gamma^{n-2} |w^*|^2\right] \\
\end{eqnarray*}
Here, we note that $n>1$ so that $2-n < 1.$ Hence, we might fix $p_n \in (1,\infty)$ such that 
$p_n(2-n) < 1.$ Introducing $q_n$ the conjugate exponent, we get:
 $$
 \int_{B(\mathbf 0,L)}  \gamma^{n-2} |w^*|^2 \leq \left( \int_{B(\mathbf 0,L)} \dfrac{1}{\gamma^{(2-n)p_n}}\right)^{\frac 1{p_n}} \|w^* ; L^{2q_n}(B(\mathbf 0,L))\|^2\,.
 $$
In the right-hand side of this last line, we note that standard 2D-Sobolev imbeddings yield 
$$
\|w^* ; L^{2q_n}(B(\mathbf 0,L))\|^2 \leq C_n \|w^* ; H^1(B(\mathbf 0,L))\|^2
$$ 
and that, from \eqref{eq_Cell}:
$$
\int_{B(\mathbf 0,L)} \dfrac{1}{\gamma^{(2-n)p_n}} \leq K[C_{ell},n] \int_{0}^{L} \dfrac{r{\rm d}r}{ r^{2(2-n)p_n}} \leq K[C_{ell},L,n],
$$
as $2(2-n)p_n < 2.$ This ends up the proof in the last case.

\medskip

\paragraph{Case $n=1$} Applying the embedding $H^2(B(\mathbf 0,L)) \subset C^{0,3/4}(B(\mathbf 0,L))$, we have:
$$
|w^*(x,y)| \leq |w^*(\mathbf 0)| +  C (x^2+ y^2)^{\frac 38} \|w^* ; H^2(B(\mathbf 0,L))\|\,, \quad \forall \, (x,y)  \in B(\mathbf 0,L).
$$ 
The remainder of the computation follows the line of the previous case. Applying Lemma \ref{lem_reg2} with $\alpha =1,$
we have, with a constant $K$ depending on the same quantities:
\begin{multline*}
\int_{B(\mathbf 0,L)} |\gamma|^{4} |\nabla \varpi|^2
\leq  K_1 \left[ \int_{B(\mathbf 0,L)}  |\vv{v}^*|^2 + \dfrac {|w^*|^2}\gamma \right] \\
\begin{array}{rcl}
&\leq & \displaystyle K \Bigg[ \|\vv{v}^*; L^2(B(\mathbf 0,L))\|^2 + \int_{B(\mathbf 0,L)} \dfrac{|w^*(\mathbf 0)|^2}{\gamma} + 
\int_{B(\mathbf 0,L)} \dfrac{r^{\frac 32}}{\gamma} \|w^* ; H^2(B(\mathbf 0,L))\|^2 \Bigg]\,, \\[12pt]
& \leq &\displaystyle K  \Bigg[ \|\vv{v}^*; L^2(B(\mathbf 0,L))\|^2 + \int_{0}^L \dfrac{|w^*(\mathbf 0)|^2 r{\rm d}r}{h+ C_{ell} r^2} +  \|w^* ; H^2(B(\mathbf 0,L))\|^2 \Bigg]\,,\\[12pt]
& \leq &\displaystyle K  \Big[ \|\vv{v}^*; L^2(B(\mathbf 0,L))\|^2 + |w^*(\mathbf 0)|^2  |\ln(h)| +  \|w^* ; H^2(B(\mathbf 0,L))\|^2 \Big]\,.
\end{array}
\end{multline*}
{\paragraph{Case $n\in (0,1)$} Again, we apply Lemma \ref{lem_reg2} with $\alpha =n,$
and apply the embedding $H^2(B(\mathbf 0,L)) \subset C^{0,1-n/2}(B(\mathbf 0,L)).$
This yields with similar computations as previously:
\begin{multline*}
\int_{B(\mathbf 0,L)} |\gamma|^{3+n} |\nabla \varpi|^2
\leq  K_n \left[ \int_{B(\mathbf 0,L)} \dfrac{|\vv{v}^*|^2}{\gamma^{1-n}} + \dfrac {|w^*|^2}{\gamma^{2-n}} \right] \\
\begin{array}{rcl}
&\leq & \displaystyle K[C_0^{reg},C_{ell},n] \Bigg[ \int_0^L \dfrac{\|\vv{v}^*; H^2(B(\mathbf 0,L))\|^2 r{\rm d}r}{\gamma^{1-n}}+ \int_0^L\dfrac{|w^*(0)|^2r + r^{3-n}\|w^* ; H^2(B(\mathbf 0,L))\|^2}{\gamma^{2-n}}   {\rm d}r \Bigg] \,, \\[12pt]
& \leq &\displaystyle K[C_0^{reg},C_{ell},n]  \Bigg[ \|\vv{v}^*; H^2(B(\mathbf 0,L))\|^2 +  \dfrac{|w^*(0)|^2}{h^{1-n}}  +   \|w^* ; H^2(B(\mathbf 0,L))\|^2 \Bigg] \,.
\end{array}
\end{multline*}
 \paragraph{Case $n=0$} We conclude with similar arguments, applying  \ref{lem_reg2} with $\alpha =0,$
and the embedding $H^2(B(\mathbf 0,L)) \subset C^{0}(B(\mathbf 0,L)).$
}
This ends the proof.
\end{proof}

\subsection{Higher order estimates: proofs of propositions \ref{prop_unifreg} and \ref{prop_reg}}
We compute now estimates involving $\nabla^k \varpi$ with $k \geq 1$.
We start with estimates away from the origin.
{ 
\begin{proof}[Proof of Proposition \ref{prop_unifreg}] 
This proposition follows from the classical regularity theory for elliptic problems, as developed in \cite{GilbargTrudingerbook}. Indeed, fix {$k \geq 0$} and $0 < \varepsilon < L < \infty.$ 
Applying {Proposition \ref{prop_reg2}} with $n=2,$ and combining with \eqref{eq_Cell}, we obtain for a constant $K$ wich depends on $C_{cvx},C_{ell},C_{2}^{reg}$ and $L$
that
\begin{eqnarray} \notag
[C_{ell}\,\varepsilon^{2}]^{5} \int_{B(\mathbf 0,L) \setminus B(0,\varepsilon)} |\nabla \varpi|^2 
& \leq &\int_{B(\mathbf 0,L)} |\gamma|^{5} |\nabla \varpi|^2  \\ 
&\leq & K \left[\|\vv{v}^*; L^2(B(\mathbf 0,L))\|^2 +\|w^* ; H^1(B(\mathbf 0,L))\|^2 \right]. 
\end{eqnarray} 
{As $\varpi = 0$ on $\partial B(\mathbf 0,L),$ we deduce from this inequality (applying the variant of Poincar\'e inequality given in \cite[Exercicse II.5.13]{Galdibooknew}) that there exists a constant $K := K[\varepsilon,C_{cvx},C_{ell},C_{2}^{reg},L]$ so that:
\begin{equation}\label{eq_elliptic0}
\|\varpi ; H^1(B(\mathbf 0,L)\setminus B(\mathbf 0,\varepsilon))\| \leq K \left[\|\vv{v}^*; L^2(B(\mathbf 0,L))\|^2 +\|w^* ; H^1(B(\mathbf 0,L))\|^2 \right].
\end{equation}
}
Then, we note that the operator $L := {\rm div} [\gamma^3 \nabla ] $ is uniformly elliptic on $B(\mathbf 0,L) \setminus B(0,\varepsilon)$ with ellipticity constant $\lambda_{\varepsilon} := \inf_{B(\mathbf 0,L) \setminus B(0,\varepsilon)} \gamma^3 \geq (C_{ell} \varepsilon^2)^3$ depending only on $C_{ell}$ and $\varepsilon.$ Finally, {applying \cite[Theorem 8.8 p.183 with Theorem 8.12 p.186]{GilbargTrudingerbook} if $k=0$  or \cite[Theorem 8.10 p.186 with Theorem 8.13, p.187]{GilbargTrudingerbook} if $k \geq1$} implies that there exists a constant $\tilde{K}^{reg}_k$ depending on $\varepsilon,C_{ell}$ and $C_{k+1}^{reg}$ so that: 
\begin{multline*}
\|\varpi ; H^{k+2}(B(\mathbf 0,L)\setminus B(\mathbf 0,\varepsilon))\| \\
\begin{array}{l}
  \leq \tilde{K}^{reg}_k
\left[ \|\varpi ; H^{1}(B(\mathbf 0,L)\setminus B(\mathbf 0,\varepsilon))\|  + \|f ; H^k(B(\mathbf 0,L)\setminus B(\mathbf 0,\varepsilon))\| \right]\,, \\[6pt]
\leq {K}^{reg}_k \left[ \|\vv{v}^*; H^{k+1}(B(\mathbf 0,L))\|^2 + \|w^* ; H^{k'-1}(B(\mathbf 0,L))\|^2 \right]\,,
 \end{array}
 \end{multline*}
 where we applied \eqref{eq_elliptic0} to reach the last line.
\end{proof}
}
 This proposition is a tool for computing the traces of derivatives of $\varpi$
on $\partial B(\mathbf 0,L).$ With this consideration, we obtain Proposition \ref{prop_reg}.

\begin{proof}[Proof of Proposition \ref{prop_reg}] 
To prepare the proof, we mention that the chain rule together with \eqref{eq_CL} yields that for any integer
$i \in \{0,\ldots,6\}$ there exists a constant $K[i]$ depending on $C_{i}^{reg},C_{ell}$ and $C^{reg}_2,$ for which: 
\begin{equation} \label{eq_Ki}
|\nabla^{i} \gamma^3(x,y)| \leq K[i] |\gamma(x,y)|^{3 - \frac i2}\,, \quad \forall \, (x,y) \in B(\mathbf 0,L)\,.
\end{equation}
This inequality no longer holds when $k >6.$ Again, this is the reason for our lemma to hold only for values of the parameter $k$ below $6.$

\medskip

We prove now by induction on $k$ for fixed $m \in (0,\infty)$ that there holds:\\[4pt]
"
there exists a constant $K^{sing}_{m,k}$ depending on $k,m,C_{ell},C_{cvx},C^{reg}_{k'},$ with $k'= \max(k+1,2),$ such that, denoting by $e_k = \max(k,3)$:
\begin{multline} \label{eq_est22bis}
\int_{B(\mathbf 0,L)} |\gamma|^{3+k+m} |\nabla^{k+1} \varpi|^2 \leq K^{sing}_{m,k} \Big\{ 
\int_{B(\mathbf 0,L)} |\gamma|^{3+m} |\nabla \varpi|^2 +  \|\vv{v}^* ; H^{k+1}(B(\mathbf 0,L))\|^2 \\+ \|w^* ; H^{e_k}(B(\mathbf 0,L))\|^2 \Big\}\,. "
\end{multline}

The case $k=0$ is obvious.
To prove the induction argument, we fix $k \in \{0,\ldots,5\}$ and introduce $D^{k+1}$ a (homogeneous) differential operator of order $k+1$ having constant coefficients. Applying $D^{k+1}$ to \eqref{eq_papp1bis} yields:
\begin{equation} \label{eq_derlap}
- \dfrac{1}{12} {\rm div} [\gamma^{3} \nabla D^{k+1} \varpi ] = {\rm div} \left[  R^{k+1} - D^{k+1} \dfrac{1}{2}(\gamma_t + \gamma_b) \vv{v}^*  \right]+ D^{k+1} w^*\,. 
\end{equation}
where:
$$
R^{k+1} :=  \dfrac {1}{12}\left( D^{k+1} [ \gamma^{3} \nabla \varpi ] - \gamma^3 \nabla D^{k+1} \varpi \right).
$$
We multiply \eqref{eq_derlap} with $\gamma^{k+m+1}D^{k+1}\varpi$.
 After integration by parts, this yields:
\begin{equation} \label{eq_induction}
\dfrac{1}{12}\int_{B(\mathbf 0,L)} \gamma^{3+m+(k+1)} |\nabla D^{k+1} \varpi|^2 = \dfrac{RHS_1}{12} + RHS_2,   
\end{equation}
where:
\begin{eqnarray*}
RHS_1 &=& \int_{\partial B(\mathbf 0,L)} \gamma^{4+m+k}\partial_{\nu} D^{k+1} \varpi D^{k+1} \varpi  - \int_{B(\mathbf 0,L)}
(k+m+ 1) \nabla \gamma \gamma^{3+m+k} \nabla D^{k+1} \varpi D^{k+1} \varpi \, \\
RHS_2 &=& \int_{B(\mathbf 0,L)} (D^{k+1} w^* - {\rm div} D^{k+1} \dfrac{(\gamma_t+\gamma_b) \vv{v}^*}{2} ) \gamma^{k+m+1} D^{k+1} \varpi \\ 
&&\quad  - \int_{B(\mathbf 0,L)} R^{k+1} \cdot (\gamma^{k+m+1} \nabla D^{k+1} \varpi + (k+m+1) \nabla \gamma \gamma^{k+m} D^{k+1} \varpi)\,.\\
  && \quad +
\int_{\partial B(\mathbf 0,L)} R^{k+1} \cdot \nu \gamma^{k+m+1} D^{k+1} \varpi  
\end{eqnarray*}
We bound the first term in $RHS_1$  by applying trace theorems and {Proposition \ref{prop_unifreg}}: there exists a constant $C[D^{k+1}]$ (depending on $D^{k+1}$ and also on $L$) for which
\begin{multline*}
\int_{\partial B(\mathbf 0,L)}\left[ |D^{k+1}\varpi|^2 + |\partial_{\nu} D^{k+1} \varpi|^2\right]{\rm d}\sigma\\
\begin{array}{rcl}
&\leq& C[D^{k+1}] \|\nabla^{k+1}\varpi ; H^{1}(\partial (B(\mathbf 0,L)/B(\mathbf 0,L/2)))  \|^2 \\[6pt]
&\leq& C[D^{k+1}] \|\varpi ; H^{k+3} (B(\mathbf 0,L)/(B(\mathbf 0,L/2)))\|^2\\[6pt]
&\leq & C[D^{k+1}] K^{reg}_{k+1} \left[ \|\vv{v}^* ; H^{k+2}(B(\mathbf 0,L))\|^2 + \|w^* ; H^{k+1}(B(\mathbf 0,L))\|^2\right]\,.
\end{array}
\end{multline*}
For the second term, we apply  \eqref{eq_CL} and a H\"older inequality: 
\begin{multline*}
\left| \int_{B(\mathbf 0,L)}
(k+m+1) \nabla \gamma \gamma^{3+m+k} \nabla D^{k+1} \varpi D^{k+1} \varpi\right| \\
\begin{array}{rcl} 
&\leq&
\displaystyle K[C_{ell},C^{reg}_2] \int_{B(\mathbf 0,L)}
 \gamma^{7/2+m+k} |\nabla D^{k+1} \varpi | | D^{k+1} \varpi|  \, \\ 
&\leq& \displaystyle \dfrac{1}{2} \int_{B(\mathbf 0,L)}
 \gamma^{3+m+(k+1)} |\nabla D^{k+1} \varpi |^2 + K[C_{ell},C^{reg}_2] \int_{B(\mathbf 0,L)} \gamma^{3+m+k} | D^{k+1} \varpi| ^2 \,.
\end{array}
\end{multline*}
To bound the last term, we note that $|D^{k+1} \varpi| \leq C[D^{k+1}]  |\nabla^{k+1} \varpi|.$
Consequently, applying the induction assumption, we obtain:
\begin{multline*}
|RHS_1| \leq C [D^{k+1}]  K[C_{ell},C_{2}^{reg},K_{m,k}^{sing},K^{reg}_{k+1}] \Big\{
\int_{B(\mathbf 0,L)} |\gamma|^{3+m} |\nabla \varpi|^2
+ \|\vv{v}^* ; H^{k+2}(B(\mathbf 0,L))\|^2 \\
+ \|w^* ; H^{e_{k+1}}(B(\mathbf 0,L))\|^2\Big\}
+ \displaystyle \dfrac{1}{2} \int_{B(\mathbf 0,L)}
 \gamma^{3+m+(k+1)} |\nabla D^{k+1} \varpi |^2 \,.
\end{multline*}

As for $RHS_2$, we distinguish between two cases.
If $k+m \geq 1$ we have that $2(m+k+1) \geq 3+m+k$ so that a standard H\"older inequality yields, with the induction assumption:
\begin{multline*}
\left| \int_{B(\mathbf 0,L)} (D^{k+1} w^* - \dfrac 12 {\rm div} D^{k+1} (\gamma_t+\gamma_b) \vv{v}^* ) \gamma^{k+m+1} D^{k+1} \varpi \right| \\
\begin{array}{cl}
\leq  & \displaystyle C[D^{k+1}]\left[ \|w^* ; H^{k+1}(B(\mathbf 0,L))\|^2 + C_{k+2}^{reg}\|\vv{v}^* ; H^{k+2}(B(\mathbf 0,L))\|^2 
+ \int_{B(\mathbf 0,L)} \gamma^{3+m+k} |\nabla^{k+1} \varpi|^2 \right]
\\[12pt]
\leq& \displaystyle C[D^{k+1}] K[K^{sing}_{m,k},C_{k+2}^{reg}]\Bigg[\int_{B(\mathbf 0,L)} \gamma^{3+m} |\nabla \varpi|^2 + \|\vv{v}^* ; H^{k+2}(B(\mathbf 0,L))\|^2 \\
& \phantom{13245qsdffdqfqqfqdfqfdq}\quad + \|w^* ; H^{e_k}(B(\mathbf 0,L))\|^2\Bigg]\,.
\end{array}
\end{multline*}
Whereas, if $0< (k+m) < 1$ so that $k=0$ in particular, we introduce arbitrarily the missing powers of
$\gamma$ and recall \eqref{eq_CL}-\eqref{eq_boundtbg}:
\begin{multline*}
\left| \int_{B(\mathbf 0,L)} (D^{k+1} w^* - \dfrac 12 {\rm div} D^{k+1} (\gamma_t+\gamma_b) \vv{v}^* ) \gamma^{k+m+1} D^{k+1} \varpi \right| \\
\begin{array}{cl}
\leq  & \displaystyle C
\left| \int_{B(\mathbf 0,L)} \left(\dfrac{|\nabla w^*|}{\gamma^{\frac 12-\frac{(k+m)}{2}}} +\dfrac{ K[C_{2}^{reg},C_{ell}]}{{\gamma^{\frac 12-\frac{(k+m)}{2}}}}\sum_{j=0}^2\gamma^{\frac j2}|\nabla^j \vv{v}^*| \right) \gamma^{\frac{(k+m)}{2}+\frac 32} D^{k+1} \varpi \right|
\end{array}
\end{multline*}
where $1 - (k+m) < 1$ so that:
$$
\int_{B(\mathbf 0,L)}\dfrac{|\nabla w^*|^2}{\gamma^{1-(k+m)}} \leq K[C_{ell}] \|w^* ; H^3(B(\mathbf 0,L))\|^2\,,
$$
and
$$
\int_{B(\mathbf 0,L)} \left|\dfrac 1{\gamma^{\frac 12-\frac{(k+m)}{2}}} \sum_{j=0}^2\gamma^{\frac j2}|\nabla^j \vv{v}^*| \right|^2 \leq K[C_{ell},C^{reg}_0] \|\vv{v}^* ; H^2(B(\mathbf 0,L))\|^2\,.
$$
Hence a H\"older inequality yields that:
\begin{multline*}
\left| \int_{B(\mathbf 0,L)} (D^{k+1} w^* - \dfrac 12 {\rm div} D^{k+1} (\gamma_t+\gamma_b) \vv{v}^* ) \gamma^{k+m+1} D^{k+1} \varpi \right| \\
\leq  \displaystyle  K[K^{sing}_{m,0},C_{2}^{reg},C_{ell}]\Bigg[\int_{B(\mathbf 0,L)} \gamma^{3+m} |\nabla \varpi|^2 + \|\vv{v}^* ; H^{k+2}(B(\mathbf 0,L))\|^2 + \|w^* ; H^{e_k}(B(\mathbf 0,L))\|^2\Bigg]\,.
\end{multline*}

To treat the last terms, we expand the differential operator $R^{k+1}$ and apply \eqref{eq_Ki} yielding that there exists also a constant $C[D^{k+1}]$ for which:
$$
|R^{k+1}| \leq C[D^{k+1}] {\sum_{l=0}^{k} |\nabla^{l+1} \varpi| |\nabla^{k+1-l} \gamma^{3}|  
		\leq  C[D^{k+1}] \sum_{l=0}^{k} K[k+1-l]  |\gamma|^{3 - \frac{k+1-l}{2} } |\nabla^{l+1} \varpi|}\,.
$$
Consequently, we bound similarly as above by applying trace theorems and Proposition \ref{prop_unifreg}:
\begin{multline*}
\left|\int_{\partial B(\mathbf 0,L)} R^{k+1} \cdot \nu \gamma^{k+m+1} D^{k+1} \varpi 
 \right| \leq C[D^{k+1}] K[C^{reg}_{k+1},C_{ell},C^{reg}_2]\|\varpi ; H^{k+1}(\partial B(\mathbf 0,L))\|^2 \, \\
 \begin{array}{cl}
  \leq & C[D^{k+1}]  K[C^{reg}_{k+1},C_{ell},C^{reg}_2] \|\varpi ; H^{k+2}(B(\mathbf 0,L)\setminus B(\mathbf 0,L/2))\|^2 \\[4pt]
  \leq & C[D^{k+1}] K[C^{reg}_{k+1},C_{ell},C^{reg}_2]\left[ \|\vv{v}^* ; H^{k+1}(B(\mathbf 0,L))\|^2 + \|w^* ; H^{e_k}(B(\mathbf 0,L))\|^2\right]\,.
\end{array}
\end{multline*}
For the other term, we recall \eqref{eq_CL} and apply the induction assumption.
This yields:
\begin{multline*}
\left| \int_{B(\mathbf 0,L)} R^{k+1} \cdot (\gamma^{k+m+1} \nabla D^{k+1} \varpi + (k+m+1) \nabla \gamma \gamma^{k+m} D^{k+1} \varpi) \right|\\
\begin{array}{cl}
\leq &\displaystyle 
 K[C^{reg}_{k+1},C_{ell},C^{reg}_2]{\sum_{l=0}^{k}  \int_{B(\mathbf 0,L)}   \gamma^{\frac {3+m+ l}2} |\nabla^{l+1} \varpi| \left( \gamma^{\frac {3+m+(k+1)}2}|\nabla D^{k+1} \varpi| +  \gamma^{\frac {3+m+k}2} |\nabla^{k+1} \varpi| \right)} \\
\leq & \displaystyle
\dfrac{1}{48} \int_{B(\mathbf 0,L)} \gamma^{3+m+(k+1)} |\nabla D^{k+1} \varpi|^2  \\
& \quad + K[C^{reg}_{k+1},C_{ell},C^{reg}_2] K^{sing}_{m,k}\Bigg[\displaystyle \int_{B(\mathbf 0,L)} |\gamma|^{3+m} |\nabla \varpi|^2+ \|\vv{v}^* ; H^{k+1}(B(\mathbf 0,L))\|^2 
\\
& \quad + \|w^* ; H^{e_k}(B(\mathbf 0,L))\|^2\Bigg]
\end{array}
\end{multline*}
We obtain finally that there exists a constant $K^{sing}_{m,k+1}$ depending on $K^{reg}_{k+1},K^{sing}_{m,k},K^{sing}_{0,k}$ and $C^{reg}_{k+2},C^{reg}_2,C_{ell}$ for which:
\begin{multline*}
\left|\dfrac{RHS_1}{12} + RHS_2\right| \\
 \leq C[D^{k+1}] K^{sing}_{m,k+1}\Bigg[ \displaystyle \int_{B(\mathbf 0,L)} |\gamma|^{3+m} |\nabla \varpi|^2 + \|\vv{v}^* ; H^{k+2}(B(\mathbf 0,L))\|^2 + \|w^* ; H^{e_{k+1}}(B(\mathbf 0,L))\|^2\Bigg]
\\
+ \dfrac{3}{48} \int_{B(\mathbf 0,L)} \gamma^{3+m+(k+1)} |\nabla D^{k+1} \varpi|^2\,.
\end{multline*}

We introduce this inequality into \eqref{eq_induction} and obtain that:
\begin{multline*}
\int_{B(\mathbf 0,L)} |\gamma|^{3+m+ (k +1)} |\nabla D^{k+1} \varpi|^2 \\
 \leq C[D^{k+1}]K^{sing}_{m,k+1}\Bigg[\displaystyle \int_{B(\mathbf 0,L)} |\gamma|^{3+m} |\nabla \varpi|^2 + \|\vv{v}^* ; H^{k+2}(B(\mathbf 0,L))\|^2 + \|w^* ; H^{e_{k+1}}(B(\mathbf 0,L))\|^2\Bigg]\,.
\end{multline*}
Given the dependencies of $K^{reg}_{k+1}$ and $K^{sing}_{m,k},K^{sing}_{m,0},$ the operator $D^{k+1}$ being arbitrary, we obtain the expected inequality for the rank $k+1.$ This ends the proof. 

\end{proof}

\subsection{The case of spheres.}
We end this section by computing sharp asymptotic expansions of  
$$
\int_{B(\mathbf 0,L)} \gamma^3 |\nabla \varpi|^2
$$ 
when $\gamma$ represents the distance function between one sphere of radius $S$ and another one of radius $R$. 
Until the end of this section, we assume that, close to the origin, we have:
$$
\gamma_t(x,y) = S - \sqrt{S^2 - (x^2+y^2)} \,, \quad  \gamma_b(x,y) = -R + \sqrt{R^2 - (x^2+y^2)} \,. 
$$ 
Hence, $\gamma$ depends on $r= \sqrt{x^2+y^2}$ only and we have the Taylor expansion:
\begin{equation} \label{eq_tagamma}
\left\{
\begin{array}{rcl}
\gamma_t(r) &= &\dfrac{r^2}{2S} + O(r^4) \,,  \,\\
\gamma_b(t) &= & - \dfrac{r^2}{2R} + O(r^4)\,,
\end{array}
\right.
\quad
\gamma(r) = h + \dfrac{r^2}{2R_1} + \dfrac{r^4}{8R_3^3} + O(r^6)\,, \,
\quad
\text{ close to $r=0,$}  
\end{equation}
where ${R_1}$ and ${R_3}$ satisfy \eqref{eq_R}.\\
Then, we split $f = f_0 + f_1 + f_{R},$ where:
$$
f_0 = w^*(\mathbf 0) \,, \quad f_1 = \dfrac{r}{2} \left(\dfrac{1}{R} - \dfrac{1}{S} \right) \vv{v}^*(\mathbf 0) \cdot \vv{e}_r     + x \partial_x w^*(\mathbf 0) + y \partial_y w^*(\mathbf 0)\,,
\quad 
f_R = f - (f_0 + f_1).
$$
Due to the linearity of the divergence problem \eqref{eq_papp1bis}-\eqref{eq_papp2bis}, the solution $\varpi$ admits  a corresponding decomposition:
$\varpi = \varpi_0+ \varpi_1 + \varpi_R$ with obvious notations. In this section, we compute
separately the asymptotic expansions of the quantities
$$
\int_{B(\mathbf 0,L)} \gamma^3 |\nabla \varpi_0|^2, \quad \int_{B(\mathbf 0,L)} \gamma^3 |\nabla \varpi_1|^2,  \quad
\int_{B(\mathbf 0,L)} \gamma^3 |\nabla \varpi_R|^2.
$$ 
First, as $f_R$ corresponds to $f$ from which the first orders in the Taylor expansion around $0$
are subtracted, we may introduce $(w^*_R,\vv{v}^*_R) \in C^{\infty}(\overline{B(\mathbf 0,L)};\mathbb R \times \mathbb R^2) $ s.t.:
$$
f_R = {w}^*_R - \dfrac{1}{2} {\rm div} ((\gamma_t + \gamma_b){\vv{v}}^*_R)\,, 
$$
We have then:
$$
{w}^*_R(\mathbf 0) =  0\,, \quad \nabla_{x,y} {w}^*_R(\mathbf 0) = \mathbf{0}\,, \quad {\vv{v}}_R^*(\mathbf 0) = \mathbf{0}\,.
$$
Hence, {Proposition \ref{prop_favorable}} entails that, for $h \in (0,1]:$ 
$$
\int_{B(\mathbf 0,L)} \gamma^3 |\nabla \varpi_R|^2 = O(1) \left\{\|\vv{u}^* ;  H^2(B(\mathbf 0,L))\|^2 + \|w^* ; H^3(B(\mathbf 0,L))\|^2 \right\}. 
$$ 
Here and in what follows, we denote by $O(1)$ a quantity which depends on $(R,S,L)$ and $h$ and is bounded by a constant depending only on $R,S,L$ whatever the value
of $h \in (0,1].$

\medskip

It remains to treat the two cases of $\varpi_0$ and $\varpi_1.$
We remark at once that, for $\varpi_0,$ the associated source term $f_0$ is constant, while, for $\varpi_1,$ the associated source term reads
$f_1 = f_1(r,\theta) = f_c r\cos(\theta) + f_s r \sin(\theta)$ where:
 $$
 f_c = \dfrac{1}{2} \left(\dfrac 1R- \dfrac 1S\right)\vv{v}^*(0) \cdot \vv{e}_1 + \partial_x w^*(0)
 \qquad
 f_s =  \dfrac{1}{2} \left(\dfrac 1R- \dfrac 1S\right)\vv{v}^*(0) \cdot \vv{e}_2 + \partial_y w^*(0).
 $$ 

\medskip

We start with $\varpi_0$:
\begin{proposition} \label{prop_casf0}
Under the assumption that $\gamma$ is radial and satisfies \eqref{eq_tagamma}, there holds:
\begin{equation} 
\int_{B(\mathbf 0,L)} \gamma^3 |\nabla \varpi_0|^2 = 72 \pi |f_0|^2 \left[ \dfrac{R_1^2}{h} - \dfrac{3 R_1^4}{R_3^3} |\ln(h)| \right] + O(1)\,. 
\end{equation}
\end{proposition}
\begin{proof}
By linearity, we only treat the case $f_0=1.$ Under our symmetry assumptions, the unique solution $\varpi_0$ to \eqref{eq_papp1bis}-\eqref{eq_papp2bis} with $f=1$ is certainly a radial function 
and explicit computations yield that:
$$
\varpi_0(r) =  \int_{r}^{L} \dfrac{6s}{\gamma^3(s)}{\rm d}s \,, \quad \partial_r \varpi_0(r) = - \dfrac{6r}{\gamma^3(r)}  \,, \quad\forall \, r \in (0,L)\,.
$$
Consequently, we have:
$$
\int_{B(\mathbf 0,L)} \gamma^3 |\nabla \varpi_0|^2 =  72\pi \int_0^L \dfrac{r^3{\rm d}r}{\gamma^3(r)} .
$$
At this point, we note that 
$$\gamma(r) = h + \dfrac{r^2}{2R_1} + \dfrac{r^4}{8R_3^3} + rem(r)
$$ 
with $|rem(r)| \leq C r^{6}$ for all $r \in (0,L)\,.$ Consequently, introducing $r_0$ small enough so that we might expand $\gamma$ in power series, we have:  
\begin{eqnarray*}
 \int_{0}^{L} \frac{r^3}{\gamma(r)^3}{\rm d}r\, &=& \int_{0}^{r_0} \frac{r^3}{\gamma(r)^3}{\rm d}r + O(1)\, \\
 &=&  \int_0^{r_0} \dfrac{r^3{\rm d}r}{(h + \frac{r^2}{2R_1} )^3}
 - \dfrac{3}{8R_3^3}\int_0^{r_0} \dfrac{r^7{\rm d}r}{(h + \frac{r^2}{2R_1} )^4}  + O(1) \, \\
 &=& \dfrac{1}{h} \int_0^{\frac {r_0}{\sqrt{h}}} \dfrac{s^3{\rm d}s}{(1 + \frac{s^2}{2R_1} )^3} - \dfrac{3}{8R_3^3}\int_0^{\frac {r_0}{\sqrt{h}}} \dfrac{s^7{\rm d}s}{(1 + \frac{s^2}{2R_1} )^4} + O(1),
 \end{eqnarray*}
 where explicit computations yield:
 $$
  \int_0^{\frac {r_0}{\sqrt{h}}} \dfrac{s^3{\rm d}s}{(1 + \frac{s^2}{2R_1} )^3}  = R_1^2 + h O(1)\,, \quad
  \int_0^{\frac {r_0}{\sqrt{h}}} \dfrac{s^7{\rm d}s}{(1 + \frac{s^2}{2R_1} )^4} = 8 R_1^4 |\ln(h)| + O(1). 
 $$
 Finally, we obtain:
$$
\int_{B(\mathbf 0,L)} \gamma^3 |\nabla \varpi_0|^2  = 72 \pi \left[ \dfrac{R_1^2}{h} -  \dfrac{3 R_1^4}{R_3^3} |\ln(h)| \right] + O(1)\,.
$$

\end{proof}

We end this section with the case of an $f_1$-like source term: 
\begin{proposition} \label{prop_casf1}
Assume that $\gamma$ is radial and satisfies \eqref{eq_tagamma}. Given a source term of the form $f(r,\theta) =  f_c r \cos(\theta) + f_s r \sin(\theta) $ on $B(\mathbf 0,L)$, there holds:
\begin{equation} 
\int_{B(\mathbf 0,L)} \gamma^3 |\nabla \varpi|^2 =  \left( |f_c|^2 + |f_s|^2 \right) \dfrac{288\pi R_1^3}{5}   |\ln(h)| + O(1) \,,  
\end{equation}
\end{proposition}
\begin{proof}
Up to shift $\theta$ with a phase and call the linearity of our problem, we  prove the above result in the
case $f_c =1$ and $f_s=0.$ Then, the proof is divided into 3 steps.

\medskip

\paragraph{\em Step 1: reduction to an ode.}
Under our symmetry assumptions, the unique solution to \eqref{eq_papp1bis}-\eqref{eq_papp2bis} reads $\varpi(r,\theta) = q_{h}(r) \cos(\theta)$ where $q_h$ is the unique solution to the ode:
\begin{eqnarray}
\dfrac{1}{r} \partial_r \left[ \gamma(r)^3r \partial_r q_h \right]  - \dfrac{\gamma^3(r) q_h(r)}{r^2} = 
- 12 r\,, \phantom{0}&& r \in (0,L) \,  \label{eq_edoqh}\\
q_h(r) =0\,, \phantom{f_cr} && r \in \{0,L\}\,. \label{eq_bcqh}
\end{eqnarray}
The boundary condition in $r=L$ is the translation of $\varpi(r,\theta) =0$ on $r=L$
while the boundary condition in $r=0$ is derived from the condition that:
$$
\pi\left[  \int_{0}^{L} \gamma(r)^3 \left[ |\partial_r q_h(r)|^2 + \dfrac{|q_h(r)|^2}{r^2}\right]r  {\rm d}r\right]
= \int_{B(\mathbf 0,L)} \gamma^3 |\nabla \varpi|^2 {\rm d}x {\rm d}y < \infty. 
$$
We note that \eqref{eq_edoqh} rewrites:
$$
\partial_{rr} q_h + \left(\dfrac{1}{r} + 3 \dfrac{\gamma'(r)}{\gamma(r)} \right) \partial_r q_h - \dfrac{q_h}{r^2} = - \dfrac{12r}{\gamma^3}\,.
$$

\medskip

\paragraph{\em Step 2: Construction of an approximate solution.}
We introduce then $q$ the unique solution to an auxiliary problem. 
The construction and properties of this auxiliary function are stated in the following lemma
whose proof is postponed to the appendix:
\begin{lemma} \label{lem_q}
Given $R_1 >0,$ there exists a unique $q \in C([0,\infty)) \cap C^{\infty}((0,\infty))$ solution to 
\begin{eqnarray*}
\partial_{ss} q +  \left(\dfrac{1}{s} +  \dfrac{3s}{R_1  (1+\frac{s^2}{2R_1})} \right) \partial_s q - \dfrac{q}{s^2} = 
-\dfrac{12 s}{(1+\frac {s^2}{2R_1})^3} \,, & & s \in (0,\infty) \,, \\
q(0) = 0 \quad \lim_{s \to \infty} q(s) = 0\,.  &&\,. 
\end{eqnarray*}
Furthermore we have the asymptotic description:
\begin{equation}
q(s) = \dfrac{48 R_1^3}{5 s^3} + rem(s) \quad
\partial_sq(s) =  - \dfrac{144 R_1^3}{5 s^4} + \dfrac{rem(s)}{s} \,.\label{eq_decayq}
\end{equation}
where $|rem(s)| \leq K[R_1]/s^4$ for $s > 1.$
\end{lemma}
With that auxiliary function at-hand, we set:
$$
\hat{q}_h(r) = \dfrac{1}{h^{\frac 32}} q \left( \dfrac{r}{\sqrt{h}}\right)  - \dfrac{r^2}{L^2h^{\frac 32}} q\left(\dfrac{L}{\sqrt{h}} \right)\,,  \forall\, r \in (0,L)\,, \quad
\hat{\omega} = \hat{q}_h(r) \cos(\theta) \,  \text{ on $B(\mathbf 0,L)$}.
$$
Introducing $\hat{\gamma}(r) = (h+r^2/(2R_1)),$ we obtain by substitution that $\hat{q}_h$
is a solution to 
\begin{eqnarray*}
\dfrac{1}{r} \partial_r \left[ \hat{\gamma}(r)^3r \partial_r \hat{q}_h \right]  - \dfrac{\hat{\gamma}^3(r) \hat{q}_h(r)}{r^2} = 
-12 r - \dfrac{1}{L^2h^{\frac 32}} q\left(\dfrac{L}{\sqrt{h}} \right) \hat{\chi}(r) \,, \phantom{0}&& r \in (0,L) \,  \label{eq_edoqh2}\\
\hat{q}_h(r) =0\,, \phantom{f_cr} && r \in \{0,L\}\,. \label{eq_bcqh2}
\end{eqnarray*}
where 
$$
\hat{\chi}(r) =\dfrac{2}{r} \partial_r \left[ \hat{\gamma}(r)^3r^2 \right] -  \hat{\gamma}^3(r).
$$
Consequently, $\hat{\omega}$ is an $H^1$-solution to:
\begin{eqnarray*}
- \dfrac{1}{12}{\rm div}[ \hat{\gamma}^3 \nabla \hat{\varpi} ]= \left( r + \dfrac{1}{12 L^2h^{\frac 32}} q\left(\dfrac{L}{\sqrt{h}} \right) \hat{\chi}(r) \right) \cos(\theta)\,, \phantom{0}&& \text{ on $B(\mathbf 0,L),$} \, \\
\hat{\varpi} = 0 \,  \phantom{12 r + \dfrac{1}{L^2} q\left(\dfrac{L}{\sqrt{h}} \right) \hat{\chi}(r) } &&\text{ on $r=L$}\,.
\end{eqnarray*}
We note that, for $n \geq 3,$ there holds:
$$
\hat{I}_n := \int_{B(\mathbf 0,L)} \hat{\gamma}^n |\nabla \hat{\varpi}|^2 {\rm d}x {\rm d}y = 
\pi\left[  \int_{0}^{L} \hat{\gamma}(r)^n \left[ |\partial_r \hat{q}_h(r)|^2 + \dfrac{|\hat{q}_h(r)|^2}{r^2}\right]r  {\rm d}r\right]
$$
Replacing $\hat{q}_h$ with its value and noticing that $\dfrac{q(L/\sqrt{h})}{L^2h^{\frac 32}} =O(1)$ because of \eqref{eq_decayq}, this entails:
\begin{eqnarray*}
\hat{I}_n &=& \dfrac{\pi}{h^3} \left[  \int_{0}^{L} \left(h+ \dfrac{r^2}{2R_1}\right)^n \left[ \dfrac{|\partial_r q({r}/{\sqrt{h}})|^2}{h} + \dfrac{|{q}(r/\sqrt{h})|^2}{r^2}\right]r  {\rm d}r\right] \\
&&+ O\left( \left[1 +  \int_{0}^{L} \left(h+ \dfrac{r^2}{2R_1}\right)^{2n} \left[ \dfrac{|\partial_r q({r}/{\sqrt{h}})|^2}{h} + \dfrac{|{q}(r/\sqrt{h})|^2}{r^2}\right]r  {\rm d}r\right]\right)\,.
\end{eqnarray*}
We bound the first integral on the right-hand side by changing variable $r \to \sqrt{h}s.$ This yields:
\begin{multline*}
 \int_{0}^{L} \left(h+ \dfrac{r^2}{2R_1}\right)^n \left[ \dfrac{|\partial_r q({r}/{\sqrt{h}})|^2}{h} + \dfrac{|{q}(r/\sqrt{h})|^2}{r^2}\right]r  {\rm d}r\\
 = h^{n} \int_0^{L/\sqrt{h}} \left(1+ \dfrac{s^{2}}{2R_1}\right)^n  \left[ |\partial_s q(s)|^2 + \dfrac{|{q}(s)|^2}{s^2}\right] s {\rm d}s\,.
\end{multline*}
Given the asymptotic expansion of $q$ this entails:
\begin{multline*}
 \int_{0}^{L} \left(h+ \dfrac{r^2}{2R_1}\right)^n \left[ \dfrac{|\partial_r q({r}/{\sqrt{h}})|^2 }{h}+ \dfrac{|{q}(r/\sqrt{h})|^2}{(r/\sqrt{h})^2}\right]r  {\rm d}r 
 =
 \left\{
 \begin{array}{ll}
  \dfrac{h^{3}R_1^3}{40} \left(  48^2 |\ln(h)| + O(1)\right) & \text{ if $n=3$}\\[8pt]
 h^{3}O(1) & \text{ if $n \geq 4$}.
 \end{array}
 \right.
 \end{multline*}
  In particular we obtain finally that:
\begin{equation} \label{eq_asymptotpichapeau}
\int_{B(\mathbf 0,L)} \hat{\gamma}^{n} |\nabla \hat{\varpi}|^2  = 
\left\{
\begin{array}{ll}
 \dfrac{288R_1^3\pi}{5} |\ln(h)| + O(1)   & \text{ for $n=3$}\\[8pt]
O(1) & \text{ for arbitrary $n \geq 4.$}
\end{array}
\right.
\end{equation}

\paragraph{\em Step 3: Computing the distance between $\varpi$ and $\hat{\varpi}$}
 We introduce the difference $D[\varpi] := \varpi - \hat{\varpi}.$ It is the solution to: 
 \begin{eqnarray}
- \dfrac{1}{12}{\rm div}[ {\gamma}^3 \nabla D[\varpi] ]= -\dfrac{1}{12} {\rm div} \left[ {(\hat{\gamma}^3 -{\gamma}^3)} \nabla \hat{\varpi}  \right] - \dfrac{1}{12 L^2h^{\frac 32}} q\left(\dfrac{L}{\sqrt{h}} \right) \hat{\chi}(r){\cos(\theta)} \,, \phantom{0} \!\!\!\!\!\!\!\!\!\!\!\!\!\!\!\!&& \text{ on $B(\mathbf 0,L)$} \,  \label{eq_Dvarpi}\\
D[\varpi]= 0 \,  \phantom{ \dfrac{1}{12} {\rm div} \left[ (\gamma^3 - \hat{\gamma}^3) \nabla \hat{\varpi}  \right] +  \dfrac{1}{12 L^2} q\left(\dfrac{L}{\sqrt{h}} \right) \hat{\chi}(r) \,,} \!\!\!\!\!\!\!\!\!\!\!\!\!\!\!\!&&\text{ on $\partial B(\mathbf 0,L)$}\,.
\end{eqnarray}
Multypling \eqref{eq_Dvarpi} with $D[\varpi]/\gamma,$ we obtain, after integration by parts:
\begin{equation} \label{eq_estimDvarpi}
\dfrac{1}{12}\left[ \int_{B(\mathbf 0,L)}  \gamma^2 |\nabla D[\varpi]|^2 - \dfrac 14 \int_{B(\mathbf 0,L)} \nabla |\gamma|^2 \nabla |D[\varpi]|^2  \right] = RHS.
\end{equation}
On the left-hand side of this identity, we have:
\begin{multline*}
\dfrac{1}{12}\left[ \int_{B(\mathbf 0,L)}  \gamma^2 |\nabla D[\varpi]|^2 - \dfrac 14 \int_{B(\mathbf 0,L)} \nabla |\gamma|^2 \nabla |D[\varpi]|^2  \right]   \\
\begin{array}{rcl}
&=&\displaystyle \dfrac{1}{12}\left[ \int_{B(\mathbf 0,L)}  \gamma^2 |\nabla D[\varpi]|^2 + \dfrac 14 \int_{B(\mathbf 0,L)} [\Delta |\gamma|^2]  |D[\varpi]|^2  \right]  \\
	&\geq& \displaystyle \dfrac{1}{12} 	\left[ \int_{B(\mathbf 0,L)}  \gamma^2 |\nabla D[\varpi]|^2 + \dfrac c2 \int_{B(\mathbf 0,L)}  \gamma  |D[\varpi]|^2  \right].
\end{array}
\end{multline*}
as $\Delta |\gamma|^2 = 2 [\gamma \Delta \gamma  +  |\nabla \gamma|^2]   \geq 2 c \gamma$
on $B(\mathbf 0,L)$. On the right-hand side, we obtain the bound above:
$$
|RHS| \leq \dfrac{I_1}{12} + \dfrac{q(L/\sqrt{h})}{12L^2h^{\frac 32}} I_2
$$
where, after integration by parts:
\begin{eqnarray*}
I_1 &=& \left| \int_{B(\mathbf 0,L)} {\rm div} \left[ (\gamma^3 - \hat{\gamma}^3) \nabla \hat{\varpi}  \right]  \dfrac{D[{\varpi}]}{\gamma}\right| \\
 &\leq & \int_{B(\mathbf 0,L)} \dfrac{ |\gamma^3 - \hat{\gamma}^3|}{\gamma^2} |\nabla \gamma| |\nabla \hat{\varpi} | |D[\varpi]| +  \int_{B(\mathbf 0,L)} \dfrac{ |\gamma^3 - \hat{\gamma}^3|}{\gamma}  |\nabla \hat{\varpi} | |\nabla D[\varpi]|   \\
       &\leq & \dfrac{C}{\varepsilon} \int_{B(\mathbf 0,L)}   \dfrac{ (\gamma^3 - \hat{\gamma}^3)^2}{\gamma^2} \left( \dfrac{|\nabla \gamma|^2}{|\gamma|^3} + \dfrac{1}{\gamma^2}  \right) |\nabla \hat{\varpi}|^{2} + \varepsilon   \int_{B(\mathbf 0,L)} \left( \gamma^2 |\nabla D[\varpi]|^2 + \gamma |D[\varpi]|^2\right).
\end{eqnarray*}
for arbitrary $\varepsilon >0.$ At this point, we remark that, as $\gamma$ and $\hat{\gamma}$ share the same (non-vanishing) Taylor expansion up to the second order, we have, with a constant $K$
depending only on $R,S:$  
$$
|\gamma^3 - \hat{\gamma}^3 | \leq K   \hat{\gamma}^{4}.
$$
Recalling \eqref{eq_CL} to bound $\nabla \gamma$ and \eqref{eq_asymptotpichapeau} we obtain:
$$
\int_{B(\mathbf 0,L)}   \dfrac{ (\gamma^3 - \hat{\gamma}^3)^2}{\gamma^2}\left( \dfrac{|\nabla \gamma|^2}{|\gamma|^3} + \dfrac{1}{\gamma^2}  \right)   |\nabla \hat{\varpi}|^{2} \leq K\int_{B(\mathbf 0,L)} \hat{\gamma}^{4}  |\nabla \hat{\varpi}|^{2} = O(1),
$$
and finally, for arbitrary positive $\varepsilon,$ we have:
$$
I_1 \leq \dfrac{O(1)}{\varepsilon} + \varepsilon   \int_{B(\mathbf 0,L)} \left( \gamma^2 |\nabla D[\varpi]|^2+  \gamma |D[\varpi]|^2 \right).
$$
Similarly, we bound $I_2:$
\begin{eqnarray*}
I_2  &=& \left| \int_{B(\mathbf 0,L)} \hat{\chi}(r) \cos(\theta)  \dfrac{D[\varpi]}{\gamma}\right|\\ 
&\leq &   \dfrac{C}{\varepsilon} \int_{0}^{L}  \left|\dfrac{1}{r} \partial_r \left[ \hat{\gamma}(r)^3r^2 \right] -  \hat{\gamma}^3(r)\right|^2\dfrac{r{\rm d}r}{\gamma^3}   + \varepsilon \int_{B(\mathbf 0,L)} \gamma |D[\varpi]|^2, 
\end{eqnarray*}
where there is a constant $K$ depending only on the characteristics of $\gamma$ for which
$$
\left|\dfrac{1}{r} \partial_r \left[ \hat{\gamma}(r)^3r^2 \right] -  \hat{\gamma}^3(r)\right| \leq 
K \gamma^3
$$
Hence, we have:
$$
\int_{0}^{L}  \left|\dfrac{1}{r} \partial_r \left[ \hat{\gamma}(r)^3r^2 \right] -  \hat{\gamma}^3(r)\right|^2\dfrac{r{\rm d}r}{\gamma^3}  = O(1)
$$
and finally, for arbitrary $\varepsilon >0,$ there holds:
$$
I_2 \leq  \varepsilon \int_{B(\mathbf 0,L)} \gamma |D[\varpi]|^2 + \dfrac{O(1)}{\varepsilon}.
$$
Introducing the above computations of the left-hand side and right-hand side into \eqref{eq_estimDvarpi}, noticing that 
$$
\dfrac{1}{12 L^2h^{\frac 32}} q\left(\dfrac{L}{\sqrt{h}} \right)  = O(1)
$$
thanks to \eqref{eq_decayq}, we obtain, choosing $\varepsilon$
 sufficiently small:
 \begin{equation} \label{eq_estimDvarpifin}
 \int_{B(\mathbf 0,L)}  \gamma^2 |\nabla D[\varpi]|^2 +  \int_{B(\mathbf 0,L)}  \gamma  |D[\varpi]|^2 = O(1).
 \end{equation}

\paragraph{\em Step 4: Conclusion.} We are now in position to compute the asymptotics of 
$$
I_h := \int_{B(\mathbf 0,L)} \gamma^3 |\nabla \varpi|^2 .
$$
Indeed, there holds:
\begin{eqnarray*}
I_h &=& \int_{B(\mathbf 0,L)} \hat{\gamma}^3 |\nabla \hat{\varpi}|^2 +  \int_{B(\mathbf 0,L)} [\gamma^3 -\hat{\gamma}^3] |\nabla \hat{\varpi}|^2 \\
&& + \int_{B(\mathbf 0,L)} \gamma^3 |\nabla D[\varpi]|^2 + 2 \int_{B(\mathbf 0,L)} \gamma^3 \nabla D[\varpi] : \nabla \varpi\\
&=&  \int_{B(\mathbf 0,L)} \hat{\gamma}^3 |\nabla \hat{\varpi}|^2+ J_1 + J_2 + J_3 . 
\end{eqnarray*}
As previously, we bound $|\gamma^3 - \hat{\gamma}^3| \leq K \hat{\gamma}^4 $ so that
recalling \eqref{eq_asymptotpichapeau} with $n=4$ we have:
$$
|J_1| \leq  K \int_{B(\mathbf 0,L)}\hat{\gamma}^4  |\nabla \hat{\varpi}|^2 = O(1).
$$
Then, \eqref{eq_estimDvarpifin} implies:
$$
|J_2| \leq K \int_{B(\mathbf 0,L)} \gamma^2 |\nabla D[\varpi]|^2   = O(1).
$$
We conclude with similar arguments and applying {Proposition \ref{prop_reg2}} in case $n=1$ (here $\varpi$
is a solution to \eqref{eq_papp1bis}-\eqref{eq_papp2bis} with a special $\vv{v}^*$ and  $w^*=0$):
\begin{eqnarray*}
|J_3 | &= & \left| \int_{B(\mathbf 0,L)} \gamma^3 \nabla D[\varpi] : \nabla \varpi\right| \\
 	&\leq &  \int_{B(\mathbf 0,L)} \gamma^2 |\nabla D[\varpi]|^2 +  \int_{B(\mathbf 0,L)} \gamma^4 |\nabla \varpi|^2  =  O(1).
\end{eqnarray*}
Hence, there holds:
\begin{eqnarray*}
I_h & =& \int_{B(\mathbf 0,L)} \hat{\gamma}^3 |\nabla \hat{\varpi}|^2 + O(1) = \dfrac{288 \pi R_1^3}{5} |\ln(h)| + O(1)\,,
\end{eqnarray*}
because of \eqref{eq_asymptotpichapeau}. This ends the proof of Proposition \ref{prop_casf1}.
\end{proof}

\section{Proof of {\bf Theorem \ref{thm_mainexact}}} \label{sec_Takfarinas}

In this section, we complete the proof of {Theorem \ref{thm_mainexact}}.  We first introduce some notations:
\begin{itemize}
\item $\chi_L$ is a truncation function which vanishes outside $B(\mathbf 0,L).$ Namely, we set
$$
\chi_L(x,y) = \chi \left( \dfrac{|(x,y)|}{L}\right)\,,  \quad \forall \, (x,y)  \in \mathbb R^2\,,
$$
with $\chi \in C^{\infty}_c(\mathbb R ; [0,1])$ satisfying 
$$
\chi(t) = 1 \,, \; \forall \, t \in [-1/2,1/2] \,, \quad  \quad
\chi(t) = 0 \,, \; \forall \, t \in \mathbb R \setminus [-1,1] \,. 
$$
\item we recall, that given $\ell >0,$ we denote:
$$
\mathcal{G}_{\ell} = \{(x,y,z) \in \mathbb R^3\, \text{ s.t .}\, |(x,y)| < \ell \text{ and } z \in (\gamma_b(x,y),h+\gamma_t(x,y))\}\,,
$$
\item $\Omega^*$ is a smooth subdomain of $\Omega \setminus \mathcal G_{L/4}$ which does not depend on $h$ and satisfies:
$$
\Omega^* \subset \mathcal F \text{ whatever $h \in (0,1]$}\,,  \qquad \partial \Omega \setminus \partial \mathcal G_{L/2} \subset  \partial \Omega^*.
$$
We note that such an $\Omega^*$ exists as, thanks to assumptions (A2) and (A3), there exists 
$\tilde{\delta} > 0$ depending on $\delta$ and $C_{ell}$ such that, whatever the value of $h \in (0,1]$ there holds 
$$
\{ (x,y,z) \in  \Omega \, \text{ s.t. } {\rm dist}((x,y,z),\partial \Omega) < \tilde{\delta}\} \setminus \mathcal G_{L/4}  \subset \mathcal F\,.
$$ 
\end{itemize}

\medskip

\subsection{Construction of asymptotic approximation} \label{sec_asvf}
Let fix one boundary condition $\vv{u}^* \in  C^{\infty}(\partial \Omega)$ satisfying \eqref{eq_noflux}. At this step,  we explain how to construct a velocity-field $\mathbf v[\mathbf{u}^*]$ which shall approximate the solution to the Stokes problem with boundary conditions 
$\vv{u}^* \in  C^{\infty}(\partial \Omega)$ in the regime $h << 1.$

\medskip

Let split $\vv{u}^*$ into its tangential and normal parts (according to the tangent space in the origin):
$$
\vv{u}^* = \vv{u}_{\sslash}^{*} + u^*_{\bot} \vv{e}_z \in C^{\infty}(\partial \Omega).
$$ 
We introduce ${q}_{in}$ the unique solution to 
\begin{equation} \label{eq_pourq}
\begin{array}{rcll}
-\dfrac{1}{12} {\rm div}(\gamma^3 \nabla q_{in}) &=&  u^*_{\bot} - \dfrac{1}{2}{\rm div}  (\gamma_t + \gamma_b )  \vv{u}^*_{\sslash} - \dfrac{1}{2}(h-2\gamma_b){\rm div} \vv{u}_{\sslash}^{*}  \,, & \text{ on $B(\mathbf 0,L)\,,$} \\
q_{in} &=& 0 \,, & \text{ on $\partial B(\mathbf 0,L)$}\,. 
\end{array}
\end{equation}
We recall that the data $\gamma_t,\gamma_b$  and $\vv{u}^*$ are smooth so that we have existence of a unique $q_{in}$ solution to this equation
satisfying:
\begin{equation}
q_{in} \in C^{\infty}(B(\mathbf 0,L)). \label{eq_regtp}
\end{equation}
Second, we construct an auxiliary pair velocity-field/pressure in the aperture domain. 
We denote:
\begin{equation} \label{eq_ptilde}
{p}_{in}(x,y) =  \chi_L(x,y) q_{in} (x,y) \,, \quad \forall \, (x,y,z) \in \mathcal G_{L}\,.
\end{equation}
and
\begin{eqnarray} \label{def_tvpar}
{\vv{v}}_{in,\, \sslash\!}({x},{y},{z})  &=&  \dfrac{1}{2} ( z - (h+ {\gamma}_t))( z - {\gamma}_b) \nabla_{x,y\,}  p_{in}  + \left(\dfrac{h+ \gamma_t - z}{\gamma } \right) \chi_L  \vv{u}^*_{\,\sslash} \,, \\[4pt] \label{def_tvort0}
{{v}}_{in,\bot}({x},{y},{z}) &=&  \dfrac 12 {\rm div}_{x,y}  \left[ \int^{h+\gamma_t}_{z} (s-(h+\gamma_t))(s-\gamma_b)\nabla_{x,y\,}  p_{in} \, {\rm d}s\right]   \\
&& + \,  \int^{h+\gamma_t}_{z} {\rm div}_{x,y} \left[ \left(\dfrac{h+\gamma_t-s}{\gamma } \right)\chi_L\vv{u}^*_{\,\sslash}  \right] {\rm d}s\,,   \notag
\end{eqnarray}
for $(x,y,z) \in \mathcal G_{L}.$ We note that this velocity-field vanishes outside $\mathcal G_{L}$. We keep notations to denote its trivial extension to $\mathcal F$ in what follows.

\medskip

Combining \eqref{eq_regtp} with \eqref{eq_pourq}-\eqref{def_tvpar}-\eqref{def_tvort0}, we obtain that:
\begin{itemize} 
\item[(VR1)] $\vv{v}_{in} \in C^{\infty}(\overline{\mathcal F})$ and has support in $\mathcal G_{L},$\\[-10pt]
\item[(VR2)] ${\rm div} \ \vv{v}_{in} = 0$ on $\mathcal F,$\\[-10pt]
\item[(VR3)] $\vv{v}_{in}  =  0  \text{ on $\partial \mathcal B$}\,, $\\[-10pt]
\item[(VR4)] $\vv{v}_{in} = \vv{u}^* \text{ on $\partial \Omega \cap \partial \mathcal G_{L/2}$  }  $\\[-10pt]\end{itemize}

All properties satisfied by $\vv{v}_{in}$ are obvious but :
$$
v_{in,\bot}(x,y,\gamma_b(x,y)) = u_{\bot}^*(x,y)\,, \quad \forall \, (x,y) \in B(\mathbf 0,L/2).
$$
This property is a consequence of $q_{in}$ solution to \eqref{eq_pourq}. 
We warn the reader that, to check this implication, one cannot commute the divergence and integral operators prior to taking the trace on $z =\gamma_b(x,y).$
Indeed, the differential operator ${\rm div}_{x,y}$ does not include only derivatives parallel to the boundary $z=\gamma_b(x,y)$ so that one
has to expand the differential operator ${\rm div}_{x,y}$ inside the integral to perform this computation.

\medskip

The internal velocity-field $\vv{v}_{in}$ does not match the boundary condition $\vv{u}^*$ outside $\mathcal G_{L/2}$
(unless $\vv{u}^*$ vanishes). As we expect no singular behavior of the exact solution to the Stokes problem far from the singularity, we extend the approximation in a simple way outside $\mathcal G_{L/2}$. 
Precisely, we define  
$$
\vv{v}_{ext}^*= \vv{u}^* - \vv{v}_{in}{|_{\partial \Omega}} \quad \text{ on $\partial \Omega.$}
$$
As $\vv{v}_{in} \in C^{\infty}(\overline{\mathcal F})$ and ${\vv{u}}^*$ is smooth, we have
that $\vv{v}_{ext}^* \in C^{\infty}(\partial \Omega).$ Moreover,
 the properties of the boundary values of $\vv{v}_{in}$ that we mentioned above ensure that $\vv{v}_{ext}^*$ vanishes on $\partial \Omega \cap \partial \mathcal G_{L/2}.$ Consequently, we might extend $\vv{v}_{ext}^*$ by $0$ to $\partial  \Omega^*,$ defining in this way a smooth function. Then, we denote $(\vv{v}_{ext},q_{ext})$ the unique solution to:
\begin{eqnarray}
\label{eq_stokesext1}\Delta \vv{v}_{ext} - \nabla q_{ext} &=& 0\,, \phantom{\vv{v}^*_{ext}}  \text{ on $\Omega^*$}\,,\\
\label{eq_stokesext2} \nabla \cdot \vv{v}_{ext} &=&0\,,\phantom{\vv{v}^*_{ext}} \text{  on $\Omega^*$}\,, \\
\label{eq_stokesext3} \vv{v}_{ext} &= &\vv{v}^*_{ext} \,, \phantom{0} \text{ on $\partial \Omega^*$}\,.
\end{eqnarray}
This solution indeed exists as:
\begin{eqnarray*}
\int_{\partial \Omega^*} \vv{v}_{ext}^* \cdot \vv{n}^* {\rm d}\sigma& = &\int_{\partial \Omega} \vv{v}_{ext}^* \cdot \vv{n} {\rm d}\sigma \, \\
&= & \int_{\partial \Omega} \vv{u}^*  \cdot \vv{n}{\rm d}\sigma - \int_{\partial \Omega}{\vv{v}}_{in}   \cdot \vv{n}{\rm d}\sigma \, \\
&=& - \int_{\partial \mathcal F} {{\vv{v}}}_{in} \cdot \vv{n} {\rm d} \sigma  =0\,.
\end{eqnarray*}
As $\partial \Omega^*$ and $\vv{v}_{ext}^*$ are smooth we also have that $\vv{v}_{ext}$ 
is smooth on $\overline{\Omega^*}$ and vanishes on $\partial \Omega^* \cap \partial [\mathcal F \setminus \Omega^*].$ Hence, the trivial extension of $\vv{v}_{ext}$   to $\mathcal F$
(that we still denote $\vv{v}_{ext}$ for simplicity) satisfies:
\begin{itemize} 
\item[(VR5)] ${\vv{v}}_{ext}$ is continuous piecewise smooth and has support in $\Omega^*$
\item[(VR6)] ${\rm div} \ {\vv{v}}_{ext} = 0$ on $\mathcal F$
\item[(VR7)] ${\vv{v}}_{ext}  =  0  \text{ on $\partial \mathcal B$}\,,$ 
\item[(VR8)] ${\vv{v}}_{ext} = \vv{u}^* - {\vv{v}}_{in} \text{ on $\partial \Omega$} $ 
\end{itemize}
Our asymptotic approximation $\vv{v}[\vv{u}^*]$ reads then 
\begin{equation} 
\vv{v}[\vv{u}^*] = {\vv{v}}_{in} + \vv{v}_{ext}.
\end{equation}

\medskip

Even though the vector-field $\vv{v}_{ext}$ is important in order to obtain a velocity-field $\vv{v}[\vv{u}^*]$
that matches the boundary conditions $\vv{u}^*$ on the whole $\partial \Omega,$  this external velocity-field does not contain any information on the way the velocity-field $\vv{v}[\vv{u}^*]$ diverges when $h$ is small. Indeed, we have the following lemma:
\begin{lemma} \label{lem_vext}
There exists a constant $K$ depending on $C^{reg}_2,C_{ell},\partial \Omega$ and $\delta$ (introduced in (A3)) such that:
$$
\int_{\mathcal F} |\nabla \vv{v}_{ext}|^2 \leq K\left[  \|\vv{u}^* ; H^2(B(\mathbf 0 ,L))\|+ \|\vv{u}^* ; H^\frac 12(\partial \Omega)\| \right]^2\, .
$$
\end{lemma}
\begin{proof}
Indeed, as solution to the Stokes problem on $\Omega^*,$ the exterior velocity-field $\vv{v}_{ext}$ satisfies
\begin{eqnarray*}
\int_{\mathcal F} |\nabla \vv{v}_{ext}|^2 &=&\int_{\Omega^*} |\nabla \vv{v}_{ext}|^2  \\
&\leq & C^{*} \| \vv{v}_{ext}^* ; H^{\frac 12}(\partial \Omega^*)\|^2.
\end{eqnarray*}
where the constant $C^*$ depends only on $\Omega^*$ which depends itself on the geometry of $\partial \Omega$ away from the singularity and $\delta$. Moreover, recalling that $\vv{v}_{ext}^*$ vanishes in $\mathcal G_{L/2},$ we might construct a function $\zeta \in C^{\infty}(\overline{\Omega^*})$ such that $\zeta=1$ on $Supp(\vv{v}_{ext}^*)$ and $\zeta= 0$
on $\partial \Omega^* \setminus \partial \Omega.$ We have then, introducing $\Gamma[\vv{u}^*] \in H^1(\Omega)$ a lifting of $\vv{u}^*$ on the whole $\Omega$: 
\begin{eqnarray*}
\| \vv{v}_{ext}^* ; H^{\frac 12}(\partial \Omega^*)\| &\leq& \| \zeta \vv{u}^* ; H^{\frac 12}(\partial \Omega^*)\| +  \| \zeta {\vv{v}}_{in} ; H^{\frac 12}(\partial \Omega^*) \| \, \\
&\leq& 
C^* \| \zeta \Gamma[\vv{u}^*] ; H^{1}(\Omega^*) \| +  \|\zeta {\vv{v}} _{in}; H^{1}(\Omega^*)\|\, \\
&\leq & 
C^*_{\zeta} \left( \| \vv{u}^* ; H^{\frac 12}(\partial \Omega) \| +  \|{\vv{v}}_{in} ; H^1(\Omega^*)\|\right)\,. 
\end{eqnarray*}
As $\Omega^*$ remains away from the singularity, we might bound from below $\gamma$ by a constant $\hat{\delta}$ depending only on $C_{ell}$ on $\mathcal{G}_{L} \cap \Omega^*$ and 
apply the explicit formulas for ${\vv{v}}_{in}$ to obtain the following bounds:
\begin{multline*}
\|{\vv{v}}_{in} ; H^1(\Omega^*)\|  
\leq C^* \left( 1 + \dfrac{1}{\hat{\delta}^3}\right) \left( \|\gamma_b ; C^2(\overline{B(\mathbf 0,L)}) \| + \|\gamma_t ; C^2(\overline{B(\mathbf 0,L)})\|\right) \times \dots \\
 \ldots \times \left( \|{q}_{in} ; H^3(B(\mathbf 0,L) \setminus B(\mathbf 0,L/2))\| + \|\vv{u}^* ; H^2(B(\mathbf 0,L))\|\right). 
\end{multline*}
{We now wish to apply {Proposition \ref{prop_unifreg}} to ${q}_{in}$ in order to yield a constant that bounds its $H^3$-norm by something which does not depend on $h.$ However, one term in the right-hand side of  \eqref{eq_pourq} depends on $h.$ To avoid this difficulty we expand 
$q_{in} = q_{in}^{(1)}+ hq_{in}^{(2)}$ with pressure field components defined  as respective solutions to:
\begin{equation} \label{eq_pourq1}
\begin{array}{rcll}
-\dfrac{1}{12} {\rm div}(\gamma^3 \nabla q^{(1)}_{in}) &=&  (u^*_{\bot} + \gamma_b{\rm div} \vv{u}_{\sslash}^{*} ) - \dfrac{1}{2}{\rm div}  (\gamma_t + \gamma_b )  \vv{u}^*_{\sslash}   \,, & \text{ on $B(\mathbf 0,L)\,,$} \\
q^{(1)}_{in} &=& 0 \,, & \text{ on $\partial B(\mathbf 0,L)$}\,. 
\end{array}
\end{equation}
and
\begin{equation} \label{eq_pourq2}
\begin{array}{rcll}
-\dfrac{1}{12} {\rm div}(\gamma^3 \nabla q^{(2)}_{in}) &=& -\dfrac 12 {\rm div} \vv{u}_{\sslash}^{*}  \,, & \text{ on $B(\mathbf 0,L)\,,$} \\
q^{(2)}_{in} &=& 0 \,, & \text{ on $\partial B(\mathbf 0,L)$}\,. 
\end{array}
\end{equation}
In particular, we remark that $q^{(1)}_{in}$ is a solution to \eqref{eq_papp1bis}-\eqref{eq_papp2bis} with a source term $f$ given by \eqref{eq_formf} associated with
$$
w_{\bot}^{*} = u_{\bot}^{*} + \gamma_b {\rm div} \vv{u}_{\sslash}^*  \qquad \vv{v}^* = \vv{u}_{\sslash}^*\,,
$$
and that $q^{(2)}_{in}$is a solution to \eqref{eq_papp1bis}-\eqref{eq_papp2bis} with a source term $f$ given by \eqref{eq_formf} associated with
$$
w_{\bot}^{*} = -\dfrac 12 {\rm div} \vv{u}_{\sslash}^*  \qquad \vv{v}^* = 0.
$$
Hence, we apply {Proposition \ref{prop_unifreg}}  to both  $q^{(1)}_{in}$ and $q^{(2)}_{in}$ and obtain:
$$
 \|{q}^{(1)}_{in} ; H^3(B(\mathbf 0,L) \setminus B(\mathbf 0,L/2))\| + h \|{q}^{(2)}_{in} ; H^3(B(\mathbf 0,L) \setminus B(\mathbf 0,L/2))\| \leq K^{1}_{reg}  \| \vv{u}^* ; H^2(B(\mathbf 0,L))\|. 
$$
This entails finally:
\begin{multline} \label{eq_restevstar}
\|{\vv{v}}_{in} ; H^1(\Omega^*)\|  
\leq C^* \left( 1 + \dfrac{1}{\hat{\delta}^3}\right) \left( \|\gamma_b ; C^2(\overline{B(\mathbf 0,L)}) \| + \|\gamma_t ; C^2(\overline{B(\mathbf 0,L)})\|\right) \times \dots \\
 \ldots \times  K^{reg}_{1}  \| \vv{u}^* ; H^2(B(\mathbf 0,L))\| .
\end{multline}
}
\end{proof}

\subsection{Main steps of the proof of Theorem \ref{thm_mainexact}}
We fix now $\vv{u}^* \in C^{\infty}(\partial \Omega)$ satisfying \eqref{eq_noflux} as in the assumptions of Theorem \ref{thm_mainexact}.
We split this boundary condition in 
$$
\vv{u}^* = \vv{u}^*_{as} + \vv{u}^*_R
$$
with:
\begin{eqnarray*}
\vv{u}^*_{as} &=& u^*_{\bot}(\mathbf 0)\vv{e}_z +  \vv{u}^*_{\sslash}(\mathbf 0) +( x \partial_x u^*_{\bot}(\mathbf 0) + y \partial_y u^*_{\bot}(\mathbf 0)) \vv{e}_z \,,\\
\vv{u}^*_R &=& \vv{u}^* - \vv{u}^*_{as} \,.
\end{eqnarray*}
Note that, straightforward computations entail:
$$
\int_{\partial \Omega} \vv{u}^*_i \cdot \vv{n} {\rm d}\sigma  =0, \qquad \vv{u}^*_i \in C^{\infty}(\partial \Omega),  \quad \forall \, i \in \{as,R\}\,.
$$
So, we might define the solutions of Stokes system \stokes\  as given by {Theorem \ref{thm_galdi}} for boundary 
data $\vv{u}^*,\vv{u}^*_{as},\vv{u}^*_R.$ We denote these solutions $(\vv{u},p)$, $(\vv{u}_{as},p_{as})$, $(\vv{u}_R,p_R)$
respectively. We also introduce the asymptotic approximations:
$$
\vv{v}_{as} = \vv{v}[\vv{u}^*_{as}] \quad  \vv{v}_{R} = \vv{v}[\vv{u}^*_{R}]. 
$$

Let denote $\vv{v} = \vv{v}_{as}.$ Then, as $\vv{u}^{*}_{as,\,\sslash}$ is constant,
we have that 
$$
(h-2\gamma_b) {\rm div} {\vv{u}^*_{as,\, \sslash}} = 0
$$
so that the pressure associated with $\vv{v}_{as}$ satisfies \eqref{eq_q}-\eqref{cab_q} and, consequently, 
$\vv{v}_{as}$ satisfies \eqref{eq_vx}-\eqref{eq_vy} in $\mathcal G_{L/2}.$ To complete the proof, 
it remains to compute 
$$
\|\vv{u}- \vv{v} ; V\| = \|\vv{u}_{as} + \vv{u}_{R} - \vv{v}_{as} ; V\| \leq \|\vv{v}_{as} - \vv{u}_{as} ; V\| + \|\vv{u}_R ; V \|. 
$$
Next subsection is devoted to the computation of $ \|\vv{u}_R;V\|$ and the following one to $\|\vv{v}_{as} - \vv{u}_{as};V\|.$
In particular, the proof of Theorem \ref{thm_mainexact} is completed by applying Proposition \ref{prop_uR} and Proposition
\ref{prop_vas} in the general case and by applying Proposition \ref{prop_uR}, Proposition \ref{prop_v1} and Proposition \ref{prop_v0}
in the radial case.

\subsection{Asymptotics of $\vv{u}_R$}
We start with the remainder term $\mathbf u_R.$ Keeping notation $\vv{v}_R$ for $\vv{v}[\mathbf{u}_R^*],$ we prove: 
\begin{proposition} \label{prop_uR}
There exists a constant $K$ depending on $\partial \Omega,$ $\delta,$ $C^{reg}_{3},$ $C_{cvx},$ $C_{ell}$ and $L$ such that, if $h \in (0,1]$, there holds:
$$
 \int_{\mathcal F} |\nabla \vv{u}_R|^2 \leq  K \left[ \|\vv{u}^* ; H^{3}(B(\mathbf 0,L))\|^2 +\|\vv{u}^* ; H^{\frac 12}(\partial \Omega)\|^2  \right].
$$
\end{proposition} 
We recall that thanks to the variational characterization {Proposition \ref{prop_varcar}}, $\vv{u}_R$ realizes the minimum of $V$ norms among
divergence-free velocity-fields $\vv{w}$ satisfying the same boundary conditions as $\vv{u}_R$
({\em i.e.} $\vv{w} \in V[\vv{u}_R^*]$ with our notations).
By construction, we have $\vv{v}_R \in V[\vv{u}_R^*]$ so that the above proposition 
is a consequence of the lemma:
\begin{lemma} \label{lem_R}
There exists a constant $K$ depending on on $\partial \Omega,$ $\delta,$ $C^{reg}_{3},$ $C_{cvx},$ $C_{ell}$ and $L$ for which,
if $h \in (0,1]:$
\begin{equation}
\int_{\mathcal F} |\nabla \vv{v}_R|^2 \leq K \left[ \|\vv{u}^* ; H^{3}(B(\mathbf 0,L))\|^2 +\|\vv{u}^* ; H^{\frac 12}(\partial \Omega)\|^2  \right].
\end{equation}
\end{lemma}
\begin{proof}
We drop index $R$ in $\vv{v}$ for the whole proof. We recall that, by construction, we have $\vv{v} = {\vv{v}}_{in} + \vv{v}_{ext}$ where ${\vv{v}}_{in}$ and $\vv{v}_{ext}$ are computed {\em via} 
\eqref{def_tvpar}-\eqref{def_tvort0} and \eqref{eq_stokesext1}-\eqref{eq_stokesext3} respectively.  
 We first apply Lemma \ref{lem_vext} showing that we only have to focus on the following contribution of ${\vv{v}}_{in}:$
$$
\int_{\mathcal F} |\nabla {\vv{v}}_{in}|^2 = \int_{\mathcal G_L} |\nabla {\vv{v}}_{in}|^2\,.
$$
Explicit computations show that:
\begin{eqnarray*}
|\nabla_{x,y} {\vv{v}}_{in,\, \sslash}| &\leq & C_L \sum_{i=t,b} \left( |\nabla_{x,y} \gamma_i| \gamma |\nabla_{x,y} p_{in}| 
+ \gamma^2 |\nabla_{x,y}^2 p_{in}| + \left(\dfrac{|\nabla_{x,y} \gamma_i|}{\gamma} +1\right)   |\vv{u}^*_{R,\,\sslash}| + |\nabla_{x,y} \vv{u}_{R,\,\sslash}^*| \right)\,,  \\
|\partial_z {\vv{v}}_{in,\, \sslash}|
& \leq & C_L \left( \gamma |\nabla_{x,y} p_{in}| + \dfrac{1}{\gamma} |\vv{u}^*_{R,\,\sslash}|\right)\,,\\[4pt]
|\partial_z {{v}}_{in,\bot}|&\leq &C_L |\nabla_{x,y} {\vv{v}}_{in,\, \sslash}|\,, 
\end{eqnarray*}
and:
\begin{multline*}
|\nabla_{x,y} {{v}}_{in,\bot}| \leq C_L  \Bigg\{\sum_{i=t,b} \left( (|\nabla_{x,y} \gamma_i|^2 \gamma +  |\nabla^2_{x,y} \gamma_i|\gamma^2) |\nabla_{x,y} p_{in}|
+ \gamma^2 |\nabla_{x,y} \gamma_i| |\nabla_{x,y}^2 p_{in}|  \right) \\
\begin{array}{l}
+ 
\left( {|\nabla_{x,y}^2 \gamma_t| + |\nabla_{x,y}^2 \gamma| + |\nabla_{x,y} \gamma_t| + |\nabla_{x,y} \gamma|} + \dfrac{|\nabla_{x,y} \gamma_t| |\nabla_{x,y} \gamma| +  |\nabla_{x,y} \gamma|^2}{\gamma} +\gamma \right)   |\vv{u}^*_{R,\,\sslash}|    \\
 + \left({ |\nabla_{x,y} \gamma_t| + |\nabla \gamma| } + \gamma \right)  |\nabla_{x,y} \vv{u}_{R,\,\sslash}^*|+ \gamma |\nabla^2_{x,y} \vv{u}_{R,\,\sslash}^*| + \gamma^3 |\nabla_{x,y}^3 p_{in}|\Bigg\}\,,
\end{array}
\end{multline*}
with $C_L$ a constant depending on $L.$ 
Introducing that, for all $(x,y) \in B(\mathbf 0,L),$ 
there holds 
$$
|\nabla_{x,y} \gamma_b| + |\nabla_{x,y} \gamma_t|  +|\nabla_{x,y} \gamma|   \leq K[C^{reg}_2,C_{ell}] \gamma^{\frac 12}\, ,
\quad 
|\nabla^2_{x,y} \gamma_i| +  |\nabla^2_{x,y} \gamma|  +|\nabla^2_{x,y} \gamma| \leq K[C_{2}^{reg}]\,,
$$
we obtain that, on $\mathcal G_{L},$ there holds:
$$
|\nabla {\vv{v}}_{in}| \leq K \left(  \gamma |\nabla_{x,y} p_{in}| + \gamma^2 |\nabla_{x,y}^2 p_{in}| + 
\gamma^3 |\nabla_{x,y}^3 p_{in}| + \left(1+ \frac{1}{\gamma}\right) |\vv{u}^*_{R,\,\sslash}| + |\nabla_{x,y} \vv{u}^*_{R,\,\sslash}| +  \gamma |\nabla_{x,y}^{2} \vv{u}^*_{R,\,\sslash}| \right)\,. 
$$
Integrating this inequality entails that:
\begin{multline} \label{eq_encoreune}
\int_{\mathcal G_{L}} |\nabla {\vv{v}}_{in}|^2 {\rm d}x{\rm d}y{\rm d}z\leq  K[C_{ell},C^{reg}_2,L]  \|\vv{u}^*_{R,\,\sslash} ; H^2(B(\mathbf 0,L))\|^2 \\
+ K[C_{ell},C^{reg}_2,L] \int_{B(\mathbf 0,L)} \left(\gamma^3 |\nabla_{x,y} p_{in}|^2 + \gamma^5 |\nabla_{x,y}^2 p_{in}|^2 + \gamma^7 |\nabla_{x,y}^3 p_{in}|^2 + \dfrac{|\vv{u}^*_{R,\,\sslash}|^2}{\gamma}\right) {\rm d}x{\rm d}y .
\end{multline}
In this last identity, we note that, for $k\in \{1,2,3\}:$ 
$$
|\nabla_{x,y}^k p_{in}(x,y)| \leq |\nabla^k_{x,y} q_{in}(x,y)| + C_L\sum_{j=0}^{k-1} \mathbf{1}_{B(\mathbf 0,L) \setminus B(\mathbf 0,L/2)}(x,y) |\nabla_{x,y}^j q_{in}(x,y)|\,, \quad \forall \,(x,y) \in B(\mathbf 0,L) .  
$$
{So, we split again $q_{in} = q_{in}^{(1)}+hq_{in}^{(2)}$ with $q_{in}^{(1)}$ and $q_{in}^{(2)}$ defined respectively as the solutions to 
\begin{equation} \label{eq_pourq1R}
\begin{array}{rcll}
-\dfrac{1}{12} {\rm div}_{x,y}(\gamma^3 \nabla_{x,y} q^{(1)}_{in}) &=&  (u^*_{R,\bot} + \gamma_b{\rm div}_{x,y} \vv{u}_{R,\, \sslash}^{*} ) - \dfrac{1}{2}{\rm div}_{x,y}  (\gamma_t + \gamma_b )  \vv{u}^*_{R,\, \sslash}   \,, & \text{ on $B(\mathbf 0,L)\,,$} \\
q^{(1)}_{in} &=& 0 \,, & \text{ on $\partial B(\mathbf 0,L)$}\,. 
\end{array}
\end{equation}
and
\begin{equation} \label{eq_pourq2R}
\begin{array}{rcll}
-\dfrac{1}{12} {\rm div}_{x,y}(\gamma^3 \nabla_{x,y} q^{(2)}_{in}) &=& -\dfrac 12 {\rm div}_{x,y} \vv{u}_{R,\,\sslash}^{*}  \,, & \text{ on $B(\mathbf 0,L)\,,$} \\
q^{(2)}_{in} &=& 0 \,, & \text{ on $\partial B(\mathbf 0,L)$}\,. 
\end{array}
\end{equation}
So we have that $q^{(2)}_{in}$is a solution to \eqref{eq_papp1bis}-\eqref{eq_papp2bis} with a source term $f$ given by \eqref{eq_formf} associated with
$$
w_{\bot}^{*} = - \dfrac 12 {\rm div}_{x,y} \vv{u}_{R,\,\sslash}^*  \qquad \vv{v}^* = 0.
$$
where we note that $\|{\rm div}_{x,y} \vv{u}_{R,\,\sslash}^* ; L^{\infty}(B(0,L))\| \leq C\|\vv{u}^*_R ; H^3(B(\mathbf 0,L))\|.$
Applying propositions \ref{prop_unifreg}, \ref{prop_reg} and \ref{prop_reg2} to $q_{in}^{(2)},$ we obtain that, 
\begin{multline*}
\int_{B(\mathbf 0,L)} \left(\gamma^3 |\nabla_{x,y} q^{(2)}_{in}|^2 + \gamma^5 |\nabla_{x,y}^2 q^{(2)}_{in}|^2 + \gamma^7 |\nabla_{x,y}^3 q^{(2)}_{in}|^2  \right) {\rm d}x{\rm d}y 
\\ +  \|q_{in}^{(2)} ; H^3(B(\mathbf 0,L) \setminus B(\mathbf 0,L/2) )\| 
\leq \dfrac{K\|\vv{u}_R^* ; H^3(B(\mathbf 0,L))\|^2}{h}.
\end{multline*}
with $K$ depending on $C^{reg}_{3},$ $C_{cvx},$ $C_{ell}$ and $L.$

As for $q_{in}^{(1)},$ we remark that it is a solution to \eqref{eq_papp1bis}-\eqref{eq_papp2bis} with a source term $f$ given by \eqref{eq_formf} associated with
$$
w_{\bot}^{*} = u_{R,\bot}^{*} + \gamma_b {\rm div}_{x,y} \vv{u}_{R,\,\sslash}^*  \qquad \vv{v}^* = \vv{u}_{R,\, \sslash}^*.
$$
In this case, we have that  $w_{\bot}^{*}(\mathbf 0)=0$ and also that  $\nabla w_{\bot}^*(\mathbf 0),\vv{v}^*(\mathbf 0)$ vanish.
Consequently, we are in position to apply 
propositions \ref{prop_unifreg}, \ref{prop_reg} and \ref{prop_favorable} to $q_{in}^{(1)},$ yielding that:
\begin{multline*}
\int_{B(\mathbf 0,L)} \left(\gamma^3 |\nabla_{x,y} q^{(1)}_{in}|^2 + \gamma^5 |\nabla_{x,y}^2 q^{(1)}_{in}|^2 + \gamma^7 |\nabla_{x,y}^3 q^{(1)}_{in}|^2 \right)  \\
+ \|q_{in}^{(1)} ; H^3(B(\mathbf 0,L) \setminus B(\mathbf 0,L/2) )\|   \leq K \|\vv{u}^*_{R} ; H^3(B(\mathbf 0,L))\|^2   \,,
\end{multline*}  
with $K$ depending again on $C^{reg}_{3},$ $C_{cvx},$ $C_{ell}$ and $L.$ 
Combining the computations for $q_{in}^{(1)}$ and $q_{in}^{(2)},$ and arguing that
$$
\|\vv{u}^*_{R} ; H^3(B(\mathbf 0,L))\| \leq C \|\vv{u}^* ; H^3(B(\mathbf 0,L))\|
$$ 
we get finally:
\begin{equation} \label{eq_bornecasR}
\int_{B(\mathbf 0,L)} \left(\gamma^3 |\nabla_{x,y} p_{in}|^2 + \gamma^5 |\nabla^2_{x,y} p_{in}|^2 + \gamma^7 |\nabla^3_{x,y} p_{in}|^2 \right)  \leq K \|\vv{u}^* ; H^3(B(\mathbf 0,L))\|^2  \,,
\end{equation}
with $K$ depending on $C^{reg}_{3},$ $C_{cvx},$ $C_{ell}$ and $L.$
}

For the last term on the right-hand side of \eqref{eq_encoreune}, we add that $\vv{u}_{\sslash}^*(\mathbf 0) =\mathbf 0$,  implying:
$$
|\vv{u}_{R,\,\sslash}^*(x,y)| \leq (|x| + |y|) \|\vv{u}^* ; H^3({B(\mathbf 0,L)})\|\,, \quad \text{ on $B(\mathbf 0,L)$}\,.
$$
Consequently, going to polar coordinates yields (recall $\gamma$ satisfies \eqref{eq_Cell} whatever the value of $h \in (0,1]$):
$$
\int_{B(\mathbf 0,L)} \dfrac{|\vv{u}^*_{R,\,\sslash}|^2}{\gamma} \leq \dfrac{2\pi  \|\vv{u}^* ; H^3({B(\mathbf 0,L)})\|^2\ }{C_{ell}} \int_0^L \dfrac{r^3 {\rm d}r}{ r^2}  \leq {\dfrac{\pi L^2}{C_{ell}}} { \|\vv{u}^* ; H^3({B(\mathbf 0,L)})\|^3}.
$$
This ends the proof.
\end{proof}

\subsection{Asymptotics of  $\vv{v}_{as} - \vv{u}_{as}$}
We proceed with a first bound on the singular term in the general case ({\em i.e.}, without structure assumption on $\gamma$). We prove:
\begin{proposition}\label{prop_vas} If $h \in (0,1],$ there exists a constant $K$ depending on $\partial \Omega,$ $\delta,$ $C^{reg}_{3},$ $C_{cvx},$ $C_{ell}$ and $L$ such that there holds:
$$
\int_{\mathcal F} |\nabla \vv{u}_{as} - \nabla \vv{v}_{as}|^2 \leq K[ |u^*_{\bot}(\mathbf 0)|^2 |\ln(h)| + \|\vv{u}^* ; H^3(B(\mathbf 0,L))\|^2].
$$ 
\end{proposition}
\begin{proof}
We recall that $\vv{v}_{as} = \vv{v}_{in} + \vv{v}_{ext}.$ 
It is sufficient to compute a constant $K$ (independent of $h$) such that, if $h \in (0,1]$, for all $\vv{w} \in \mathcal V_0$, we have:
$$
\left| \int_{\mathcal F} (\Delta {\vv{v}}_{in} - \nabla {p}_{in}) \cdot \vv{w} \right| \leq K  \left[ \int_{\mathcal F} |\nabla \vv{w}|^2\right]^{\frac 12}. 
$$
(with ${p}_{in}$ given by \eqref{eq_ptilde}). Indeed,  we have then:
$$
\left| \langle \Delta {\vv{v}}_{as} - \nabla p_{in}, \vv{w} \rangle \right| \leq  \left| \int_{\mathcal F} (\Delta \vv{v}_{in} - \nabla p_{in}) \cdot \vv{w} \right| + \left|\int_{\mathcal F} \nabla \vv{v}_{ext} : \nabla \vv{w}\right|
$$
Applying Lemma \ref{lem_vext} together with the remark that the mapping $\vv{u}^* \mapsto \vv{u}^*_{as}$ is linear continuous $H^3(B(\mathbf 0,L)) \to H^2(B(\mathbf 0,L)) \cap H^{\frac 12}(\partial \Omega)$, we bound the other term by:
$$
\left|\int_{\mathcal F} \nabla \vv{v}_{ext} : \nabla \vv{w}\right| \leq K \|\vv{u}^* ; H^3(B(\mathbf 0,L))\| \|\vv{w} ; V\|\,.
$$
with $K$ depending on $\partial \Omega,$ $C_{2}^{reg},$ $C_{ell}$ and $\delta.$ 
 
\medskip

We emphasize that $\vv{v}_{in}$ and $p_{in}$ have support in $\mathcal G_L$
so that:
$$
I[\vv{w}] :=  \int_{\mathcal F} (\Delta \vv{v}_{in} - \nabla p_{in}) \cdot \vv{w}=  \int_{\mathcal G_L} (\Delta \vv{v}_{in} - \nabla p_{in}) \cdot \vv{w} \,.
$$
By construction, we have that:
$$
\partial_{zz} \vv{v}_{in,\, \sslash} = \nabla_{x,y} p_{in}\,, \qquad \partial_z p_{in} = 0
$$
Consequently, there holds:
$$
\Delta \vv{v}_{in} - \nabla p_{in} = \Delta_{x,y} \vv{v}_{in,\,\sslash} + \Delta {v}_{in,\bot} \vv{e}_z.
$$
For any $\vv{w} \in \mathcal V_0$ we can then bound $I[\vv{w}]$ by integrating by parts:
\begin{eqnarray*}
| I[\vv{w}] |  &=& \left| \int_{\mathcal G_{L}} \nabla_{x,y} \vv{v}_{in,\,\sslash} : \nabla_{x,y} \vv{w} + 
\int_{\mathcal G_L} \nabla {v}_{in,\bot} \cdot \nabla w_{\bot} \right|\\
&\leq & \left[\int_{\mathcal G_L} |\nabla_{x,y} \vv{v}_{in,\,\sslash}|^2 + |\nabla {v}_{in,\bot}|^2 \right]^{\frac 12} \|\nabla \vv{w} ;  L^2(\mathcal F)\|.
\end{eqnarray*}
The remainder of the proof follows the line of the proof of {Lemma \ref{lem_R}}.
First, we bound $|\nabla_{x,y} \vv{v}_{in,\,\sslash}|$ and $|\nabla {v}_{in,\bot}|$
as in the proof of this lemma, this yields: 
\begin{multline*}
|\nabla_{x,y} \vv{v}_{in,\,\sslash}| + |\nabla {v}_{in,\bot}| \\
\leq K[C_{ell},C^{reg}_2]  \Bigg(
 \gamma^{\frac 32} |\nabla_{x,y} p_{in}| + \gamma^2 |\nabla^2_{x,y} p_{in}| + 
\gamma^3 |\nabla^3 p_{in}|  \\
+\left( 1 +  \dfrac 1{\gamma^{\frac 12}}\right) |\vv{u}^*_{as,\,\sslash}|  {+ |\nabla_{x,y}\vv{u}^*_{as,\,\sslash}| + |\nabla^2_{x,y}\vv{u}^*_{as,\,\sslash}| }
\Bigg)\,.
\end{multline*}
We obtain then:
\begin{multline} \label{eq_controlvas}
\left|\int_{\mathcal G_{L}} |\nabla_{x,y} \vv{v}_{in,\,\sslash}|^2 + |\nabla {v}_{in,\bot}|^2 \right| \\
 \leq C \left[ {\|\vv{u}^*_{as,\,\sslash} ; H^2(B(\mathbf 0,L))\|^2} 
+ 
\int_{B(\mathbf 0,L)} \!\!\!\!\!\!\left(\gamma^4 |\nabla_{x,y} p_{in}|^2 + \gamma^5 |\nabla^2_{x,y} p_{in}|^2 + \gamma^7 |\nabla^3_{x,y} p_{in}|^2 \right)  \right],
\end{multline}
and we bound the last integrals on the right-hand side by computing $p_{in}$ with respect to $q_{in}$ and  applying {Propositions \ref{prop_reg2}, \ref{prop_unifreg} and \ref{prop_reg}}  as in the proof of {Lemma \ref{lem_R}}. Note that $q_{in}$ has only one component because ${\rm div}_{x,y} \mathbf{u}^*_{as,\,\sslash} = 0.$
\end{proof}

We end this section by considering the case where the aperture admits a cylindrical invariance: 
$\gamma = \gamma(\sqrt{x^2+y^2}).$ In this case we introduce $(\vv{u}_0,p_0)$
and $(\vv{u}_1,p_1)$ the solutions to the Stokes system associated with boundary conditions:
\begin{eqnarray*}
\vv{u}^*_0 &=& u^*_{\bot}(\mathbf 0)\vv{e}_z \,, \\
\vv{u}^*_1 &=&  \vv{u}^*_{\sslash}(\mathbf 0) +( x \partial_x u^*_{\bot}(\mathbf 0) + y \partial_y u^*_{\bot}(\mathbf 0)) \vv{e}_z \,,
\end{eqnarray*}
We note that straightforward computations entail:
$$
\int_{\partial \Omega} \vv{u}^*_i \cdot \vv{n} {\rm d}\sigma  =0, \qquad \vv{u}^*_i \in C^{\infty}(\partial \Omega),  \quad \forall \, i \in \{0,1\}\,.
$$
So, we have indeed existence and uniqueness of the pairs $(\vv{u}_i,p_i)_{i=0,1}.$
We also introduce $\vv{v}_0=\vv{v}[\vv{u}_0^*],\;\vv{v}_1= \vv{v}[\vv{u}_1^*]$ the respective approximations of $\vv{u}_0$ and $\vv{u}_1$ constructed applying the steps depicted in  Section \ref{sec_asvf}. We note that, due to the linearity of the Stokes problem and of our construction, we have:
$$
\vv{u}_{as} = \vv{u}_0 + \vv{u}_1\,, \, \quad \vv{v}_{as} = \vv{v}_0 + \vv{v}_1 \,. 
$$

First, remarking that $\vv{u}^*_1(\mathbf 0) \cdot \vv{e}_z$ vanishes (so that there is no logarithmic terms yielding from the application of {Proposition \ref{prop_reg2}}), we reproduce the computations in the proof of the previous proposition and obtain at first:
\begin{proposition} \label{prop_v1} If $h \in (0,1],$ there exists a constant $K$ depending on $\partial \Omega,$ $\delta,$  $C^{reg}_{3},$ $C_{cvx},$ $C_{ell}$ and $L$ such that there holds: 
$$
\int_{\mathcal F} |\nabla \vv{u}_{1} - \nabla \vv{v}_{1}|^2 \leq K\|\vv{u}^* ; H^3(B(\mathbf 0,L))\|^2.
$$ 
\end{proposition}

\medskip

We complete the study of $\vv{u}_{as}- \vv{v}_{as}$ by computing the asymptotics of the 
most singular term: $\vv{u}_0 - \vv{v}_0.$
We prove:
\begin{proposition} \label{prop_v0} If $h \in (0,1]$ and $\gamma$ is radial, there exists a constant $K$ depending on $\partial \Omega,$ $\delta,$ $C^{reg}_{3},$ $C_{cvx},$ $C_{ell}$ and $L$ such that there holds:
$$
\int_{\mathcal F} |\nabla \vv{u}_0 - \nabla \vv{v}_0|^2 \leq K\|\vv{u}^* ; H^2(B(\mathbf 0,L))\|^2.
$$
\end{proposition}
\begin{proof}
As the mapping $\vv{u}^* \mapsto \vv{u}_0^*$ is continuous $H^2(B(\mathbf 0,L)) \to H^3(B(\mathbf 0,L)) \cap H^{\frac 12}(\partial \Omega),$ we treat only the case $\vv{u}_{0}^*= \mathbf{e}_z$
and we drop superfluous index $0.$

\medskip

In this cylindrical case, a particular feature of $\vv{v}_0$ is that, in the problem solved by ${q}_{in}$ the weight $\gamma$ is invariant by rotation 
around the origin and the source term is a constant. Consequently, ${q}_{in}$ is a radial function and an explicit formula is available (as in the proof of Proposition \ref{prop_casf0}). Up to assume that $\chi_L$ is also
radial, we get that ${\vv{v}}_{in,\, \sslash}$ is directed along the radial unit vector $\vv{e}_r.$
More precisely, we have:
\begin{equation} \label{eq_formv}
{\vv{v}}_{in,\, \sslash}(x,y,z) = {v}_r(r,z) \vv{e}_r \,, \quad \text{ a.e. in $\mathcal G_{L}$.} 
\end{equation}
with an explicit formula in the aperture:
$$
v_r(r,z) = - \dfrac{3r}{\gamma^3(r)} (z-(h+\gamma_t(r)))(z-\gamma_b(r)) \,, \quad \text{ a.e. in $\mathcal G_{L/2}$.} 
$$

\medskip

Following the proof of Proposition \ref{prop_vas} in the previous section, to bound the distance between ${\vv{v}}_{0}$ and $\vv{u}_0$, we compute integrals 
$$
{I}_{in}[\vv{w}] := \int_{\mathcal G_L} \left( \Delta {\vv{v}_{in}} - \nabla {p}_{in}\right) \cdot \vv{w}.
$$ 
Again, as in the previous section, explicit computations and integrations by parts yield that:
$$
{I}_{in}[\mathbf w] = \int_{\mathcal G_L} \Delta_{x,y} {\vv{v}}_{in,\, \sslash} \cdot  \vv{w}_{\sslash} - \int_{\mathcal G_L} \nabla_{x,y} {v}_{in,\bot} \cdot \nabla_{x,y} w_{\bot} - \int_{\mathcal G_L} \partial_z {v}_{in,\bot} \partial_z w_{\bot}.
$$
Here, we introduce that 
\begin{equation} 
\partial_z w_{\bot} = - {\rm div}_{x,y} \vv{w}_{\sslash}\, \quad \partial_z {v}_{in,\bot} =  - {\rm div}_{x,y} {\vv{v}}_{in,\, \sslash}\,.  
\end{equation}
We plug these identities in the above integrals and integrate by parts.
Because of the radial form of ${\vv{v}}_{in,\, \sslash}$ (see \eqref{eq_formv}), we have 
$$\nabla_{x,y} {\rm div}_{x,y} {\vv{v}}_{in,\, \sslash} =\Delta_{x,y}  {\vv{v}}_{in,\, \sslash}.
$$
This entails:
$$
{I}_{in}[\mathbf w] = \int_{\mathcal G_L}  2 \Delta_{x,y}  {\vv{v}}_{in,\,\sslash} \cdot \vv{w}_{\sslash} - \int_{\mathcal G_L} \nabla_{x,y} {v}_{in,\bot} \cdot \nabla_{x,y} w_{\bot} 
$$ 
For the last integral, we bound $|\nabla_{x,y} {v}_{in,\bot}|$ similarly as in the proof of {Lemma \ref{lem_R}}. 
Introducing the bounds on $\gamma_b,\gamma_t$ and $\gamma,$ this entails: 
$$
|\nabla_{x,y} {v}_{in,\bot}| \leq C \left( \gamma^{2} |\nabla_{x,y} {p}_{in}| + \gamma^{\frac 52} |\nabla_{x,y}^2 {p}_{in}| + \gamma^3 |\nabla_{x,y}^3 {p}_{in}| \right)
$$
so that, introducing now that ${p}_{in}= \chi_L {q}_{in}$ and applying Proposition \ref{prop_unifreg}:
$$
\int_{\mathcal G_L} |\nabla_{x,y}{v}_{in,\bot}|^2 \leq K[C^{reg}_{2},C_{ell},C_{cvx}] \left( \int_{B(\mathbf 0,L)} \left[ \gamma^5 |\nabla {q}_{in}|^2 + \gamma^6 |\nabla^2 {q}_{in}|^2 + 
\gamma^7 |\nabla^3 {q}_{in}|^2 \right] + 1\right).
$$
Applying Propositions \ref{prop_reg2} and \ref{prop_reg} to ${q}_{in}$ entails:
$$
\int_{\mathcal G_L} |\nabla_{x,y}{v}_{in,\bot}|^2 \leq K[C^{reg}_{3},C_{ell},C_{cvx}] \,.
$$

\medskip

Then, we truncate $\vv{w}_{\sslash}$ with $\chi_{L/2}$ that vanish on the lateral boundaries of $\partial \mathcal G_{L/2}$ (and is equal to $1$ on $\mathcal G_{L/4}$) 
and we obtain that:
$$
\left| \int_{\mathcal G_L} \Delta_{x,y}  {\vv{v}}_{in,\,\sslash} \cdot \vv{w}_{\sslash}\right|
\leq  
\int_{\mathcal G_{L/2}} |\Delta_{x,y}  {\vv{v}}_{in,\,\sslash}| \cdot |\chi_{L/2} \vv{w}_{\sslash}| 
+ C_L \|{\vv{v}}_{in,\,\sslash} ; H^1(\Omega \setminus \mathcal G_{L/4}) \| \|\vv{w}_{\sslash} ; H^1(\Omega)\|.
$$
With similar arguments as in the proof of Lemma \ref{lem_vext} (see \eqref{eq_restevstar}), we bound the last term on the right-hand side:
$$
 \|{\vv{v}}_{in,\,\sslash} ; H^1(\Omega \setminus \mathcal G_{L/4}) \| \leq K[C_{ell},C^{reg}_2,\delta].
$$

\medskip

Finally, we apply the formula for $ {\vv{v}}_{in,\,\sslash}$ which implies in particular that:
$$
 \Delta_{x,y} {\vv{v}}_{in,\,\sslash} =  \left[  r \partial_{rr} \left( \dfrac{v_r}{r}\right) + 3  \partial_r \left( \dfrac{v_r}{r}\right)\right]\,.
$$
{Hence introducing that $\vv{w}$ vanishes on the upper and lower boundaries of $\mathcal G_{L/2},$ we bound with Cauchy Schwartz inequalities and a Hardy inequality
in the $z$ direction:
\begin{eqnarray*}
\int_{\mathcal G_L}  \left| \Delta_{x,y}  {\vv{v}}_{in,\,\sslash} \right|
\cdot |\chi_{L/2} \vv{w}_{\sslash}|   
&\leq& 
\int_{0}^{2\pi} \int_0^{L} \int_{\gamma_b(r)}^{h+\gamma_t(r)} (z-\gamma_b) \left| \Delta_{x,y}  {\vv{v}}_{in,\,\sslash} \right| \cdot \dfrac{|\chi_{L/2} \vv{w}_{\sslash}| }{z-\gamma_b} {\rm d}r {\rm d}z {\rm d}{\theta} \\
&\leq &
C  \left[\int_{\mathcal G_L} \left|  \gamma \left[  r \partial_{rr} \left( \dfrac{v_r}{r}\right) + 3  \partial_r \left( \dfrac{v_r}{r}\right) \right] \right|^2 \right]^{\frac 12}  \|\nabla \vv{w} ; L^2(\mathcal F)\|.
\end{eqnarray*}
We emphasize that the constant $C$ is universal and in particular independent of $h$.}
With the explicit formula for $v_r$ we get:
$$
\left|  \gamma \left[  r \partial_{rr} \left( \dfrac{v_r}{r}\right) + 3  \partial_r \left( \dfrac{v_r}{r}\right) \right] \right| \leq  K[C^{reg}_2,C_{ell}] \dfrac{r}{\gamma}
$$
As 
$$
\int_{\mathcal G_{L/2}} \left|\dfrac{r}{\gamma}\right|^2 \leq C \int_{0}^{L/2} \dfrac{r^3 {\rm d}r}{\gamma(r)} \leq K[C_{ell}], 
$$
we get that, for a constant $K$ depending on $C^{reg}_3$ and $C_{ell},C_{cvx},$ there holds: 
$$
|{I}_{in}[\mathbf w]| \leq  K \|\mathbf{u}^*;H^2(B(\mathbf 0,L))\|  \|\nabla \vv{w} ; L^2(\mathcal F)\|.
$$
This ends the proof.
\end{proof}

\section{Proof of {\bf Theorem \ref{cor_mainexact}}}
In this last section, we exhibit a particular case where the above informations yield a sharp asymptotic
expansion of the quantity:
$$
\int_{\mathcal F} |\nabla \vv{u}|^2.
$$
Throughout this last section, we assume that $\Omega = \mathbb R^3 \setminus \mathcal B^*$ 
with $\mathcal B^*$ a sphere of radius $R$ and  that $\mathcal B$ is a sphere of radius $S.$
We recall that we have then:
\begin{eqnarray*}
\gamma_t(x,y) &=& \dfrac{x^2 + y^2}{2S} + O((x^2+y^2)^2)\,, \\
\gamma_b(x,y) &=& - \dfrac{x^2 + y^2}{2R} + O((x^2+y^2)^2)\,, \\   
\gamma(x,y) &=& \dfrac{x^2 + y^2}{2R_1} - \dfrac{(x^2+y^2)^2}{8R_3^3}+ O((x^2+y^2)^{3})\,,   
\end{eqnarray*}
where $R_1$ and $R_3$ satisfy \eqref{eq_R}.\ We fix also a smooth boundary data $\mathbf{u}^*$ and denote by $\vv{u}$ the exact solution to the Stokes problem with boundary condition $\vv{u}^*$ and $\vv{v}[\vv{u}^*]$ the approximation that is constructed in Section \ref{sec_asvf}. 

\medskip

We introduce the notations of the previous section: indices $0,1,as,R$ distinguish the components of $\vv{u}^*$ and $\vv{u}$ and $\vv{v}[\vv{u}^*].$ 
Hence, we have 
$$
\vv{v}[\vv{u}^*] = \vv{v}_0 + \vv{v}_1 + \vv{v}_R,
$$
and a similar decomposition for $\vv{u}.$ We decompose 
$$
\|\vv{u}- \vv{v}[\vv{u}^*] ; V \| \leq  \| \vv{u}_{0} - \vv{v}_0 ; V \| + \|\vv{u}_1 - \vv{v}_1 ; V\| + \|\vv{u}_R ; V\| + \|\vv{v}_R ; V\|
$$
Applying Proposition \ref{prop_uR} and Proposition \ref{prop_v1}, Proposition \ref{prop_v0}  in the cylindrical case we deduce:
$$
\int_{\mathcal F} |\nabla (\mathbf u- \mathbf v[\vv{u}^*])|^2=  O (1) \left\{  \|\vv{u}^* ; H^3(B(\mathbf 0,L))\|^2 + \|\mathbf{u}^* ; H^{\frac 12}(\partial \Omega)\|^2\right\},
$$
where we keep the convention that landau notations $O(1)$ stand for quantities depending on $(h,R,S)$ which remains bounded by a constant depending on $R,S$ only for $h \in (0,1].$

\medskip

The weak formulation of the Stokes problem (remarking that $\vv{u}-\vv{v}[\vv{u}^*]$ vanishes on $\partial \mathcal F$) yields also that:
$$
\int_{\mathcal F} \nabla \vv{u} : \nabla (\vv{u}- \vv{v}[\vv{u}^*]|) = 0\,, \quad 
$$
Hence, we have:
$$
\int_{\mathcal F} |\nabla \vv{u}|^2 = \int_{\mathcal F} |\nabla \vv{v}[\vv{u}^*]|^2 +  O (1) \left\{  \|\vv{u}^* ; H^3(B(\mathbf 0,L))\|^2 + \|\mathbf{u}^* ; H^{\frac 12}(\partial \Omega)\|^2\right\}.
$$
To compute the first integral on the right-hand side of this last equality, we split:
$$
\int_{\mathcal F} |\nabla \vv{v}[\vv{u}^*]|^2 = \sum_{i=0,1,R} E_i + 2 \left( E_{01} + E_{0R} + E_{1R}\right)\,,
$$
where for $(i,j) \in \{0,1,R\}$ we define:
$$
E_i  = \int_{\mathcal F} |\nabla \vv{v}_i|^2 \qquad E_{ij} = \int_{\mathcal F} \nabla \vv{v}_i : \nabla \vv{v}_j\,.
$$
We complete the proof by studying the asymptotics of all these integrals.
The first-order terms will yield by computing $E_0$ and $E_1.$ 

\subsection{Study of positive terms}
We recall that, by applying Proposition \ref{prop_uR}, we get at first that 
\begin{equation} \label{eq_firstboundvR}
\int_{\mathcal F} |\nabla \vv{v}_R|^2 = O(1)\left\{ \|\vv{u}^* ; H^3(B(\mathbf 0,L))\|^2 + \|\vv{u}^* ; H^{\frac 12}(\partial \Omega)\|^2\right\}.
\end{equation}
Concerning the other terms, we remark that, by construction $\vv{v}_i = {\vv{v}}_{i,in} + \vv{v}_{i,ext}$
so that:
$$
\int_{\mathcal F} |\nabla \vv{v}_{i} |^2  = \int_{\mathcal F} |\nabla {\vv{v}}_{i,in}|^2 + 2\int_{\Omega^*} \nabla {\vv{v}}_{i,in} : \nabla {\vv{v}_{i,ext}} +  \int_{\Omega^*} |\nabla \vv{v}_{i,ext}|^2,
$$
where, reproducing the computations in the proof of Lemma \ref{lem_vext} (see \eqref{eq_restevstar}), we obtain that, for $i=0,1$:
$$
\int_{\Omega^*} |\nabla \vv{v}_{i,ext}|^2 + \int_{\Omega^*} |\nabla {\vv{v}}_{i,in}|^2 = O(1) \| \mathbf{u}^* ; H^3(B(\mathbf 0,L))\|^2. 
$$
Hence, we get that, for $i=0,1$:
$$
E_i = {E}_{i,in}  + O (1) \| \mathbf{u}^* ; H^3(B(\mathbf 0,L))\|^2, \quad {E}_{i,in} = \int_{\mathcal F} |\nabla {\vv{v}}_{i,in}|^2.
$$

\subsubsection{Asymptotics of ${E}_{1,in}$ }
Let first consider $\vv{v}_1.$ We drop index $1$ in the sequel and we recall that $\vv{v}_{1,in},$ denoted here by ${\vv{v}}_{in},$ is constructed as follows:
\begin{eqnarray*} 
{\vv{v}}_{in,\,\sslash\!}({x},{y},{z})  &=&  \dfrac{1}{2} ( z - (h+{\gamma}_t))( z - {\gamma}_b) \nabla_{x,y\,}  {p}_{in}  - \left(\dfrac{z -(h+ \gamma_t)}{\gamma } \right) \chi_L  \vv{u}^*_{\,\sslash}(\mathbf 0) \,, \\[4pt] \label{def_tvort}
{{v}}_{in,\bot}({x},{y},{z}) &=&  \dfrac 12 {\rm div}_{x,y}  \left[ \int^{h+\gamma_t}_{z} (s-(h+\gamma_t))(s-\gamma_b)\nabla_{x,y\,}  {p}_{in} \, {\rm d}s\right]   \\
&& + \,\int^{h+\gamma_t}_{z}  {\rm div}_{x,y}\left[  \left(\dfrac{s -( h+ \gamma_t)}{\gamma } \right)\chi_L\vv{u}^*_{\,\sslash}(\mathbf 0) {\rm d}s \right]\,,   \notag
\end{eqnarray*}
for $(x,y,z) \in \mathcal G_{L},$ where:
$$
{p}_{in}(x,y) =  \chi_L(x,y) {q}_{in} (x,y) \,, \quad \forall \, (x,y,z) \in \mathcal G_{L}\,.
$$
with $\chi_L$ a suitable truncation function and ${q}_{in}$ the unique solution to
$$
\begin{array}{rcll}
-\dfrac{1}{12} {\rm div}_{x,y}(\gamma^3 \nabla_{x,y} {q}_{in}) &=&  (x\partial_x u^*_{\bot}(\mathbf 0) + y \partial_y u^*_{\bot}(\mathbf 0) ) \vv{e}_z- {\rm div}_{x,y} \left[ \dfrac{(\gamma_t+ \gamma_b)}{2}  \vv{u}^*_{\sslash}(\mathbf 0)\right]\,, & \text{ on $B(\mathbf 0,L)\,,$} \\
{q}_{in} &=& 0 \,, & \text{ on $\partial B(\mathbf 0,L)$}\,. 
\end{array}
$$
From {Proposition \ref{prop_unifreg}}, we obtain first that:
$$
{E}_1 = \int_{\mathcal G_{L/2}} |\nabla {\vv{v}}_{in}|^2 + O(1) \|\mathbf u^*; H^{3}(B(\mathbf 0,L))\|^2,
$$
hence we may replace ${p}_{in}$ by ${q}_{in}$ in computations from now on. Then, it comes from the proof of {Proposition \ref{prop_vas}} (see \eqref{eq_controlvas}) that:
\begin{equation} \label{eq_firstboundv1}
\int_{\mathcal G_{L/2}} |\nabla_{x,y}{\vv{v}}_{in,\,\sslash\!} |^2 + |\nabla{{v}}_{in,\bot} |^2 = O(1) \|\mathbf u^*; H^{3}(B(\mathbf 0,L))\|^2.
\end{equation}
We obtain:
$$
{E}_{1,in} = \int_{\mathcal G_{L/2}}  | \partial_z {\vv{v}}_{in,\,\sslash\!} |^2  + O(1) \|\mathbf u^*; H^{3}(B(\mathbf 0,L))\|^2.
$$
Explicit computations yield that, on $\mathcal G_{L/2},$ there holds:
$$
\partial_z {\vv{v}}_{in,\,\sslash\!} = \dfrac{(z- (h+\gamma_t)) + (z - \gamma_b)}{2} \nabla_{x,y} {q}_{in} - \dfrac{\vv{u}^*_{\sslash\!}(\mathbf 0)}{\gamma}.
$$
Consequently, we have that: ${E}_{1,in} = {E}^1_{1,in} + {E}^2_{1,in}$ where (note that the cross-term vanishes by integrating w.r.t. $z$-variable at first)
$$
{E}^1_{1,in} =  \dfrac 14 \int_{\mathcal G_{L/2}}  |(z-(h+\gamma_t)) + (z - \gamma_b)|^2 | \nabla_{x,y} \tilde{q}_{in} |^2\,, \quad
{E}^2_{1,in} = \int_{\mathcal G_{L/2}} \dfrac{|\vv{u}^*_{\sslash\!}(\mathbf 0)|^2}{\gamma^2}.
$$
We end up the proof by computing the asymptotics of ${E}^{1}_{1,in}$ and ${E}^2_{1,in}.$

\medskip

Concerning ${E}^2_{1,in},$ we expand, for sufficiently small $r_0:$
\begin{eqnarray*}
{E}^2_{1,in} &=& 2\pi |\vv{u}^*_{\sslash\!}(\mathbf 0)|^2 \int_0^{L/2}  \dfrac{r{\rm d}r}{\gamma(r)} \, \\
	&=& 2\pi |\vv{u}^*_{\sslash\!}(\mathbf 0)|^2 \left(  \int_{0}^{r_0}  \dfrac{r{\rm d}r}{(h + \frac{r^2}{2R_1})} + O(1)\right) \, \\[4pt]
	&=& 2\pi R_1  |\vv{u}^*_{\sslash\!}(\mathbf 0)|^2  |\ln(h)| + O(1) \|\mathbf u^*;H^3(B(\mathbf 0,L))\|^2\,.
\end{eqnarray*}

As for ${E}^1_{1,in},$ we go back to the computations of Section \ref{sec_Gilbarg}. Indeed, integrating at first with respect to $z,$ we get:
$$
{E}^1_{1,in} = \dfrac{1}{12}  \int_{B(\mathbf 0,L/2)} \gamma^3 |\nabla  {q}_{in}|^2, 
$$
where we apply {Proposition \ref{prop_casf1}} to ${q}_{in}$ to compute the asymptotics of this last quantity. This yields:
$$
{E}^1_{1,in} = \dfrac{24\pi}{5} R_1 |\ln (h)| | R_1\nabla_{x,y} u_{\bot}^*(\mathbf 0) + \dfrac{(R-S)}{2(R+S)} \vv{u}_{\sslash\!}(\mathbf 0)|^2 + O(1) \|\mathbf u^*;H^3(B(\mathbf 0,L))\|^2\,.
$$
Finally, we obtain:
\begin{multline*}
{E}_{1,in} = \left( 2\pi R_1  |\vv{u}^*_{\sslash\!}(\mathbf 0)|^2 + \dfrac{24 \pi}{5} R_1  | R_1\nabla_{x,y} u_{\bot}^*(\mathbf 0) +\dfrac{(R-S)}{2(R+S)}\vv{u}^*_{\sslash\!}(\mathbf 0)|^2 \right) |\ln(h)| \\
  + O(1) \|\mathbf u^*;H^3(B(\mathbf 0,L))\|^2.
\end{multline*}

\subsubsection{Asymptotics of ${E}_{0,in}$}
We focus now on $E_{0,in}$ and drop index $0$ for simplicity. 
Let first recall that ${\vv{v}}_{in}$ is constructed as:
\begin{eqnarray*} 
{\vv{v}}_{in,\,\sslash\!}({x},{y},{z})  &=&  \dfrac{1}{2} ( z - (h+ {\gamma}_t))( z - {\gamma}_b) \nabla_{x,y\,} {p}_{in}\,, \\[4pt] 
{{v}}_{in,\bot}({x},{y},{z}) &=&  \dfrac 12 {\rm div}_{x,y}  \left[ \int^{h+\gamma_t}_{z} (s- ( h+ \gamma_t))(s-\gamma_b)\nabla_{x,y\,}  {p}_{in} \, {\rm d}s\right]   
\end{eqnarray*}
for $(x,y,z) \in \mathcal G_{L},$ where:
$$
{p}_{in}(x,y) =  \chi_L(x,y) {q}_{in} (x,y) \,, \quad \forall \, (x,y,z) \in \mathcal G_{L}\,.
$$
with $\chi_L$ a suitable truncation function and ${q}_{in}$ the unique solution to
$$
\begin{array}{rcll}
-\dfrac{1}{12} {\rm div}(\gamma^3 \nabla {q}_{in}) &=& {u}^*_{\bot}(\mathbf 0) & \text{ on $B(\mathbf 0,L)\,,$} \\
{q}_{in} &=& 0 \,, & \text{ on $\partial B(\mathbf 0,L)$}\,. 
\end{array}
$$
We recall further that, in this radial case, we may compute ${q}_{in}$ explicitly:
$$
{q}_{in}(r) =  {u}^*_{\bot}(\mathbf 0)  \int_{r}^{L} \dfrac{6s}{\gamma^3(s)}{\rm d}s \,, \quad \forall \, r \in (0,L)\,.
$$
so that
$$
\partial_r {q}_{in}(r) = - {u}^*_{\bot}(\mathbf 0)  \dfrac{6r}{\gamma^3(r)}  \,, \quad \forall \, r \in (0,L)\,.
$$
This entails that ${\vv{v}}_{in} = v_r \vv{e}_r + v_{z} \mathbf e_z$  with:
$$
\begin{array}{rcl}
v_r(x,y,z) &=& \dfrac{\partial_{r} q_{in}}{2} \left((z-(h+\gamma_t(r)))(z- \gamma_b(r)) \right)\,,  \\[4pt]
&=& - {u}^*_{\bot}(\mathbf 0) \dfrac{3r}{\gamma^3(r)} \left((z-(h+\gamma_t(r)))(z- \gamma_b(r))\right)\,,
\end{array}
\quad 
\forall \, (x,y,z) \in \mathcal G_{L/2}\,.
$$
We assume from now on that ${u}^*_{\bot}(\mathbf 0) = 1$ for simplicity. We note that we have
$$
|{u}^*_{\bot}(\mathbf 0) | \leq O(1) \|\mathbf{u}^* ; H^2(B(\mathbf 0,L))\|\,,  
$$
so that all $O(1)$ terms in the following computations will turn into $O(1) \|\mathbf{u}^* ; H^2(B(\mathbf 0,L))\|^2$ in the final result.

As in the computations for ${E}_{1,in}$, from {Proposition \ref{prop_unifreg}}, we obtain first that:
$$
{E}_{0,in} = \int_{\mathcal G_{L/2}} |\nabla {\vv{v}}_{in}|^2 + O(1),
$$
and we replace ${p}_{in}$ by ${q}_{in}$. Also, we already computed in the proof of {Proposition \ref{prop_v0}} that 
\begin{equation} \label{eq_firstboundv0}
\int_{\mathcal F} |\nabla_{x,y} {v}_{in,\bot}|^2 = O(1)\,.
\end{equation}
Consequently, we have:
$$
{E}_{0,in} = \int_{\mathcal G_{L/2}} \left( |\partial_r v_r|^{2} +
\left|\dfrac{v_r}{r} \right|^2 +  |\partial_{z} v_r|^2 + |\partial_{z} v_z|^2 \right) + O(1)  \,, 
$$
where, due to the incompressibility condition satisfied by ${\vv{v}}_{in}:$
\begin{eqnarray*}
\int_{\mathcal G_{L/2}} |\partial_{z} v_{z}|^2 &=& \int_{\mathcal G_{L/2}}\left|\dfrac{1}{r}\partial_{r}[rv_{r}]\right|^2  = \int_{\mathcal G_{L/2}} \left( |\partial_r v_{r}|^{2} + 
\left|\dfrac{v_{r}}{r} \right|^2 \right) + \int_{r=L}\int_{z=\gamma_b(L)}^{z=h+\gamma_t(L)} |v_{r}|^2 \\[2pt]
&=& \int_{\mathcal G_{L/2}} \left( |\partial_r v_{r}|^{2} + 
\left|\dfrac{v_r}{r} \right|^2  \right)  + O(1)\,
\end{eqnarray*}
as $v_{r}$ remains bounded independently of $h$ away from the origin. We get thus:
$$
{E}_{0,in}  = I_{z} + 2 I_r +O(1)\,, 
$$
with:
$$
I_z = \int_{\mathcal G_{L/2}}  |\partial_{z} v_r|^2 \,, \qquad 
I_r  = \int_{\mathcal G_{L/2}} \left( |\partial_r v_r|^{2} + 
\left|\dfrac{v_r}{r} \right|^2 \right)\,.
$$

\medskip

\paragraph{\em Computation of $I_z$} Replacing the integrand in $I_z$ with its values yields:
\begin{eqnarray*}
I_z &=&  \dfrac{1}{4}\int_{B(\mathbf 0,L/2)} \int^{h+\gamma_t(r)}_{\gamma_b(r)} |\partial_r{q}_{in}(r)|^2 \left[ (z-(h+\gamma_t(r))) + (z-\gamma_b(r))\right]^2 {\rm d}z {\rm d}x {\rm d}y \, \\
	&=& \dfrac{1}{4} \left[ \int_{B(\mathbf 0,L/2)} |\nabla {q}_{in}|^2 \gamma^3  {\rm d}x{\rm d}y \right] \int_0^{1} (2s-1)^{2} {\rm d}s\, \\
	&=& \dfrac{1}{12}  \left[ \int_{B(\mathbf 0,L/2)} |\nabla {q}_{in}|^2 \gamma^3  {\rm d}x{\rm d}y \right] \,.
\end{eqnarray*}
At this point, we apply {Proposition \ref{prop_casf0}} to ${q}_{in}$ yielding:
$$
 \left[ \int_{B(\mathbf 0,L/2)} |\nabla {q}_{in}|^2 \gamma^3  {\rm d}x{\rm d}y \right]  = 72 \pi \left[  \dfrac{R_1^2}{h} - 3 \dfrac{R_1^4}{R_3^3} |\ln(h)|\right] + O(1).
$$
This entails finally:
$$
I_z = \dfrac{6\pi R_1^2}{h} - 18\pi \dfrac{R_1^4}{R_3^3} |\ln(h)| + O(1)\,.
$$

\medskip

\paragraph{\em Computation of $I_r$}
We proceed with the computation of  $I_r = I_r^1 + I_r^2 $ where:
$$ 
I^1_r :=   \int_{\mathcal G_{L/2}} |\partial_{r} v_r|^2\,,    \qquad 
I^2_r :=  \int_{\mathcal G_{L/2}} \left|\dfrac{v_r}{r}\right|^2 \,.
$$
We first compute $I_r^2$ by replacing $v_r$ with its values:
\begin{eqnarray*}
I_r^2 &=& 18\pi  \int_0^{L/2} \int_{\gamma_b(r)}^{h+\gamma_t(r)}  \dfrac{[(z-(h+\gamma_t(r)))(z-\gamma_b(r))]^2}{\gamma(r)^6}\, r{\rm d}r {\rm d}z\\
	&=& 18 \pi \int_{0}^{L/2} \dfrac{r{\rm d}r}{\gamma(r)} \int_0^{1} [(s-1)s]^2{\rm d}s\,\\
	&=&  \dfrac{3\pi}{5} \int_{0}^{L/2} \dfrac{r{\rm d}r}{\gamma(r)} \,.
\end{eqnarray*}
We already computed (see the computation of ${E}^2_{1,in}$) that:
\begin{eqnarray*}
 \int_{0}^{L/2} \dfrac{r{\rm d}r}{\gamma(r)} &=& R_1 |\ln(h)| + O(1)\,.
\end{eqnarray*}
so that we obtain finally:
$$
I_r^2 = \dfrac{3\pi}{5} R_1 |\ln(h)| + O(1)\,.
$$

Second, we expand $\partial_r v_r.$ We have:
$$
\partial_r v_r = \dfrac 12 \left( \partial_{rr} {q}_{in}(r) (z-(h+\gamma_t(r)))(z-\gamma_b(r)) 
+ \partial_r {q}_{in}(r) \partial_r [(z-(h+\gamma_t(r))) (z-\gamma_b(r))] \right).
$$
Consequently: 
\begin{eqnarray*}
I_r^1 &=& \dfrac 14 \Bigl[ \int_{\mathcal G_{L/2}} |\partial_{rr} {q}_{in}(r) (z-(h+\gamma_t(r)))(z-\gamma_b(r))|^2 \\
&& + \int_{\mathcal G_{L/2}} |\partial_r{q}_{in}(r)  \partial_r [(z-(h+\gamma_t(r)))(z-\gamma_b(r))]|^2 \\
&&+ 2 \int_{\mathcal G_{L/2}} \partial_{rr} {q}_{in} \, \partial_{r} {q}_{in}\, \partial_{r} [(z-(h+\gamma_t(r)))(z-\gamma_b(r))] (z-(h+\gamma_t(r)))(z-\gamma_b(r)) \Big]\,\\
&=& I_{r,a} + I_{r,b} + I_{r,c}
\end{eqnarray*}
After tedious but straightforward computations, we get:
\begin{eqnarray*}
I_{r,a} &=&  15 \pi R_1 |\ln(h)| + O(1)\,,\\
I_{r,b} &=&  \left( 24\pi R_1 - \dfrac{24 \pi R_1^3}{RS}  	\right)|\ln(h)| + O(1)\,, \\
I_{r,c} &=& -30 \pi R_1 |\ln(h)| + O(1)\,.   
\end{eqnarray*}
and
\begin{eqnarray*}
I_r 
&=& \dfrac{24\pi }{5} \left[2 R_1   -  \dfrac{5 R_1^3}{RS} \right] |\ln(h)| + O(1)\,.
\end{eqnarray*}
Combining computations of $I_r$ and $I_z,$ we obtain: 
$$
E_{1,in} =  6 \pi \left[  \dfrac{R_1^2}{h} +  |\ln(h)| \left(   \dfrac{16}{5}R_1 - \dfrac{8 R_1^3}{RS} - 3 \dfrac{R_1^4}{R_3^3}\right)  \right] + O(1)\,.
$$

\subsection{Asymptotics of cross terms}
We proceed with the computation of the asymptotics of cross terms $E_{ij},$ $i\neq j \in \{0,1,R\}^2.$
We recall that with similar arguments as previously, there holds:
$$
E_{ij} = {E}_{ij,in} + O(1)\left\{ \|\mathbf{u}^*; H^3(B(\mathbf 0,L))\|^2 + \|\mathbf{u}^*;H^{\frac 12}(\partial \Omega)\|^2\right\}, 
$$
with obvious notations. 

\subsubsection{Asymptotics of ${E}_{01,in}$}
We first treat the term $E_{01}.$ For this term, we assume without restricting the generality
that $R \nabla {u}^*_{\bot} (\mathbf 0) + \vv{u}^*_{\sslash\,}(\mathbf 0)$ is parallel to $\vv{e}_1.$
 As a consequence, we obtain that ${q}_{1,in}$ reads $\varphi(r) \cos(\theta)$ in polar
coordinates so that, as a function of $(x,y)$ it satisfies:
$$
{q}_{1,in}(-x,-y) = -{q}_{1,in}(x,y) \quad \forall \, (x,y) \in B(\mathbf 0,L). 
$$
Hence, introducing the rotation matrix, 
$$
S_1 := 
\begin{pmatrix}
-1  & 0 & 0 \\
0 & -1 & 0 \\
0 &0 & 1 
\end{pmatrix}
$$
there holds:
$$
S_1{\vv{v}}_{1,in} (S_1 (x,y,z)) =  -{\vv{v}}_{1,in}(x,y,z)\,, \quad \forall \, (x,y,z) \in \mathcal G_{L}
$$
On the opposite, we have that $q_{0,in}$ satisfies $q_{0,in}(-x,-y) = q_{0,in}(x,y)$ so that
we have the symmetries:
$$
S_1 {\vv{v}}_{0,in} (S_1 (x,y,z)) =  {\vv{v}}_{0,in}(x,y,z)\,, \quad \forall \, (x,y,z) \in \mathcal G_{L}.
$$
Going to the derivatives, this entails that:
$$
{E}_{10,in} = \int_{\mathcal G_{L}} \nabla {\vv{v}}_{0,in}  : \nabla {\vv{v}}_{1,in} = - \int_{\mathcal G_{L}} \nabla {\vv{v}}_{0,in}  : \nabla {\vv{v}}_{1,in} = 0.
$$

\subsubsection{Asymptotics of ${E}_{0R,in}$}
{By definition, we have :
$$
E_{0R,in} =  \int_{\mathcal G_{L}} \nabla \vv{v}_{0,in} : \nabla \vv{v}_{R,in} 
$$
We introduce:
$$
\tilde{p}_{R,in} = p_{R,in} -  \int_{B(\mathbf 0,L)}  {p}_{R,in} (x,y)v_{0,\bot,in}(x,y,\gamma_b(x,y)){\rm d}x {\rm d}y,
$$
and, applying that $\vv{v}_{R,in}$ is divergence free, we transform:
\begin{eqnarray*}
E_{0R,in} &=&  \int_{\mathcal G_{L}} (\nabla \vv{v}_{R,in} - \tilde{p}_{R,in} \mathbb I_{3})  : \nabla \vv{v}_{0,in}  \\
&=&  \int_{\mathcal G_{L}} (\nabla_{x,y} \vv{v}_{R,in,\, \sslash} - \tilde{p}_{R,in} \mathbb I_{2}) : \nabla_{x,y} \vv{v}_{0,in,\,\sslash} +
\int_{\mathcal G_{L}} \partial_z \vv{v}_{R,in,\,\sslash}  \cdot \partial_z \vv{v}_{0,in,\,\sslash} \\
&& - \int_{\mathcal G_{L}}  \tilde{p}_{R,in} \partial_z {v}_{0,in,\bot}  +  \int_{\mathcal G_{L}} \partial_z {v}_{R,in,\bot}  \cdot \partial_z {v}_{0,in,\bot}
 +  \int_{\mathcal G_{L}} \nabla_{x,y} \vv{v}_{R,in,\,\bot}  \cdot \nabla_{x,y} \vv{v}_{0,in,\bot} 
\end{eqnarray*}
In this identity, we integrate by parts the integrals on the first line of the right-hand side yielding:
\begin{eqnarray*}
&& \int_{\mathcal G_{L}} (\nabla_{x,y} \vv{v}_{R,in,\, \sslash} - \tilde{p}_{R,in} \mathbb I_{2}) : \nabla_{x,y} \vv{v}_{0,in,\,\sslash} +
\int_{\mathcal G_{L}} \partial_z \vv{v}_{R,in,\,\sslash}  \cdot \partial_z \vv{v}_{0,in,\,\sslash}\\
 && =- \left(  \int_{\mathcal G_{L}} (\Delta_{x,y} \vv{v}_{R,in,\,\sslash} - \nabla_{x,y} \tilde{p}_{R,in})  \cdot \vv{v}_{0,in,\,\sslash} +
\int_{\mathcal G_{L}} \partial_{zz} \vv{v}_{R,in,\,\sslash} \cdot  \vv{v}_{0,in,\,\sslash} \right)
\\
&& =- \left(  \int_{\mathcal G_{L}} (\Delta_{x,y} \vv{v}_{R,in,\,\sslash} - \nabla_{x,y} {p}_{R,in} ) \cdot \vv{v}_{0,in,\,\sslash} +
\int_{\mathcal G_{L}} \partial_{zz} \vv{v}_{R,in,\,\sslash}  \cdot  \vv{v}_{0,in,\,\sslash} \right)
\\
&& = - \left( \int_{\mathcal G_{L}} \Delta_{x,y} \vv{v}_{R,in,\,\sslash} \cdot \vv{v}_{0,in,\,\sslash} \right)
\end{eqnarray*}
where we used that 
\begin{itemize}
\item $\vv{v}_{0,in,\,\sslash}$ vanishes on $\partial \mathcal G_{L}$
\item $\partial_{zz} \vv{v}_{R,in,\,\sslash}- \nabla_{x,y} p_{R,in} = 0$ in $\mathcal G_{L}$
\end{itemize}
We also note that $\tilde{p}_{R,in}$ does not depend on $z$ and apply boundary conditions for $v_{0,\bot,in}$ yielding:
\begin{eqnarray*}
\int_{\mathcal G_{L}}  \tilde{p}_{R,in} \partial_z {v}_{0,in,\bot}
&=& \int_{B(\mathbf 0,L)}  \tilde{p}_{R,in} v_{0,in,\bot}(x,y,\gamma_b(x,y)) \\
&=& 0 .
\end{eqnarray*}
because of our choice for $\tilde{p}_{R,in}.$ Finally, we have:
\begin{eqnarray*}
E_{0R,in} &=&    \int_{\mathcal G_{L}} \Delta_{x,y} \vv{v}_{R,in,\,\sslash} \cdot \vv{v}_{0,in,\,\sslash} 
+  \int_{\mathcal G_{L}} \partial_z {v}_{R,in,\bot}  \cdot \partial_z {v}_{0,in,\bot}\\
&& 
 +  \int_{\mathcal G_{L}} \nabla_{x,y} {v}_{R,in,\,\bot}  \cdot \nabla_{x,y} {v}_{0,in,\bot} 
 \end{eqnarray*}
 }
From \eqref{eq_firstboundvR} and \eqref{eq_firstboundv0}, we have that:
$$
\int_{\mathcal F} |\nabla_{x,y} {v}_{0,in,\bot}|^2 + \int_{\mathcal F} |\nabla {\vv{v}}_{in,R}|^2 = O(1)(\|\mathbf{u}^*; H^3(B(\mathbf 0,L))\|^2 ).
$$ 
Consequently, it remains:
\begin{eqnarray*}
{E}_{0R,in} &=& \int_{\mathcal G_{L}} \Delta_{x,y} \vv{v}_{R,in,\,\sslash} \cdot \vv{v}_{0,in,\,\sslash} 
+  \int_{\mathcal G_{L}} \partial_z {v}_{R,in,\bot}  \cdot \partial_z {v}_{0,in,\bot}\\
&& + O(1)\|\mathbf{u}^*; H^3(B(\mathbf 0,L))\|^2
\\
\end{eqnarray*}
Applying that ${\vv{v}}_{R,in}$ and ${\vv{v}}_{0,in}$ are incompressible, we get:
$$
 \int_{\mathcal G_{L}} \partial_z {v}_{R,in,\bot}   \partial_z {v}_{0,in,\bot} = -\int_{\mathcal F}  {\rm div}_{x,y} {\vv{v}}_{0,in,\,\sslash} {\rm div}_{x,y}{\vv{v}}_{R,in,\,\sslash} 
$$
As ${\vv{v}}_{0,\, \sslash,in}$ vanishes on $\partial \mathcal G_L,$ we integrate this identity by parts:
leading to:
\begin{eqnarray*}
\int_{\mathcal G_{L}} \Delta_{x,y} \vv{v}_{R,in,\,\sslash} \cdot \vv{v}_{0,in,\,\sslash} 
+  \int_{\mathcal G_{L}} \partial_z {v}_{R,in,\bot}  \cdot \partial_z {v}_{0,in,\bot}\\
= 
\int_{\mathcal G_{L}} ( \Delta_{x,y} \vv{v}_{R,in,\,\sslash} + \nabla_{x,y}{\rm div}_{x,y}\vv{v}_{R,in,\,\sslash} )  \cdot \vv{v}_{0,in,\,\sslash} 
\end{eqnarray*}
so that forgetting remainder terms for conciseness:
$$
|{E}_{0R,in}| \leq \int_{\mathcal F} |\nabla_{x,y}^2 {\vv{v}}_{R,in,\,\sslash}| | {\vv{v}}_{0,in,\,\sslash}|.
$$
At this point, we recall the explicit form of ${\vv{v}}_{R,\,\sslash,in}$ :
$$
{\vv{v}}_{R,in,\,\sslash} = (z-(h+\gamma_t))(z-\gamma_b) \nabla {p}_{R,in} + \left(\dfrac{ h+ \gamma_t  - z}{\gamma}\right) \chi_L \vv{u}^*_{R,\,\sslash}.
$$
Consequently, applying \eqref{eq_boundtbg} there holds:
\begin{multline*}
|\nabla_{x,y}^2 \vv{v}_{R,in,\,\sslash}| \leq  K[R,S] \Bigg( \gamma |\nabla {p}_{R,in}|  + \gamma^{\frac 32} |\nabla^2{p}_{R,in}| + \gamma^2 |\nabla^{3} {p}_{R,in}|  
\\+ |\nabla^2 \vv{u}^*_{R,\,\sslash}| + \dfrac{|\nabla \vv{u}^*_{R,\,\sslash}|}{\gamma^{\frac 12}} + \dfrac{|\vv{u}^*_{R,\,\sslash}|}{\gamma}\Bigg)
\end{multline*}
We recall we have also:
$$
|{\vv{v}}_{0,in,\,\sslash}| \leq |{u}_{\bot}^*(\mathbf 0)|\dfrac{r}{\gamma}.
$$
Finally, applying that $r \leq K[R,S] {\gamma^{\frac 12}},$ we get:
\begin{multline*}
{E}^1_{0R,in} \leq K[R,S] |{u}_{\bot}^*(\mathbf 0)| \int_{0}^L  \left( \gamma^{\frac 32} |\nabla {p}_{R,in}|  +  \gamma^{2} |\nabla^2 {p}_{R,in}| + \gamma^{\frac 52} |\nabla^{3} {p}_{R,in}| \right){r}{\rm d}r\\ +
\int_{0}^L  \left( |\nabla^2 \vv{u}^*_{R,\,\sslash}| + {|\nabla \vv{u}^*_{R,\,\sslash}|} + \dfrac{|\vv{u}^*_{R,\,\sslash}|}{\gamma^{\frac 12}} \right){r}{\rm d}r.
\end{multline*}
Hence:
\begin{multline*}
{E}^1_{0R,in} \leq K[R,S] |{u}_{\bot}^*(\mathbf 0)| \Biggl\{ \left( \int_{0}^L \left( \gamma^{3} |\nabla {p}_{R,in}|^2  +  \gamma^{4} |\nabla^2 {p}_{R,in}|^2 + \gamma^{5} |\nabla^{3} {p}_{R,in}|^2 \right) {r}{\rm d}r \right)^{\frac 12} \\
+ \| \vv{u}^*_{R,\,\sslash} ; H^2(B(\mathbf 0,L))\| \Biggr\},
\end{multline*}
so that, {introducing $q_{R,in} = q_{R,in}^{(1)} + h q^{(2)}_{R,in}$} and applying Propositions \ref{prop_favorable} and \ref{prop_reg}, we obtain
as for \eqref{eq_bornecasR}:
$$
{E}^1_{0R,in} = O(1)\| \vv{u}^* ; H^3(B(\mathbf 0,L))\|^2 \,.
$$

\subsubsection{Asymptotics of ${E}_{1R,in}$}
The computations of $E_{1R,in}$ follows the line of the preceding section. We first remark that \eqref{eq_firstboundvR} and \eqref{eq_firstboundv1} imply:
$$
\int_{\mathcal F} |\nabla_{x,y} {\vv{v}}_{1,in}|^2 +  |\nabla {v}_{1,in,\bot}|^2 + \int_{\mathcal F} |\nabla{\vv{v}}_{R,in}|^2 = O(1)\| \vv{u}^* ; H^3(B(\mathbf 0,L))\|^2
$$ 
so that:
$$
{E}_{1R,in} = \int_{\mathcal F} \partial_z \vv{v}_{1,in,\,\sslash} \partial_z \vv{v}_{R,in,\,\sslash} + O(1)\| \vv{u}^* ; H^3(B(\mathbf 0,L))\|^2.
$$
Explicit formulas yield that, for $i=1,R,$ we have:
$$
\partial_{z} \vv{v}_{i,in,\,\sslash} =  \nabla {p}_{i,in}( (z-(h+\gamma_t)) + (z-\gamma_b)) +  \dfrac{\vv{u}_{i,\,\sslash}^{*}}{\gamma}\,.
$$
Hence, integrating at first w.r.t. $z$ and deleting vanishing terms, we obtain:
\begin{eqnarray*}
\left| \int_{\mathcal F} \partial_z \vv{v}_{1,in,\, \sslash} \partial_z \vv{v}_{R,in,\, \sslash} \right|
&\leq&   K[R,S] \left( \int_{B(\mathbf 0,L)} |\nabla p_{R,in}| |\nabla p_{1,in}| \gamma^3 +  \int_{B(\mathbf 0,L)} \dfrac{|\vv{u}_{R,\, \sslash}^*| |\vv{u}_{1,\, \sslash}^*|}{\gamma}\right)\\
&\leq & K[R,S] \Bigl( \int_{B(\mathbf 0,L)}  |\nabla p_{1,in}|^2 \gamma^{\frac 72} + |\nabla p_{R,in}|^2  \gamma^{\frac 52}  \\
&& \quad + \int_{B(\mathbf 0,L)} \dfrac{|\vv{u}_{R,\, \sslash}^*|^2}{\gamma^{\frac 32}} + \dfrac{ |\vv{u}_{1,\, \sslash}^*|^2}{\gamma^{\frac 12}}\,. 
\end{eqnarray*}
As we already computed several times, the vanishing properties of $\vv{u}*_{R}$ in the origin imply that
$$
\int_{B(\mathbf 0,L)} \dfrac{|\vv{u}_{R,\, \sslash}^*|^2}{\gamma^{\frac 32}} + \dfrac{ |\vv{u}_{1,\, \sslash}^*|^2}{\gamma^{\frac 12}} \leq O(1) \|\vv{u} ; H^2(B(\mathbf 0,L))\|^2
$$
and, applying Proposition \ref{prop_favorable} and \ref{prop_reg2}, we get: 
$$
\int_{B(\mathbf 0,L)}  |\nabla q_{1,in}|^2 \gamma^{\frac 72} + |\nabla q_{R,in}|^2  \gamma^{\frac 52}  \leq K[R,S]  \|\vv{u}^* ; H^3(B(\mathbf 0,L))\|^2
$$
Hence, we have finally: ${E}_{1R,in}  = O(1)\| \vv{u} ; H^3(B(\mathbf 0,L))\|^2.$ 

\paragraph{Acknowledgement}
This paper was written while the first author was visiting the REO Team at INRIA Rocquencourt
that he thanks for its hospitality. The first author is supported by the project DYFICOLTI
ANR-13-BS01-0003-01.

\appendix

\section{Proof of Lemma \ref{lem_q}}
This appendix is devoted to the proof of 
\begin{lemma} \label{lem_q_app}
Given $R >0,$ there exists a $q \in C^{\infty}((0,\infty)) \cap C([0,\infty))$ solution to 
\begin{eqnarray} \label{eq_qs}
\partial_{ss} q +  \left(\dfrac{1}{s} +  \dfrac{3s}{R  (1+\frac{s^2}{2R})} \right) \partial_s q - \dfrac{q}{s^2} = 
-\dfrac{12 s}{(1+\frac {s^2}{2R})^3} \,, & & s \in (0,\infty) \,, \\
q(0) = 0 \quad \lim_{s \to \infty} q(s) = 0\,.  &&\,.  \label{eq_qs2}
\end{eqnarray}
Furthermore we have the asymptotic description:
\begin{equation}
q(s) = \dfrac{48 R^3}{5 s^3} + O(\frac 1 {s^4})\quad
\partial_sq(s) =  - \dfrac{144 R^3}{5 s^4} + O(\frac 1 {s^5}) \,.
\end{equation}
\end{lemma}
\begin{proof}
From now on, we let $R >0$ as in the statement of our lemma. As the proof seems standard, we only sketch the 
main steps.

\medskip

First, we fix $L >0$ and solve:
\begin{eqnarray} \label{eq_qsp}
-\dfrac 1{12} \left( \partial_{ss} q +  \left(\dfrac{1}{s} +  \dfrac{3s}{R  (1+\frac{s^2}{2R})} \right) \partial_s q - \dfrac{q}{s^2} \right) = 
\dfrac{s}{(1+\frac {s^2}{2R})^3} \,, & & s \in (0,L) \,, \\
q(0) = 0, \quad q(L) = 0\,.  && \label{eq_qsp2}
\end{eqnarray}
To this end, we introduce $\gamma_1(x,y) := (1+(x^2+y^2)/(2R))$ and introduce the bilinear form:
$$
(( \varpi, \varphi ))_{1} = \dfrac 1{12} \int_{B(\mathbf 0,L)} |\gamma_1|^3 \nabla \varpi : \nabla \varphi
$$
on $H^1_0(B(\mathbf 0,L)).$ As $\gamma_1$ is smooth and does not vanish on $B(\mathbf 0,L)$ we obtain, by applying the Stampacchia theorem, existence of a unique $\varpi_L \in H^1_0(B(\mathbf 0,L))$ solution to 
$$
(( \varpi_L, \varphi ))_{1}  = \int_{B(\mathbf 0,L)} x\varphi(x,y) {\rm d}x {\rm d}y\,, \quad \forall \, \varphi \in H^1_0(B(\mathbf 0,L))\,.
$$ 
Due to the invariance by rotation of $\gamma_1$ and the computation domain $B(\mathbf 0,L)$ for this weak formulation, we have that the unique solution reads in polar coordinates: 
$$
\varpi_L(s,\theta) = q_L(s)\cos(\theta) \quad \forall \, (s,\theta) \in (0,L) \times (-\pi,\pi)\,,
$$ 
with $q_L \in C^{\infty}((0,L)) \cap C([0,L])$ solution to \eqref{eq_qsp}-\eqref{eq_qsp2}.  

\medskip

Second, we remark that we have a maximum principle for \eqref{eq_qsp} and that
$$
\dfrac{s}{(1+\frac {s^2}{2R})^3} \leq \dfrac{(2R)^{\frac 32}}{s^2} \quad \forall s>0\,,
$$
this yields in particular that
$$
0 \leq q_L(s) \leq  12(2R)^{\frac 32}\,, \quad \forall \, s \in (0,L)\,, \quad \forall \, L>0\,.
$$ 
Furthermore, setting $\bar{q}(s) = 48 R^3/(5 s^3)$ there holds: 
$$
-\dfrac 1{12} \left( \partial_{ss} \bar{q} +  \left(\dfrac{1}{s} +  \dfrac{3s}{R  (1+\frac{s^2}{2R})} \right) \partial_s \bar{q} - \dfrac{\bar{q}}{s^2} \right) = 
\dfrac{s}{(1+\frac {s^2}{2R})^3} + \varepsilon(s)
$$
where
$$
|\varepsilon(s)| \leq \dfrac{C_0}{s^6} \quad \forall \, s \in (1,\infty),
$$
Setting also $\hat{q}(s) = 1/s^4$, there exists a constant $c >0$ such that: 
$$
-\dfrac 1{12} \left( \partial_{ss} \hat{q} +  \left(\dfrac{1}{s} +  \dfrac{3s}{R  (1+\frac{s^2}{2R})} \right) \partial_s \hat{q} - \dfrac{\hat{q}}{s^2} \right) \geq \dfrac{c}{s^6}. \quad \forall \, s >>1.
$$
Finally, by application of the maximum principle, we obtain that there exists a constant $K$ independent of $L$
for which:
$$
q_L(s) =  \dfrac{48R^3}{5 s^3} + q_{rem}(s) \quad \text{ with } |q_{rem}(s)| \leq \dfrac{K}{s^4}\,. 
$$
As the previous estimates are independent of $L$ we can pass to the limit in $L\to \infty$ and obtain in the limit
a solution to \eqref{eq_qs}-\eqref{eq_qs2} with the expected zero-order asymptotic expansion. We do not detail further this passage to the limit. As $\gamma_1$ remains strictly positive on $(0,\infty)$ we may apply classical ellipticity results to yield that the solution-limit is indeed smooth on $(0,\infty)$ and continuous in $0.$

\medskip

Finally, we remark that $q' = \partial_s q$ satisfies:
$$
\dfrac{1}{s} \partial_s (s(1+s^2/(2R))^3 q'(s)) =   \dfrac{q(s)(1+s^2/(2R))^3}{s^2}  - {12s}.
$$
We obtain the expected asymptotic behavior of $q'$ by integrating this equation between $s_1=1$
and $s_2=s>1$.
\end{proof}


\begin{thebibliography}{100}

\bibitem{CardoneNazarovSokolowski09}
G.~Cardone, S.~A. Nazarov, and J.~Sokolowski.
\newblock Asymptotics of solutions of the {N}eumann problem in a domain with
  closely posed components of the boundary.
\newblock {\em Asymptot. Anal.}, 62(1-2):41--88, 2009.

\bibitem{Chambayada86}
G. Bayada and M. Chambat.
\newblock The transition between the {S}tokes equations and the {R}eynolds
  equation: a mathematical proof.
\newblock {\em Appl. Math. Optim.}, 14(1):73--93, 1986.

\bibitem{Cimatti83}
G. Cimatti.
\newblock How the {R}eynolds equation is related to the {S}tokes equations.
\newblock {\em Appl. Math. Optim.}, 10(3):267--274, 1983.

\bibitem{Cooley&ONeill68}
M.~D.~A. Cooley and M.~E. O'Neill.
\newblock On the slow rotation of a sphere about a diameter parallel to a
  nearby plane wall.
\newblock {\em J. Inst. Math. Applics}, 4:163--173, 1968.

\bibitem{Cooley&ONeill69}
M.~D.~A. Cooley and M.~E. O'Neill.
\newblock On the slow motion generated in a viscous fluid by the approach of a
  sphere to a plane wall or stationary sphere.
\newblock {\em Mathematika}, 16:37--49, 1969.

\bibitem{Cox74}
R.~G. Cox.
\newblock The motion of suspended particles almost in contact.
\newblock {\em Int. J. Multiphase Flow}, 1:343--371, 1974.

\bibitem{Dean&Oneill64}
W.~R. Dean and M.~E. O'Neill.
\newblock A slow motion of viscous liquid caused by the rotation of a solid
  sphere.
\newblock {\em Mathematika}, 10:13--24, 1963.

\bibitem{Galdibooknew}
G.~P. Galdi.
\newblock {\em An introduction to the mathematical theory of the
  {N}avier-{S}tokes equations}.
\newblock Springer Monographs in Mathematics. Springer, New York, second
  edition, 2011.
\newblock Steady-state problems.

\bibitem{DGVH10}
D.~G{\'e}rard-Varet and M.~Hillairet.
\newblock Regularity issues in the problem of fluid structure interaction.
\newblock {\em Arch. Ration. Mech. Anal.}, 195(2):375--407, 2010.

\bibitem{GilbargTrudingerbook}
D.~Gilbarg and N.~S. Trudinger.
\newblock {\em Elliptic partial differential equations of second order}.
\newblock Classics in Mathematics. Springer-Verlag, Berlin, 2001.

\bibitem{Happel&Brenner65}
J.~Happel and H.~Brenner.
\newblock {\em Low {R}eynolds number hydrodynamics with special applications to
  particulate media}.
\newblock Prentice-Hall Inc., Englewood Cliffs, N.J., 1965.

\bibitem{Hillairet07}
M.~Hillairet.
\newblock Lack of collision between solid bodies in a 2{D} incompressible
  viscous flow.
\newblock {\em Comm. Partial Differential Equations}, 32(7-9):1345--1371, 2007.

\bibitem{HillairetTakahashi09}
M.~Hillairet and T.~Takahashi.
\newblock Collisions in three-dimensional fluid structure interaction problems.
\newblock {\em SIAM J. Math. Anal.}, 40(6):2451--2477, 2009.

\bibitem{HillairetTakahashi10}
M.~Hillairet and T.~Takahashi.
\newblock Blow up and grazing collision in viscous fluid solid interaction
  systems.
\newblock {\em Ann. Inst. H. Poincar\'e Anal. Non Lin\'eaire}, 27(1):291--313,
  2010.


\bibitem{HouotMunnier08}
J.~ G. Houot and A. Munnier.
\newblock On the motion and collisions of rigid bodies in an ideal fluid.
\newblock {\em Asymptot. Anal.}, 56(3-4):125--158, 2008.

\bibitem{MunnierRamdanipp}
A. Munnier and K.~Ramdani.
\newblock Asymptotic analysis of a Neumann problem in a domain with cusp.
Application to the collision problem of rigid bodies in a perfect fluid.
\newblock arXiv:1405.5446v1.


\bibitem{Nazarov95}
S.~A. Nazarov.
\newblock On the water flow under a lying stone.
\newblock {\em Mat. Sb.}, 186(11):75--110, 1995.


\bibitem{ONeill64}
M.~E. O'Neill.
\newblock A slow motion of viscous liquid caused by a slowly moving solid
  sphere.
\newblock {\em Mathematika}, 11:67--74, 1964.

\bibitem{ONeill&Stewartson67}
M.~E. O'Neill and K.~Stewartson.
\newblock On the slow motion of a sphere parallel to a nearby plane wall.
\newblock {\em J. Fluid Mech.}, 27:705--724, 1967.

\bibitem{Starovoitov04}
V.~N. Starovoitov.
\newblock Behavior of a rigid body in an incompressible viscous fluid near a
  boundary.
\newblock In {\em Free boundary problems (Trento, 2002)}, volume 147 of {\em
  Internat. Ser. Numer. Math.}, pages 313--327. Birkh\"auser, Basel, 2004.

\end{thebibliography}
\end{document}